\documentclass[11pt]{amsart}
\usepackage{amsmath,amsfonts,amsthm}
\usepackage [dvips]{graphicx}
\usepackage{amsmath}
\usepackage{graphicx,psfrag}  
\usepackage[psamsfonts]{amssymb}
\usepackage{amscd}
\usepackage{floatflt}
\usepackage[all,cmtip]{xy}
\title[The cohomological equation]{The cohomological equation for partially hyperbolic diffeomorphisms 
}
\author{Amie Wilkinson}
\address{Department of Mathematics, Northwestern University, 2033 Sheridan Rd., Evanston, IL 60208.  USA}
\date{\today}
\theoremstyle{plain}
\newtheorem{main}{Theorem}

\newtheorem{theorem}{Theorem}[section]
\newtheorem{proposition}[theorem]{Proposition}
\newtheorem{lemma}[theorem]{Lemma}
\newtheorem{corollary}[theorem]{Corollary}
\newtheorem{conjecture}[theorem]{Conjecture}

\newtheorem{dfinition}[theorem]{Definition}

\def\eps{\varepsilon}

\def\Diff{\hbox{Diff} }
\def\mod{\hbox{mod} }

\def\title{\em}

\def\hH{\widehat H}

\def\bar{\overline}

\def\Bbb{\bf}

\def\cW{\mathcal{W}}
\def\cK{\mathcal{K}}
\def\cB{\mathcal{B}}
\def\cF{\mathcal{F}}

\def\cJ{\mathcal{J}}
\def\hW{{\widehat{\mathcal{W}}}}
\def\U{\mathcal{U}}
\def\cG{\mathcal{G}}
\def\cS{\mathcal{S}}
\def\G{\mathcal{G}}

\def\cC{\mathcal{C}}
\def\cO{\mathcal{O}}

\def\cT{\mathcal{T}}

\def\F{\mathcal{F}}

\def\cH{\mathcal{H}}

\def\cP{\mathcal{P}}
\def\hcH{\widehat\cH}

\def\hh{\hat h}

\def\hE{\widehat E}

\def\proj{\hbox{proj}}
\def\transverse{\,\raise2pt\hbox to1em{\hfil$\top$\hfil}\hskip -1em \hbox
to1em{\hfil$\cap$\hfil}\,} 

\newcommand\RR{{\mathbb R}}

\newcommand\ZZ{{\mathbb Z}}

\newlength{\figboxwidth} 

\def\W{\mathcal{W}}

\begin{document}
\input{epsf.tex}

\maketitle

\tableofcontents

\newpage

\section*{Introduction}

Let $f:M\to M$ be a dynamical system and let $\phi\colon M\to \RR$ be a function.
Considerable energy has been devoted to describing the set of solutions to the {\em cohomological equation:}
\begin{eqnarray}\label{e=livsic}
\phi = \Phi\circ f - \Phi,
\end{eqnarray}
under varying hypotheses on the dynamics of $f$ and the regularity of $\phi$.
When a solution $\Phi\colon M\to \RR$ to this equation
exists, then $\phi$ is a called {\em coboundary}, for in the appropriate
cohomology theory we have $\phi=d\Phi$.
For historical reasons, a solution $\Phi$ to (\ref{e=livsic}) is called a {\em transfer function}.
The study of the cohomological equation has seen application in 
a variety of problems, among them: smoothness of invariant measures and conjugacies;
mixing properties of suspended flows; rigidity of group actions;
and geometric rigidity questions such as the isospectral problem.
This paper studies solutions to the cohomological equation when $f$ is a partially hyperbolic
diffeomorphism and $\phi$ is $C^r$, for some real number $r>0$.  

A a {\em partially hyperbolic diffeomorphism} $f\colon M\to M$ of a compact
manifold $M$ is one for which there exists a nontrivial, $Tf$-invariant splitting of
the tangent bundle $TM=E^s\oplus E^c\oplus E^u$ 
and a Riemannian metric on $M$ such that vectors in $E^s$ are uniformly contracted
by $Tf$ in this metric, vectors in $E^u$ are uniformly expanded, and
the expansion and contraction rates of vectors in $E^c$ is dominated by
the corresponding rates in $E^u$ and $E^s$, respectively.   An {\em Anosov diffeomorphism}
is one for which the bundle $E^c$ is trivial.

In the case where $f$ is an Anosov diffeomorphism, there is a wealth of
classical results on this subject, going back to the seminal work of
Livsi\u c, which we summarize here
in Theorem~\ref{t=livsicold}.  Here and
in the rest of the paper, the notation $C^{k,\alpha}$, for $k\in \ZZ_+$, $\alpha\in(0,1]$, means 
$C^k$, with $\alpha$-H{\"o}lder continuous $k$th derivative (where $C^{0,\alpha}$, $\alpha\in(0,1]$ simply means
$\alpha$-H{\"o}lder continuous). For $\alpha\in (0,1)$, $C^\alpha$ means  $\alpha$-H{\"o}lder continuous. More
generally, if $r>0$ is not an integer, then we will also write $C^r$
for $C^{\lfloor r \rfloor, r- \lfloor r \rfloor}$.

\begin{theorem}\label{t=livsicold}\cite{L1,L2,LS,GK1,GK2,LMM,J,dlL1}
Let $f\colon M\to M$ be an Anosov diffeomorphism and let $\phi:M\to \RR$ be H{\"o}lder continuous.

\smallskip

{\bf I. Existence of solutions.} If $f$ is $C^1$ and transitive, 
then (\ref{e=livsic}) has a continuous solution $\Phi$ if and only if
$\sum_{x\in\cO} \phi(x) = 0$,
for every $f$-periodic orbit $\cO$.  

\smallskip

{\bf II.  H{\"o}lder regularity of solutions.} If $f$ is $C^1$, then every continuous solution to (\ref{e=livsic}) is H{\"o}lder continuous.

\smallskip

{\bf III. Measurable rigidity.} Let $f$ be $C^2$ and volume-preserving.  
If there exists a measurable solution $\Phi$ to (\ref{e=livsic}),
then there is a continuous solution $\Psi$, with $\Psi = \Phi$ a.e.

More generally, if $f$ is $C^r$ and topologically transitive, for $r>1$, and  $\mu$ is a 
Gibbs state for $f$ with H{\"o}lder potential, then the same result holds:
if there exists a measurable function $\Phi$ such that (\ref{e=livsic})
holds $\mu$-a.e., then there is a continuous solution $\Psi$, with $\Psi = \Phi$,
$\mu$-a.e.

\smallskip

{\bf IV. Higher regularity of solutions.} Suppose that $r > 1$ is not an integer, and suppose that $f$ and $\phi$ are $C^{r}$. Then every continuous solution to (\ref{e=livsic}) is $C^{r}$.

If $f$ and $\phi$ are $C^1$, then every continuous solution to (\ref{e=livsic}) is $C^1$.

If $f$ and $\phi$ are real analytic, then every continuous solution to (\ref{e=livsic}) is real analytic.
\end{theorem}

There are several serious obstacles to overcome in generalizing these results
to partially hyperbolic systems.  For one, while a transitive Anosov
diffeomorphism has a dense set of periodic orbits, a transitive partially hyperbolic
diffeomorphism might have {\em no} periodic orbits (for an example, one can
take the time-$t$ map of a transitive Anosov flow, for an appropriate choice
of $t$).  Hence the hypothesis
appearing in part I can be empty: the vanishing of $\sum_{x\in\cO} \phi(x)$
for every periodic orbit of $f$ cannot be a complete invariant for solving
(\ref{e=livsic}).

This first obstacle was addressed by Katok and Kononenko \cite{KK},
who defined a new obstruction to solving equation (\ref{e=livsic}) when $f$ 
is partially hyperbolic. To define this obstruction, we
first define a relevant collection of paths in $M$, called
$su$-paths, determined by a partially hyperbolic structure.

The stable and unstable bundles $E^s$ and $E^u$ of
a partially hyperbolic diffeomorphism are tangent to foliations, 
which we denote by $\W^s$ and $\W^u$ respectively \cite{BP1}. 
The leaves of $\W^s$ and $\W^u$ are contractible, since
they are increasing unions of submanifolds diffeomorphic
to Euclidean space.
An {\em $su$-path} in $M$ is a concatenation of finitely many subpaths, 
each of which lies entirely in a single leaf of $\W^s$ or a single leaf of $\W^u$.
An {\em $su$-loop} is an $su$-path beginning and ending at the same point.

We say that a partially hyperbolic diffeomorphism $f:M\to M$ is {\em accessible} if
any point in $M$ can be reached from any other along an $su$-path.
The {\em accessibility class of $x\in M$} is the set of all $y\in M$ that can be reached 
from $x$ along an $su$-path.  Accessibility means that there is one accessibility class,
which contains all points. Accessibility is a key hypothesis in most of the results that follow.  We
remark that Anosov diffeomorphisms are easily seen to be accessible, by the transversality
of $E^u$ and $E^s$ and the connectedness of $M$.

Any finite tuple of points $(x_0,x_1,\ldots,x_k)$
in $M$ with the property that $x_i$ and $x_{i+1}$ lie in the same leaf of either $\cW^s$ or $\cW^u$,
for $i=0,\ldots,k-1$,
determines an $su$-path from $x_0$ to $x_k$; if in addition $x_k=x_0$, then the sequence determines an $su$-loop.
Following \cite{ASV}, we call such a tuple  $(x_0,x_1,\ldots,x_k)$ an
{\em accessible sequence} and if $x_0=x_k$, an {\em accessible cycle} (the
term {\em periodic cycle} is used in \cite{KK}).

For $f$ a partially hyperbolic diffeomorphism,
there is a naturally-defined {{\em periodic cycles functional}} 
$$PCF\colon \{\hbox{accessible sequences}\}\times C^\alpha(M)  \to \RR.$$
which was introduced in \cite{KK} as an obstruction to solving (\ref{e=livsic}).
For $x\in M$ and $x'\in \cW^u(x)$, we define:
  $$PCF_{(x,x')} \phi = \sum_{i=1}^\infty \phi(f^{-i}(x)) - \phi(f^{-i}(x')),$$
and for $x'\in \cW^s(x)$, we define:
  $$PCF_{(x,x')} \phi = \sum_{i=0}^\infty \phi(f^{i}(x')) - \phi(f^{i}(x)).$$
The convergence of these series follows from the H{\"o}lder continuity of 
$\phi$ and the expansion/contraction properties of the bundles $E^u$ and $E^s$.
This definition then extends to accessible sequences by setting 
$PCF_{(x_0,\ldots,x_k)}\phi =\sum_{i=0}^{k-1} PCF_{(x_i,x_{i+1})}(\phi)$.

Assuming a  hypothesis on $f$ called {\em local accessibility}\footnote{A partially hyperbolic diffeomorphism $f\colon M\to M$ is {\em locally accessible} if for every
compact subset $M_1\subset M$ there exists $k\geq 1$ such that for any $\eps>0$,
there exists $\delta>0$ that for every $x,x'\in M$ with $x\in M_1$ and 
$d(x,x') < \delta$, there is an accessible sequence $(x=x_0,\ldots, x_k=x')$ from
$x$ to $x'$ satisfying
$$d(x_i,x) \leq \eps,\quad\hbox{and}\quad d_{\cW^\ast}(x_{i+1},x_i)<2\eps,\quad\hbox{for}\quad i=0,\ldots,k-1$$
where $d_{\cW^\ast}$ denotes the distance along the $\cW^s$ or $\cW^u$ leaf common to the two points.}, \cite{KK} 
proved that the closely related {\em relative cohomological equation:}
\begin{eqnarray}\label{e=livsic2}
\phi = \Phi\circ f - \Phi + c,
\end{eqnarray}
has a solution $\Phi\colon M\to \RR$ and $c\in \RR$, with $\Phi$ continuous,
if and only if $PCF_\gamma(\phi)=0$, for every accessible cycle $\gamma$.

The local accessibility hypothesis in \cite{KK} has been verified only for very special classes of
partially hyperbolic systems, and it is not known whether there exist $C^1$-open sets of
locally accessible diffeomorphisms, or more generally, whether accessibility implies local accessibility
(although this seems unlikely).
Assuming the strong hypothesis that $E^u$ and $E^s$ are $C^\infty$ bundles,
\cite{KK} also showed that a continuous transfer function for a $C^\infty$ coboundary is always $C^\infty$.

In the first part of our main result, Theorem~\ref{t=main}, part I below, we show that
the local accessibility hypothesis in \cite{KK} can be replaced simply by accessibility.  
Accessibility is known to hold for a $C^1$ open and dense 
subset of all partially hyperbolic systems \cite{DW}, is $C^r$ open and 
dense among partially hyperbolic  systems with $1$-dimensional center \cite{HHU, BHHTU},
and is conjectured to hold for a $C^r$ open and dense subset of all partially hyperbolic
diffeomorphisms, for all $r\geq 1$ \cite{PS}.
Thus, part I of Theorem~\ref{t=main} gives a robust counterpart of part I of Theorem~\ref{t=livsicold}
for partially hyperbolic diffeomorphisms. 

Another of the aforementioned major obstacles to generalizing Theorem~\ref{t=livsicold}
to the partially hyperbolic setting is that the regularity results in part IV
fail to hold for general partially hyperbolic systems.  Veech \cite{veech} and Dolgopyat \cite{D}
both exhibited examples of partially hyperbolic diffeomorphisms (volume-preserving and ergodic)
where there is a sharp drop in regularity from $\phi$ to a solution $\Phi$.  These examples
are not accessible.  Here we show in Theorem~\ref{t=main}, part IV, 
that assuming accessibility and a $C^1$-open 
property called {\em strong $r$-bunching} (which incidentally is satisfied by the 
nonaccessible examples in \cite{veech,D}), there is no significant loss of regularity between $\phi$ and $\Phi$.  

Part III of Theorem~\ref{t=livsicold} is the most resistant to generalization, primarily because
a general notion of Gibbs state for a partially hyperbolic diffeomorphism remains
poorly understood.  
In the conservative setting, the most general
result to date concerning ergodicity of for partially hyperbolic diffeomorphisms is
due to Burns and Wilkinson \cite{BWannals}, who show that every $C^2$, volume-preserving
partially hyperbolic diffeomorphism that is center-bunched and accessible is ergodic.
Center bunching is a $C^1$-open property that roughly requires that the action
of $Tf$ on $E^c$ be close to conformal, relative to the expansion and contraction
rates in $E^s$ and $E^u$ (see Section~\ref{s=defs}).
Adopting the same hypotheses as  in \cite{BWannals}, we recover here the analogue
of Theorem~\ref{t=livsicold} part III for volume-preserving partially hyperbolic diffeomorphisms.  

We now state our main result.

\begin{main}\label{t=main} Let $f\colon M\to M$ be partially hyperbolic and accessible, and let $\phi:M\to \RR$ be H{\"o}lder continuous.

\smallskip

{\bf I. Existence of solutions.} If $f$ is $C^1$, 
then (\ref{e=livsic2}) has a continuous solution $\Phi$ for some $c\in\RR$ if and only if
$PCF_{\cC} (\phi) = 0$,
for every accessible cycle $\cC$.

\smallskip

{\bf II.  H{\"o}lder regularity of solutions.} If $f$ is $C^1$, then every continuous solution to (\ref{e=livsic2}) is H{\"o}lder continuous.

\smallskip

{\bf III. Measurable rigidity.} Let $f$ be $C^2$,  center bunched, and volume-preserving.  
If there exists a measurable solution $\Phi$ to (\ref{e=livsic2}),
then there is a continuous solution $\Psi$, with $\Psi = \Phi$ a.e.

\smallskip

{\bf IV.  Higher regularity of solutions.} Let $k\geq 2$ be an integer. Suppose that
$f$ and $\phi$ are both $C^{k}$ and that $f$ is strongly $r$-bunched, for some $r < k-1$ 
or $r=1$.
If $\Phi$ is a continuous solution to (\ref{e=livsic2}), then $\Phi$ is $C^{r}$.
\end{main}

The center bunching and strong $r$-bunching hypotheses in parts III 
and IV are $C^1$-open conditions and are defined in Section~\ref{s=defs}.
Theorem~\ref{t=main} part IV generalizes all known $C^\infty$ Liv\v sic regularity results for accessible partially hyperbolic diffeomorphisms.  In particular, it applies to all time-$t$ maps of Anosov flows and 
compact group extensions of Anosov diffeomorphisms. Accessibility is a $C^1$ open and
$C^\infty$ dense condition in these classes \cite{BWskew, BPW}.  In dimension $3$, for example, 
the time-1 map of any mixing Anosov flow is stably accessible \cite{BPW}, unless the flow is a
constant-time suspension of an Anosov diffeomorphism.

We also recover the results of \cite{D} 
in the context of compact group extensions of volume-preserving Anosov diffeomorphisms.  
Finally, Theorem~\ref{t=main} also applies to all accessible, partially hyperbolic affine transformations 
of homogeneous manifolds.  A direct corollary that encompasses these cases is:

\begin{corollary}\label{c=main} Let $f$ be $C^\infty$, partially hyperbolic and accessible.  Assume that $Tf\vert_{E^c}$ is isometric in some continuous Riemannian metric. Let $\phi  \colon M\to \RR$ be $C^{\infty}$.
Suppose there exists a continuous function $\Phi \colon M\to \RR$
such that
$$\phi = \Phi\circ f - \Phi.$$
Then $\Phi$ is $C^{\infty}$.  
If, in addition, $f$ preserves volume, then any measurable solution $\Phi$ extends to 
a $C^\infty$ solution.

For any such $f$, and any integer $k\geq 2$, there is a $C^1$ open neighborhood $\U$ of $f$ in $\Diff^{k}(M)$ such that, for any accessible $g\in \U$, and any $C^{k}$ function $\phi\colon M\to \RR$, if
$$\phi = \Phi\circ g - \Phi,$$
has a continuous solution $\Phi$, then $\Phi$ is $C^1$ and also $C^{r}$, for all $r<k-1$. If $g$ also preserves volume, then any measurable solution extends to 
a $C^r$ solution.
\end{corollary}

The vanishing of the periodic cycles obstruction in Theorem~\ref{t=main}, part I turns out to
be a practical method in many contexts for determining whether (\ref{e=livsic2}) has a solution.
On the one hand, this method has already been used  by Damjanovi\'c and Katok
to establish rigidity of certain partially hyperbolic abelian group actions \cite{DK2}; in this
(locally accessible, algebraic) context, checking that the $PCF$ obstruction vanishes reduces
to questions in classical algebraic $K$-theory (see also \cite{DK1, NitK}).  On the other
hand, for a given accessible partially hyperbolic system, the $PCF$ obstruction provides an infinite codimension
obstruction to solving (\ref{e=livsic2}), and so
the generic cocycle $\phi$ has no solutions
to (\ref{e=livsic2}).  This latter fact follows from recent work of Avila, Santamaria and Viana 
on the related question of vanishing of Lyapunov exponents for linear cocycles over partially hyperbolic
systems (see \cite{ASV}, section 9).

As part of proof of Theorem~\ref{t=main}, part II, we also prove that stable and unstable foliations of any $C^1$ partially hyperbolic diffeomorphism are transversely H{\"o}lder continuous (Corollary~\ref{c=Holder}). 
This extends to the $C^1$ setting the well-known fact that the stable and 
unstable foliations for a $C^{1+\theta}$ partially hyperbolic diffeomorphism
are transversely H{\"o}lder continuous \cite{PSW}.  As far as we know, no previous regularity results were known for $C^1$ systems,
including Anosov diffeomorphisms.
  
In a forthcoming work \cite{AVW} we will use some of the results here to prove 
rigidity theorems for partially hyperbolic diffeomorphisms and group actions.

We now summarize in more detail the previous results in this area:
\begin{itemize}
\item Veech \cite{veech} studied the case when $f$ is a partially hyperbolic toral automorphism
and established existence and regularity results for solutions to (\ref{e=livsic}).
In these examples, there is a definite
loss of regularity between coboundary and transfer function. The examples studied
by Veech differ from those treated here in that they do not have the property of 
accessibility (although they have the weaker property of essential accessibility).

\item Dolgopyat \cite{D} studied equations (\ref{e=livsic}) and (\ref{e=livsic2}) for 
a special class of partially hyperbolic diffeomorphisms -- the compact group extensions of Anosov diffeomorphisms --
in the case where the base map preserves a Gibbs state $\mu$ with H{\"o}lder potential.
Assuming rapid mixing of the group extension with respect to $\mu$, \cite{D} showed that
if the coboundary $\phi$ is $C^\infty$, then any transfer function $\Phi\in L^2(\mu\times {\hbox{Haar}})$
is also $C^\infty$.
Dolgopyat also gave an  example of a partially hyperbolic diffeomorphism with a $C^\infty$ coboundary 
whose transfer map is continuous, but not $C^1$. This example, like Veech's, is essentially accessible,
but not accessible.  We note that when the Gibbs measure $\mu$ is volume, then the rapid mixing assumption in 
\cite{D} is equivalent to accessibility.

\item De la Llave \cite{dlL2}, extended the work of \cite{KK} to give 
some regularity results for the transfer function under strong (nongeneric) local accessibility/regularity hypotheses on bundles.  De la Llave's approach focuses on bootstrapping the regularity of the transfer function from
$L^p$ to continuity and higher smoothness classes
using the transverse regularity of the stable and unstable foliations
in $M$.  For this reason, he makes strong regularity hypotheses on this transverse regularity.
\end{itemize}

While there are superficial similaries between these previous results and Theorem~\ref{t=main},
the approach here, especially in parts II and IV, is fundamentally new and does not
rely on these results.  In particular,
to establish regularity of a transfer function, we take advantage of a form of self-similarity of its graph
in the central directions of $M$.  This self-similarity, known as {\em $C^r$ homogeneity}
is discussed in more detail in the following section.

\section{Techniques in the proof of Theorem~\ref{t=main}}\label{s=techniques}

The proof of parts I and III of Theorem~\ref{t=main} use recent work of Avila, Santamaria and Viana on 
sections of bundles with various saturation properties.  
In \cite{ASV}, they apply these results to show that under suitable conditions, matrix cocycles over partially hyperbolic systems have a nonvanishing Lyapunov exponent.  
Parts I and III of Theorem~\ref{t=main} are translations of some of the main results in \cite{ASV} to the abelian cocycle setting. 

The regularity results in Theorem~\ref{t=main} -- parts II and IV --
comprise the bulk of this paper.

To investigate the regularity of a solution $\Phi$, we examine the graph of $\Phi$ in $M\times \RR$. 
If $\phi$ is H{\"o}lder continuous, then the 
stable and unstable foliations $\W^s$ and $\W^u$ for $f$ lift to two ``stable and unstable'' foliations 
$\W^s_\phi$ and $\W^u_\phi$ of $M\times \RR$, whose leaves
are graphs of H{\"o}lder continuous functions into $\RR$.  These lifted foliations are
invariant under the skew product $(x,t)\mapsto (f(x), t+ \phi(x))$.  
The fact that $\Phi$ satisfies the equation $\phi = \Phi\circ f - \Phi + c$,
for some $c\in\RR$, implies
that the graph of $\Phi$ is saturated by leaves of the lifted foliations.
The leafwise and transverse regularity of these foliations determine
the regularity of $\Phi$. In the most general setting of Theorem~\ref{t=main}, part II,
these foliations are both leafwise and transversely H{\"o}lder continuous,
and this implies the  H{\"o}lder regularity of $\Phi$ when $f$ is accessible.

The proof of higher regularity in part IV
has two main components.  We first describe a simplified version of the proof under an
additional assumption on $f$ called dynamical coherence. 

\begin{dfinition}
A partially hyperbolic diffeomorphism $f$ is {\em dynamically coherent} 
if the distributions $E^c\oplus E^u$, and $E^c\oplus E^s$
are integrable, and everywhere tangent to
foliations $\W^{cu}$ and $\W^{cs}$.
\end{dfinition}
 
If $f$ is dynamically coherent, then
there is also a central foliation $\cW^c$, tangent to $E^c$,
whose leaves are obtained by intersecting the leaves of $\W^{cu}$
and $\W^{cs}$.  The normally hyperbolic theory \cite{HPS} implies
that the leaves of $\W^{cu}$ are
then  bifoliated by the leaves
of $\W^c$ and $\W^u$, and the leaves of $\W^{cs}$ are bifoliated by the leaves
of $\W^c$ and $\W^s$.

Suppose that $f$ is dynamically coherent and that $f$ and $\phi$ satisfy the hypotheses
of part IV of Theorem~\ref{t=main}, for some $k\geq 2$ and $r<k-1$ or $r=1$.  Under these
assumptions, here are the two components of the proof.  The first part of the proof
is to show that $\Phi$ is uniformly $C^r$ 
along individual leaves of $\W^s$, $\W^u$ and $\W^c$.
The second part is to employ a result of Journ\'e to show that smoothness of $\Phi$
along leaves of these three foliations implies smoothness of $\Phi$.

To show that $\Phi$ is smooth along the leaves of $\W^s$ and $\W^u$,
we examine again the lifted foliations for the associated skew product.  The assumption
that $\phi$ is $C^{k}$ implies that the leaves of these lifted foliations are 
$C^r$ (in fact, they are $C^{k}$). This part of the proof does not require dynamical coherence
or accessibility.

To show that $\Phi$ is smooth along leaves of the central foliation, one can use accessibility
and strong $r$-bunching to show that the graph of $\Phi$ over any central leaf $\W^c(x)$ of $f$ is
$C^r$ homogeneous. More precisely, setting $N'= \W^c(x)\times \RR$ and
$N = \{(y,\Phi(y)) :  y\in \W^c(x)\}\subset N'$, we show that
the manifold $N$ is {\em $C^r$ homogeneous in $N'$}:
for any two points $p, q\in N$, there is a $C^r$ local diffeomorphism of $N'$ 
sending $p$ to $q$ and preserving $N$. $C^1$-homogeneous subsets of a manifold
have a remarkable property:

\begin{theorem}\label{t=rss}\cite{rss}
Any locally compact subset $N$ of a  $C^1$ manifold  $N'$ that is $C^1$ homogeneous in $N'$
is a $C^1$ submanifold of $N'$
\end{theorem}

If $r=1$, we can apply this result to obtain 
that the graph of $\Phi$ is $C^1$ over any center manifold.  Hence $\Phi$ is $C^1$
over center, stable, and unstable leaves, which implies that $\Phi$ is $C^1$.
This completes the proof in the case $r=1$ (assuming dynamical coherence).

In fact we do not use the results in \cite{rss} in the proof of Theorem~\ref{t=main} 
but employ a different technique
to establish smoothness, which also works for $r>1$ and in the non-dynamically
coherent case. Our methods also show:

\begin{main}\label{t.Crsubmanif} For any integer $k\geq 2$, any $C^k$ homogeneous, $C^1$
submanifold of a $C^k$ manifold is a $C^k$ submanifold.
\end{main}

Theorem~\ref{t.Crsubmanif} is stated in \cite{rss} without proof.  In a recent article \cite{skopenkov}
by the third author of \cite{rss}, the problem of whether a $C^k$ homogeneous submanifold
a $C^k$ submanifold is stated as open for $k>1$ (and conjectured to hold).  
Combining Theorems~\ref{t=rss} and ~\ref{t.Crsubmanif}, we obtain
a proof of the conjecture in \cite{skopenkov}:

\begin{corollary}\label{c=homog} Let $k\geq 1$ be an integer. Then any locally compact subset of
a $C^k$ manifold that is $C^k$ homogeneous is a $C^k$ submanifold.
\end{corollary}

A proof of Theorem~\ref{t.Crsubmanif} is given in \cite{bm} under the assumption
that the submanifold is $C^k$ homogeneous under the continuous action of a Lie group
by diffeomorphisms. 
It appears that
Theorem~\ref{t.Crsubmanif} is not an obvious consequence of Theorem~\ref{t=rss} (see the remark after the proof
of Lemma~\ref{l=smoothgraph}).

Returning to the proof of Theorem~\ref{t=main},
assuming dynamical coherence and using Corollary~\ref{c=homog}, 
one can obtain under the hypotheses of part IV that the graph of
the transfer function $\Phi$ over each center manifold is $C^{\lfloor r\rfloor}$.
With some more work, one can obtain that the graph of
the transfer function $\Phi$ over each center manifold is $C^{r}$.
A result of Journ\'e \cite{J} implies that for any $r>1$ that is not an integer, 
and any two transverse foliations with uniformly $C^{r}$ leaves, if a function $\Phi$ is uniformly $C^{r}$
along the leaves of both foliations, then it is uniformly $C^{r}$.
Since $f$ is assumed to be dynamically coherent, the $\cW^c$
and $\cW^s$ foliations transversely subfoliate the leaves of $\cW^{cs}$ .
Applying Journ\'e's result using $\cW^c$ and $\cW^s$, we obtain that
$\Phi$ is $C^{r}$ along the leaves of $\cW^{cs}$.  
Applying Journ\'e's theorem again, this time with
$\cW^{cs}$ and $\cW^u$, we obtain that $\Phi$  is $C^r$.

We have just described a proof of 
part IV under the assumption that $f$ is dynamically coherent.  If we drop the assumption
of dynamical coherence, the assertion that $\Phi$ is ``$C^r$ along center manifolds'' no
longer makes sense, as $f$ might not have center manifolds.  One can find locally invariant
center manifolds that are ``nearly'' tangent to the center distribution (as in \cite{BWannals}),
but the argument described above does not work for these manifolds.  
The analysis becomes considerably more delicate and is described in more detail 
in Section~\ref{s=journe}.  As one of the components in our argument, we prove
a strengthened version of Journ\'e's theorem (Theorem~\ref{t=journe}) that
works for plaque families as well as foliations, and replaces the assumption
of smoothness along leaves with the existence of an ``approximate $r$-jet'' 
at the basepoint of each plaque.

The main result that lies behind the proof of Theorem~\ref{t=main}, part IV is a saturated section 
theorem for fibered partially hyperbolic systems (Theorem~\ref{t=C^rsec}). A fibered partially hyperbolic
diffeomorphism is defined on a fiber bundle and is also a bundle isomorphism, covering a partially hyperbolic diffeomorphism (see Section~\ref{s=bundle}).  In this context, Theorem~\ref{t=C^rsec} states that under the additional
hypotheses that the bundle diffeomorphism is suitably bunched, and
the base diffeomorphism is accessible, then any continuous section of the bundle whose image is 
an accessibility class for the lifted map is in fact a smooth section.  Using Theorem~\ref{t=C^rsec}
it is also possible to extend in part the conclusions of Theorem~\ref{t=main} part IV to (suitably bunched) cocycles
taking values in other Lie groups.  The details are not carried out here, but the reader is referred
to \cite{NT, BWskew, AV}, where some of the relevant technical considerations are addressed (see also
the remark after the statement of Theorem~\ref{t=C^rsec} in Section~\ref{s=bundle}).

Theorem~\ref{t=C^rsec} would follow immediately if the following conjecture is correct.
 \begin{conjecture}\label{c=smoothaccess} Let $f\colon M\to M$ be $C^r$, partially hyperbolic and $r$-bunched.  Then every accessibility class for $f$ is an injectively immersed, $C^r$ submanifold of $M$.
\end{conjecture}
For locally compact accessibility classes, it should be possible to prove Conjecture~\ref{c=smoothaccess}
using the techniques from \cite{rss} to show that the accessiblity class is a submanifold
and the methods developed in this paper to show that the submanifold is smooth.  

\section{Partial hyperbolicity and bunching conditions}\label{s=defs}

We now define the bunching hypotheses in Theorem~\ref{t=main}; to
do so, we give a more precise definition of partial hyperbolicity.                     
Let $f:M\to M$ be a diffeomorphism of a compact 
manifold $M$.  We say that $f$ is {\em partially hyperbolic} if the 
following holds.
First, there is a nontrivial splitting of the tangent bundle,
$TM=E^s\oplus E^c\oplus E^u$, that is invariant under
the derivative map $Tf$.
Further, there is
a Riemannian metric for which we can choose continuous positive
functions $\nu$,  $\hat\nu$, $\gamma$ and $\hat\gamma$ with
\begin{eqnarray}\label{e=defn}
 \nu, \hat\nu < 1\quad\hbox{ and } \quad
 \nu <  \gamma < \hat\gamma^{-1} < \hat\nu^{-1}
 \end{eqnarray}
such that, for any unit vector $v\in T_pM$,
\begin{alignat}{4}\label{e=defn2}
\|&Tf v\| <  \nu(p) ,&\qquad \hbox{if }v\in E^s(p),  \\
\gamma(p) < \|&Tf v\|  <
\hat\gamma(p)^{-1} ,&\qquad \hbox{if } v\in E^c(p), \\
\hat\nu(p)^{-1}< \|&Tf v\|  ,&\qquad \hbox{if } v\in E^u(p).
\end{alignat}

We say that $f$ is {\em center bunched} if 
 the functions
$\nu, \hat\nu, \gamma$, and $\hat\gamma$
can be chosen so that:
\begin{eqnarray}\label{e=bunch}
\max\{\nu,\hat\nu\} < \gamma\hat\gamma.
\end{eqnarray}

Center bunching means that the hyperbolicity of $f$ dominates the
nonconformality of $Tf$ on the center. 
Inequality (\ref{e=bunch}) always holds
when $Tf\vert_{E^c}$ is conformal.
For then we have $\|T_pf v\| = \|T_pf|_{E^c(p)}\|$
for any unit vector $v \in E^c(p)$, and hence we can choose $\gamma(p)$
slightly  smaller and $\hat\gamma(p)^{-1}$ slightly bigger than
 $$
 \|T_pf|_{E^c(p)}\|.
 $$
By doing this we may make the ratio $\gamma(p)/\hat\gamma(p)^{-1} = \gamma(p)\hat\gamma(p)$ 
arbitrarily close to 1, and hence larger than both $\nu(p)$ and
$\hat\nu(p)$. In particular, center bunching holds whenever $E^c$ is one-dimensional.
The center bunching hypothesis considered here is natural and appears
in other contexts, e.g. \cite{BP1, BV, AV, NT, NicP}.  

For $r>0$, we say that $f$ is {\em $r$-bunched} if 
 the functions
$\nu, \hat\nu, \gamma$, and $\hat\gamma$
can be chosen so that:
\begin{eqnarray}\label{e=rbunch}
\nu < \gamma^r, \qquad \hat\nu < \hat\gamma^r\\
\nu < \gamma\hat\gamma^r, \quad \hbox{and } \quad  \hat\nu < \hat\gamma\gamma^r.
\end{eqnarray}

Note that every partially hyperbolic diffeomorphism is $r$-bunched, for some $r>0$.  
The condition of $0$-bunching is
merely a restatement of partial hyperbolicity, and $1$-bunching is center bunching.
The first pair of inequalities in (\ref{e=rbunch}) are $r$-normal hyperbolicity
conditions; when $f$ is dynamically coherent, these inequalities
ensure that the leaves of $\W^{cu}$, $\W^{cs}$, and $\W^c$ are $C^r$. 
Combined with the first group of inequalities, the second group of inequalities imply
that $E^u$ and $E^s$ are ``$C^r$ in the direction of  $E^c$.''  More
precisely, in the case that $f$
is dynamically coherent, the $r$-bunching inqualities imply that the
restriction of $E^u$ to $\W^{cu}$ leaves is a $C^r$ bundle and the
restriction of $E^s$ to $\W^{cs}$ leaves is a $C^r$ bundle.

For $r>0$, we say that $f$ is {\em strongly $r$-bunched} if the functions
$\nu, \hat\nu, \gamma$, and $\hat\gamma$
can be chosen so that:
\begin{eqnarray}\label{e=rbunchstrong}
\max\{\nu,\hat\nu\} < \gamma^r, \qquad \max\{\nu,\hat\nu\} < \hat\gamma^r\\
\nu < \gamma\hat\gamma^r, \quad \hbox{and } \quad  \hat\nu < \hat\gamma\gamma^r.
\end{eqnarray}
We remark that if $f$ is partially hyperbolic and there exists a Riemannian 
metric in which $Tf\vert_{E^c}$ is isometric, then $f$ is strongly $r$-bunched, for every $r>0$;
given a metric $\|\cdot\|$ for which $f$ satisfies (\ref{e=defn2}), 
and another metric $\|\cdot\|'$ in which $Tf\vert_{E^c}$ is isometric, 
it is a straighforward exercise to construct a Riemannian metric $\|\cdot\|''$  for which inequalities (\ref{e=rbunchstrong}) hold, with $\gamma=\hat\gamma\equiv 1$.

The reason strong $r$-bunching appears as a hypothesis in Theorem~\ref{t=main}
is the following.  Suppose that $f$ is partially hyperbolic and
that $\phi\colon M\to \RR$ is $C^1$.  Then the skew product
$f_\phi\colon M\times \RR/\ZZ \to M\times \RR/\ZZ$ given by
$$f_\phi(x,t) = (f(x), t+\phi(x))$$
is partially hyperbolic, and if $f$ is strongly $r$-bunched then $f_\phi$
is $r$-bunched.  This skew product and the corresponding lifted skew product
on $M\times \RR$ appears in a central way in our analysis, as we 
explain in the following section.

\subsection{Notation}

Let $a$ and $b$ be real-valued functions, with $b\ne 0$.
The notation $a=O(b)$ means that the ratio $|a/b|$ is bounded above,
and $a=\Omega(b)$ means $|a/b|$ is bounded below; $a=\Theta(b)$ means that
$|a/b|$ is bounded above and below. Finally,  $a=o(b)$ means that
$|a/b|\to 0$ as $b\to 0$.  
Usually $a$ and $b$ will depend on either an integer
$j$ or a real number $t$ and on one or more points in $M$. 
The constant $C$ bounding the appropriate ratios 
must be independent of $n$ or $t$ and the choice of the points. 

The notation $\alpha < \beta$, where 
$\alpha$ and $\beta$ are continuous functions, means that the inequality
holds pointwise, and the function $\min\{\alpha,\beta\}$ takes
the value $\min\{\alpha(p),\beta(p)\}$ at the point $p$.

We denote the Euclidean norm by $|\cdot|$. 
If  $X$ is a metric space and
$r>0$ and $x\in X$, the notation $B_{X}(x,r)$ denotes the open ball about $x$ of radius $r$.  
If the subscript is omitted, then the ball is understood to be in $M$.
Throughout the paper, $r$ always denotes a real number and
$j,k,\ell,m,n$ always denote integers.  $I$ denotes the interval $(-1,1)\subset \RR$,
and $I^n\subset \RR^n$ the $n$-fold product. 

If $\gamma_1$ and $\gamma_2$ are paths in $M$, then $\gamma_1\cdot\gamma_2$
denotes the concatenated path, and $\overline\gamma_1$ denotes the reverse path.

Suppose that $\F$ is a foliation of an $m$-manifold $M$ 
with $d$-dimensional smooth leaves. 
For $r>0$, we denote by 
$\F(x,r)$ the connected component of $x$ 
in the intersection of $\F(x)$ with the ball $B(x,r)$.

A {\em foliation box for
$\F$} is the image $U$ of 
$\RR^{m-d}\times \RR^{d}$ under a homeomorphism that sends each
vertical $\RR^d$-slice into a leaf of $\F$. The images of the
vertical $\RR^d$-slices will be called {\em local leaves of $\F$ in $U$}.

A {\em smooth transversal} to $\F$ in $U$ is a smooth codimension-$d$
disk in $U$ that intersects each local leaf in $U$ exactly once and whose
tangent bundle is uniformly transverse to $T\F$.
If $\Sigma_1$ and $\Sigma_2$ are two smooth transversals to $\F$ in $U$,
we have the {\em holonomy map} $h_{\F}: \Sigma_1 \to \Sigma_2$,
which takes a point in $\Sigma_1$ to the intersection  of its local leaf
in $U$ with $\Sigma_2$.

Finally, for $r>1$ a nonintegral real number, $M, N$ smooth manifolds, the $C^r$ metric 
on $C^r(M,N)$ is defined in local charts by:
$$d_{C^r}(f,g) =  d_{C^{\lfloor r\rfloor}}(f,g) + d_{C^0}(D^{\lfloor r\rfloor}f, D^{\lfloor r\rfloor}g).$$
This metric generates the (weak) $C^r$ topology on $C^r(M,N)$.

\section{The partially hyperbolic skew product associated to a cocycle}
Let $f:M\to M$ be $C^k$ and partially hyperbolic and let $\phi:M\to \RR$ be $C^{\ell,\alpha}$, for some integer $\ell\geq 0$ and $\alpha \in [0,1]$, with $0 < \ell + \alpha \leq k$.  
Define the skew product $f_\phi:M\times \RR\to M\times \RR$ by
$$f_\phi(p,t) = (f(p), t + \phi(p)).$$

The following proposition is the starting point for our proof of Theorem~\ref{t=main}.
\begin{proposition}\label{p=liftfol} There exist foliations $\W^u_\phi, \W^s_\phi$ of $M\times \RR$ with the following properties.
\begin{enumerate}
\item\label{liftfol1} The leaves of $\W^u_\phi, \W^s_\phi$ are $C^{\ell,\alpha}$.
\item\label{liftfol2} The leaves of $\W^u_\phi$ project to leaves of $\W^u$, and the leaves of $\W^s_\phi$ project to leaves of $\W^s$. Moreover, $(x',t')\in \W^s_\phi(x,t)$ if and only if $x'\in \cW^s(x)$ and 
        $$\liminf_{n\to\infty} d(f_\phi^n(x,t),   f_\phi^n(x',t'))=0.$$
\item\label{liftfol3} Define $T\colon \RR \times M \to \RR\times M$ by $T_t(z,s) = (z, s+t)$.
Then for all $z\in M$ and $s,t\in \RR$: 
$$\W^s_\phi(z, s + t) = T_t\W^s_\phi(z,s).$$
\item\label{liftfol4} If $(x,t)\in M\times\RR$ and $(x',t')\in \cW^s_\phi(x,t)$, then
$$t' - t = \sum_{i=0}^\infty \phi(f^{i}(x')) - \phi(f^{i}(x)) = PCF_{(x,x')} \phi,$$
and if  $(x',t')\in \cW^u_\phi(x,t)$, then
  $$t'-t = \sum_{i=1}^\infty \phi(f^{-i}(x)) - \phi(f^{-i}(x')) = PCF_{(x,x')} \phi.$$
\end{enumerate} 
\end{proposition}

\begin{proof}  The map $f_\phi$ covers the map  $(x,t) \mapsto (f(x), t+\phi(x))$ on the compact
manifold $M\times \RR/\ZZ$, which we also denote by $f_\phi$

In the case where $\ell\geq 1$, (\ref{liftfol1}) and (\ref{liftfol2}) follow directly from the fact that $f_\phi$ is $C^{\ell,\alpha}$
and partially hyperbolic.   The invariant foliations on $M\times \RR/\ZZ$ lift to invariant
foliations on $M\times \RR$. 

For $\ell=0$,  (\ref{liftfol1}) and (\ref{liftfol2}) are the content of Proposition~\ref{p=liftfolHolder},
which is proved in Section~\ref{s=Holder}.

Since $T_t\circ f_\phi = f_\phi\circ T_t$ for all $t\in\RR$, (\ref{liftfol3}) follows easily from  (\ref{liftfol2}).
Finally, (\ref{liftfol4}) is an easy consequence of (\ref{liftfol3}).
\end{proof}

Throughout the rest of the paper, we will mine extensively the properties
of the foliations $\W^s_\phi$ and $\W^u_\phi$: the regularity of the leaves,
their transverse regularity, and their accessibility properties.

This focus on the lifted foliations $\W^s_\phi$ and $\W^u_\phi$ is not
entirely new. Notably, Niti\c c\u a and T{\"o}r{\"o}k \cite{NT} established the 
regularity of solutions to equation (\ref{e=livsic2})
when $f$ is an Anosov diffeomorphism by examining these lifted foliations. 
The key observation in \cite{NT} is that {\em the smoothness of the leaves of 
$\W^s_\phi$ and $\W^u_\phi$ determines the smoothness of the 
transfer function along the leaves of $\W^s$ and $\W^u$.}  
The advantage of the approach in \cite{NT} is that it allowed them to prove
a natural generalization of Theorem~\ref{t=livsicold} to cocycles 
taking values in nonabelian lie groups; provided that the induced skew product
for such a cocycle is partially hyperbolic, the smoothness of the lifted invariant foliations
determines the smoothness of transfer functions when $f$ is Anosov. This focus on the 
foliations for the skew product associated to the cocycle turns out to be crucial
in our setting. 

\section{Saturated sections of admissible bundles}\label{s=asv}

In this section, we define a key property called {\em saturation} and
present some general results about saturated sections of bundles.
In the next section, we apply these results in the setting
of abelian cocycles to prove parts I and III of Theorem~\ref{t=main}.
Throughout this section, $f\colon M\to M$ denotes a partially hyperbolic
diffeomorphism.

Let $N$ be a  manifold, and let  $\pi\colon \cB \to M$  be a fiber bundle,
with fiber $N$. We say that $\cB$ is {\em admissible} if there exist foliations  $\W^s_{\hbox{\tiny lift}}$, $\W^u_{\hbox{\tiny lift}}$ of $\cB$ (not necessarily with smooth leaves) 
such that, for every $z\in \cB$ and $\ast\in \{s,u\}$, 
the restriction of $\pi$ to $\W^\ast_{\hbox{\tiny lift}}(z)$
is a homeomorphism onto $\W^\ast(\pi(z))$.

A more general definition
of admissibility for more general bundles in terms of holonomy
maps is given in \cite{ASV}; we remark that two definitions
are equivalent in this context.  
If $\pi\colon \cB\to M$ is an admissible bundle, then given any $su$-path $\gamma\colon [0,1]\to M$
and any point $z\in \pi^{-1}(\gamma(0))$, there is a unique 
path $\tilde\gamma_z\colon [0,1]\to \cB$ such that:
\begin{itemize}
\item $\pi\tilde\gamma_z = \gamma$,
\item $\tilde\gamma_z(0) = z$,
\item $\tilde\gamma_z$  is a concatenation of finitely many subpaths, 
each of which lies entirely in a single leaf of $\cW^s_{\hbox{\tiny lift}}$,
or $\cW^u_{\hbox{\tiny lift}}$. 
\end{itemize}
We call $\tilde \gamma_z$ an {\em $su$-lift path} and say that  
$\tilde\gamma_z$  is an  {\em $su$-lift loop} if $\tilde\gamma_z(0) = \tilde\gamma_z(1) = z$.
For a fixed $su$-path $\gamma$, the map $H_\gamma\colon \pi^{-1}(\gamma(0))\to \pi^{-1}(\gamma(1))$
that sends $z\in \pi^{-1}(\gamma(0))$ to $\tilde\gamma_z(1) \in \pi^{-1}(\gamma(1))$
is a homeomorphism.  It is easy to see that $H_{\gamma_1\cdot\gamma_2} = H_{\gamma_2}\circ H_{\gamma_1}$
and $H_{\overline{\gamma}} = H_\gamma^{-1}$.

Recall that any accessible sequence $\cS=(x_1,\ldots,x_k)$ determines
an $su$-path $\gamma_{\cS}$.  We fix the convention that
$\gamma_{\cS}$ is a concatenation of leafwise distance-minimizing arcs,
each lying in an alternating sequences of single leaves of $\cW^s$ or $\cW^u$.
Using this identification, we define the holonomy $H_{\cS}\colon
\pi^{-1}(x_1)\to \pi^{-1}(x_k)$ by setting  $H_{\cS} = H_{\gamma_{\cS}}$;
since the leaves of $\cW^u$, $\cW^s$, $\cW^u_{\hbox{\tiny lift}}$, and $\cW^s_{\hbox{\tiny lift}}$
are all contractible, $H_{\cS}$ is well-defined.

\begin{dfinition}\label{d=satdef} Let $\pi\colon \cB\to M$ be an admissible bundle.
A section $\sigma \colon M\to \cB$ is:
\begin{itemize}
\item {\em $u$-saturated} if for every $z\in \sigma(M)$ we have $\W^u_{\hbox{\tiny lift}}(z)\subset \sigma(M)$,
\item {\em $s$-saturated} if for every $z\in \sigma(M)$ we have $\W^s_{\hbox{\tiny lift}}(z)\subset \sigma(M)$,
\item {\em  bisaturated} if $\sigma$ is both $u$- and $s$-saturated, and
\item {\em  bi essentially saturated} if there exist sections $\sigma^u$ ($u$-saturated) and
$\sigma^s$ ($s$-saturated) such that $$\sigma^u = \sigma^s = \sigma\,\,\quad \hbox{a.e. (volume on $M$)}$$
\end{itemize}
\end{dfinition}

It follows from the preceding discussion that if $\sigma\colon M\to \cB$ is a bisaturated section, then for any
$x\in M$, for any accessible sequence $\cS$, from $x$ to $x'$, we have $H_{\cS}(\sigma(x)) = \sigma(x')$.

\begin{theorem}\label{t=asv}\cite{ASV} Let $f\colon M\to M$ be $C^1$ and partially hyperbolic,
let $\pi\colon \cB\to M$ be an admissible bundle over $M$, 
and let $\sigma\colon M\to \cB$ be a section.
\begin{enumerate}
\item\label{asv1}If $\sigma$ is bisaturated, and $f$ is accessible, then $\sigma$ is continuous.
\item\label{asv2}If $f$ is $C^2$ and center bunched, and $\sigma$ is bi essentially saturated, then there
exists a bisaturated section $\sigma^{su}$
such that $\sigma = \sigma^{su}$ a.e. (with respect to volume on $M$)
\end{enumerate}
\end{theorem}

Since we will use a proposition from the proof of Theorem~\ref{t=asv}, (\ref{asv1}) in
our later arguments, we give a sketch of the proof here, including a statement of the key proposition
(Proposition~\ref{p=asvkey} below).  
We remark that the proof of (\ref{asv2}) adapts techniques from \cite{BWannals},
where it is shown that if $f$ is $C^2$ and center bunched,
then any bi essentially saturated {\em subset} of $M$ is essentially bisaturated;
in effect, this is just Theorem~\ref{t=asv} for the bundle
$\cB=M\times\{0,1\}$, with $\W^\ast_{\hbox{\tiny lift}}(x,j) = \W^\ast(x)\times\{j\}$,
for $j\in\{0,1\}$.

\begin{proof}[Sketch of proof of Theorem~\ref{t=asv}, (\ref{asv1})]
We give a slightly modified version of the proof in \cite{ASV},
as we will need the results here in later sections.
The key proposition in the proof is:

\begin{proposition}[\cite{ASV}, Proposition 8.3]\label{p=asvkey} Suppose that $f$ is
accessible.  Then for every $x_0\in M$, there exists $w\in M$ and
an accessible sequence $(y_0(w),\ldots, y_k(w))$ connecting $x_0$ to $w$ and satisfying
the following property: for any $\eps>0$, there exist $\delta>0$ and $L>0$ such
that, for every $z\in B_M(w,\delta)$, there exists an accessible sequence $(y_0(z),\ldots y_K(z))$
connecting $x_0$ to $z$ and such that
$$d_M(y_j(z), y_j(w)) < \eps \quad\hbox{and} \quad d_{\cW^\ast}(y_{j-1}(z),y_j(z)) < L,\quad\hbox{for}\quad j=1,\ldots,K,$$
where $d_{\cW^\ast}$ denotes the distance along the stable or unstable leaf common to the two points.
\end{proposition}

For $K\in\ZZ_+$ and $L\geq 0$, we say that $\cS$ is an {\em $(K,L)$-accessible
sequence} if $\cS =(x_0,\ldots,x_K)$ and
$$d_{\cW^\ast}(x_{j-1},x_j) \leq L,\quad\hbox{for}\quad j=1,\ldots,K,$$
where $d_{\cW^\ast}$ denotes the distance along the stable or unstable leaf common to the two points.

If $\{\cS_y = (x_0(y),\ldots,x_K(y))\}_{y\in U}$ 
is a family of $(K,L)$ accessible sequences in $U$ and $x\in U$, we say that $\lim_{y\to x}\cS_y = \cS_x$  if
$$\lim_{y\to x} x_j(y) = x_j(x),\quad\hbox{for}\quad j=0,\ldots K,$$
and we say that $y\mapsto \cS_y$ is uniformly continuous on $U$ if $y\mapsto x_j(y)$
is uniformly continuous, for $j=0,\ldots, K$.
An accessible cycle  $(x_0, \ldots, x_{2k} = x_0)$ is {\em palindromic} if $x_i = x_{2k-i}$, for $i=1,\ldots,k$.
Note that a  palindromic accessible cycle determines an $su$-path of the form $\eta\cdot\overline\eta$;
in particular, if $\cS$ is a palindromic accessible cycle from $x$ to $x$ , then $H_{\cS}$ is the identity
map on $\pi^{-1}(x)$.

The following lemma is stronger than we need for the proof of part (1) of Theorem~\ref{t=asv}, but
will be used in later sections.
\begin{lemma}\label{l=pathfamily} Let $f$ be accessible. There exist $K\in\ZZ_+$, $L\geq 0$ and $\delta>0$
such that for every $x\in M$ there is a family of $(K,L)$- accessible 
sequences $\{\cS_{x,y}\}_{y\in B_M(x,\delta)}$  such that
$\cS_{x,y}$ connects $x$ to $y$, $\cS_{x,x}$ is a palindromic accessible cycle and
$\lim_{y\to x}\cS_{x,y} = \cS_{x,x}$.
The convergence $\cS_{x,y}\to \cS_{x,x}$ is uniform in $x$.
\end{lemma}

\begin{proof}[Proof of Lemma~\ref{l=pathfamily}]
Fix an arbitrary point $x_0\in M$.  Proposition~\ref{p=asvkey} gives a point 
$w\in M$, a neighborhood $U_w$ of $w$,  and a family of 
$(K_0, L_0)$ -accessible sequences $\{(y_0(w'),\ldots, y_{K_0}(w'))\}_{w' \in U_w}$ 
such that $(y_0(w'),\ldots, y_{K_0}(w'))$ connects
$x_0$ to $w'$, and $(y_0(w'),\ldots, y_{K_0}(w'))\to (y_0(w),\ldots, y_{K_0}(w))$ uniformly
in $w'\in U_w$. 

\begin{lemma}[Accessibility implies uniform accessibility]\label{l=unifaccess} 
Let $f$ be accessible. There exist constants $K_M, L_M$
such that any two points $x,x'$ in $M$ can be connected by an $(K_M, L_M)$-accessible
sequence. 
\end{lemma}

\begin{proof}[Proof of Lemma~\ref{l=unifaccess}] First note that, since any point in $U_w$ can be connected to $x_0$ by an $(K_0,L_0)$-accessible sequence, we can connect any two points in $U_w$ by a  $(2K_0,L_0)$-accessible sequence. 

Consider an arbitrary point $p\in M$ and let $(p= q_0, q_1,\ldots, q_{K_p}=w)$
be an $(K_p,L_p)$-accessible sequence connecting $p$ and $w$.  
Continuity of $\W^s$ and $\W^u$ implies that there
is a neighborhood $V_p$ of $p$ and a family of $(K_p,L_p)$-accessible sequences
$\{(p' = q_0(p'), q_1(p'),\ldots, q_{K_p}(p'))\}_{p'\in V_p}$ with the property
that $p'\mapsto (q_0(p'),\ldots, q_{K_p}(p'))$ is uniformly continuous on $V_p$, and the map
$p'\mapsto q_{K_p}(p')$ sends $V_p$ into  $U_w$ and $p$ to $w$.
It easily follows that any two points in $V_p$ can be connected by an
$(K_0+ 2K_y,L_0+L_y)$-accessible sequence.  Covering $M$ by neighborhoods $V_p$,
and extracting a finite subcover, we obtain by concatenating accessible sequences
that there exist constants $K_M, L_M$ such that any two points $x,x'$ in $M$ 
can be connected by an $(K_M, L_M)$-accessible sequence.
\end{proof}

Returning to the proof of Lemma~\ref{l=pathfamily}, we now fix a point $x\in M$, and let $(x= z_0, z_1,\ldots, z_{K_M}=w)$ be an
$(K_M, L_M)$-accessible sequence connecting $x$ to $w$. As above,
 there exists a neighborood $V_x$ of $x$ and a family of $(K_M,L_M)$-accessible sequences
$\{(x' = z_0(x'), z_1(x'),\ldots, z_{K_M}(x'))\}_{x'\in V_x}$ with the property
that the map $$x'\mapsto ( z_0(x'), \ldots, z_{K_M}(x') ) $$ is uniformly continuous on $V_x$, and the map
$x'\mapsto z_{K_M}(x')$ sends $V_x$ into $U_w$ and $x$ to $w$.

For
$x'\in V_x$, we define $\cS_{x,x'}$ by concatenating the accessible sequences
$(x= z_0(x), z_1(x),\ldots, z_{K_M}(x)=w)$, $(w=y_{K_0}(w), \ldots, y_{0}(w) = x_0) $,
$(x_0 = y_0(z_{K_M}(x') ),\ldots, y_{K_0}(z_{K_M}(x') ) = z_{K_M}(x'))$ and 
$(z_{K_M}(x'),\ldots, z_{0}(x') = x' )$.
Then $\{\cS_{x,x'}\}_{x'\in V_x}$ is a family of $(K,L)$-accessible sequences
with the property that $\cS_{x,x'}$ connects $x$ to $x'$, where $K=2K_0 + 2K_M$
and $L= L_0 + L_M$.

Since  $x'\mapsto ( z_0(x'), \ldots, z_{K_M}(x') ) $ is uniformly continuous on $V_x$,  and
$$\lim_{w'\to w} (y_0(w'),\ldots, y_{K_0}(w')) = (y_0(w),\ldots, y_{K_0}(w)),$$  
we obtain that 
$\lim_{x'\to x}\cS_{x,x'} = \cS_{x,x}$. By construction, $\cS_{x,x}$ is palindromic.

Finally, observe that all of the steps in this construction are uniform over $x$, and so
we can choose $\delta>0$ such that $B_M(x,\delta)\subset V_x$, for all $x$, and
further, $\lim_{x'\to x}\cS_{x,x'} = \cS_{x,x}$  uniformly in $x$.  This completes the
proof of Lemma~\ref{l=pathfamily}.
\end{proof}

Returning to the proof of Theorem~\ref{t=asv}, part (\ref{asv1}), fix a point $x\in M$,
and let $\{\cS_{x,x'}\}_{x'\in B_M(x,\delta)}$ be the family of accessible paths given by
Lemma~\ref{l=pathfamily}.  Since $\lim_{x'\to x} \cS_{x,x'} = \cS_{x,x}$ and the lifted
foliations are continuous, it follows that
$$ \lim_{x'\to x} H_{\cS_{x,x'}} = H_{\cS_{x,x}},$$
uniformly on compact sets.  Since $\cS_{x,x}$ is palindromic, we have $H_{\cS_{x,x}} = id\vert_{\pi^{-1}(x)}$.

Let $\sigma\colon M\to \cB$ be a bisaturated section.  Then for any accessible sequence $\cS$ from $x$ to $x'$,
we have $H_{\cS}(\sigma(x)) = \sigma(x')$.  But then
$$ \lim_{x'\to x} \sigma(x') = \lim_{x'\to x} H_{\cS_{x,x'}}(\sigma(x)) = H_{\cS_{x,x}}(\sigma(x)) = \sigma(x),$$
which shows that $\sigma$ is continuous at $x$.
\end{proof}

\begin{proposition}[Criterion for existence of bisaturated section]\label{p=closedloop}  Let $f$
be $C^1$, partially hyperbolic and accessible, and let $\pi\colon \cB\to M$ be admissible.
   Let $z\in \cB$ and let $x=\pi(z)$.  Then there exists a bisaturated section $\sigma\colon M\to \cB$ with
   $\sigma(x) = z$ if and only if  for every  $su$-loop $\gamma$ in $M$ with $\gamma(0)=\gamma(1) = x$,
   the lift $\tilde\gamma_z$ is
   an $su$-lift loop (with $\tilde\gamma_z(0)=\tilde\gamma_z(1) = z$).
\end{proposition}
   
\begin{proof} We first prove the ``if'' part of the proposition.  Define 
$\sigma:M\to \cB$ as follows. We first set $\sigma(x) = z$.  For each $x'\in M$,
fix an $su$-path $\gamma\colon [0,1]\to M$  from $x$ to $x'$.  Since $\cB$ is an admissible bundle,
$\gamma$ lifts to a path $\tilde\gamma_z\colon [0,1]\to \cB$ along the leaves of $\cW^s_{\hbox{\tiny lift}}$
and  $\cW^u_{\hbox{\tiny lift}}$ with $\tilde\gamma_z(0) = z$.  We set $\sigma(x') = \tilde\gamma_z(1)$.
Clearly $\pi\sigma(x') = x'$.

We first check that $\sigma$ is well-defined.  Suppose that $\gamma'\colon [0,1]\to M$ is another
$su$-path from $x$ to $x'$.  Concatenating $\gamma$ with $\overline\gamma'$, we obtain an $su$-loop
$\gamma \overline\gamma'$ from $x$ to $x$.  By the hypotheses, the lift of 
$\gamma \overline\gamma'$ through $z$ is an $su$-lift loop
in $\cB$.  But this implies that  $\tilde\gamma_z(1) = \tilde\gamma'_z(1)$.

The same argument shows that $\sigma$ is bisaturated.  Fix $y\in M$ and let $y'\in \cW^s(y)$.
We claim that $\sigma(y')\in \cW^s_{\hbox{\tiny lift}}(\sigma(y))$. To see this,
fix two $su$-paths in $M$, one from $x$ to $y$, and one from $x$ to $y'$.   Concatenating
these paths with a path from $y$ to $y'$ along $\cW^s(y)$, we obtain an $su$-loop $\gamma$
through $x$.  By hypothesis, the lift $\tilde\gamma_z$ is a lifted  $su$-loop.  It is easy to see that this means that $\sigma(y')\in \cW^s_{\hbox{\tiny lift}}(\sigma(y))$.  Hence $\sigma$ is $s$-saturated.  Similarly, $\sigma$
is $u$-saturated, and so $\sigma$ is bisaturated.

The ``only if'' part of the proposition is straightforward.
\end{proof}

\begin{remark} Upon careful inspection of the proofs in this subsection, one sees that the existence of
foliations $\cW^s_{\hbox{\tiny lift}}$ and $\cW^u_{\hbox{\tiny lift}}$ is not an essential 
component of the arguments. For example, instead of assuming the existence of 
these foliations, one might instead assume (in the context
where $\cB$ is a smooth fiber bundle) the existence of $E^u$ and $E^s$ {\em connections} on $\cB$, that
is, the existence of subbundles $E^u_\phi$ and $E^s_\phi$ of $T\cB$, disjoint from $\ker T\pi$,
that project to $E^u$ and $E^s$ under $T\pi$. In this context, at least when $E^u_\phi$ and $E^s_\phi$ are smooth,
there is a natural notion of a bisaturated section.  In particular, for every
$us$-path $\gamma$ in $M$ and $z\in \pi^{-1}(\gamma(0))$, there is a unique lift $\tilde\gamma_z$
to a path in $\cB$, projecting to $\gamma$ and everywhere tangent to $E^u_\phi$ or $E^s_\phi$.
Bisaturation of $\sigma$ in this context means that for every $su$-path $\gamma$ from $x$ to $x'$,
one has $\tilde\gamma_{\sigma(x)}(1) = \sigma(x')$.  The same proof as above
shows that a bisaturated section in this sense is also continuous.

For this reason, \cite{ASV} introduce the notions of {\em bi-continuous} and {\em bi-essentially continuous}
sections, which extract
the essential properties of a bisaturated section used in the proof of Theorem~\ref{t=asv}.
While we have no need for this more general notion here, it is worth observing that bi-continuity
might have applications in closely related contexts.

\end{remark}

\subsection{Saturated cocycles: proof of Theorem~\ref{t=main}, parts I and III}

We now translate the previous results into the context of abelian cocycles. Let $\phi:M\to \RR$
be such a cocycle, and let $\cB=M\times\RR$ be the trivial bundle with fiber $\RR$.  Then
$\cB$ is an admissible bundle; we define the lifted foliations $\W^\ast_{\hbox{\tiny lift}}$, $\ast\in\{s,u\}$ to be 
the $f_\phi$-invariant foliations $\W^\ast_\phi$ given by Proposition~\ref{p=liftfol}. There is a natural
identification between functions $\Phi\colon M\to\RR$ and sections $\sigma_\Phi\colon M\to \cB$
via $\sigma_\Phi(x) = (x,\Phi(x))$.  Definition~\ref{d=satdef} then extends to functions
$\Phi\colon M\to \RR$ in the obvious way, where saturation is defined with respect to
the $\W^\ast_\phi$-foliations.

\begin{proposition}\label{p=satprop} Suppose that $f$ is partially hyperbolic and $\phi$ is H{\"o}lder continuous. 
 \begin{enumerate} 
 \item\label{satpart1} Assume that $f$ is accessible, and let $\Phi\colon M\to \RR$
 be continuous.
 Then there exists $c\in \RR$ such that
\begin{eqnarray}\label{e=coh}
 \phi = \Phi\circ f - \Phi + c,
 \end{eqnarray} 
if and only if $\Phi$ bisaturated. 
 
 \item\label{satpart2} If $f$ is volume-preserving and ergodic, and $\Phi:M\to \RR$ is a measurable function
 satisfying (\ref{e=coh}) ($m$-a.e.), for some $c\in\RR$, then  $\Phi$ is  bi essentially saturated.
\end{enumerate}
\end{proposition}

\begin{proof}   (\ref{satpart1}) Suppose that $\Phi$ is a continuous solution to (\ref{e=coh}).  Then (\ref{e=coh}) 
     implies that for all $x\in M$ and all $n$, we have: 
      $$f^n_\phi (x,\Phi(x)) = (f^n(x), \Phi(f^n(x)) + cn).$$ 
     Let $x'\in\cW^s(x)$.  Then
     $$\liminf_{n\to\infty} d(f_\phi^n(x,t),   f_\phi^n(x',t'))=$$
     $$\lim_{n\to\infty} d((f^n(x), \Phi(f^n(x))) ,  (f^n(x'), \Phi(f^n(x'))) )= 0,$$
     and so $(x,\Phi(x)),(x',\Phi(x'))$ lie on the same $W^s_\phi$ leaf.  This implies that $\Phi$ is $s$-saturated.        Similarly $\Phi$ is $u$-saturated, and hence bisaturated. 
 
 Suppose on the other hand that $\Phi$ is continuous and bisaturated.  
 Define a function $c\colon M\to \RR$ by $c(x) = \phi(x) - \Phi(f(x)) + \Phi(x)$.   We want to show that
 $c$ is a constant function.  
Proposition~\ref{p=liftfol}, (\ref{liftfol3}) implies that, for all $z\in M$ and $s,t\in \RR$:
\begin{eqnarray}\label{e=translate}
\W^s_\phi(z, s + t) = T_t\W^s_\phi(z,s).
\end{eqnarray}

Suppose that  $y\in \W^s(x)$. Saturation of $\Phi$ and $f_\phi$-invariance of
$\W^s_\phi$, $\W^u_\phi$ imply that:
\begin{eqnarray}\label{e=equalmanif}
\W^s_\phi(f(x),\Phi(f(x))) = \W^s_\phi(f(y),\Phi(f(y))),&\hbox{ and }\\
f_\phi(\W^s_\phi(x,\Phi(x))) = f_\phi(\W^s_\phi(y,\Phi(y))).&
\end{eqnarray}
On the other hand, invariance of the $\W^s_\phi$-foliation under $f_\phi$ implies that,
for all $z\in M$:
\begin{eqnarray*}
f_\phi(\W^s_\phi(z,\Phi(z)))& = &\W^s_\phi(f(z),\Phi(z) + \phi(z))\\
&=&\W^s_\phi(f(z), \Phi(f(z)) + (\Phi(z) - \Phi(f(z)) + \phi(z)))\\
& =& T_{\Phi(z) - \Phi(f(z)) + \phi(z)}\left(\W^s_\phi(f(z), \Phi(f(z))) \right).
\end{eqnarray*}
Equations (\ref{e=equalmanif}) and
(\ref{e=translate}) now imply that
 $$\Phi(x) - \Phi(f(x)) + \phi(x) = \Phi(y) - \Phi(f(y)) + \phi(y);$$
in other words, $c(x) = c(y)$.  Hence the function $c$ is constant along $\W^s$-leaves;
similarly, $c$ is constant along $\W^u$-leaves.  Accessibility implies that $c$
is constant.  Hence $\Phi$ and $c$ satisfy (\ref{e=livsic2}).

     (\ref{satpart2}) Let $\Phi$ be a measurable solution to (\ref{e=coh}). 
     We may assume that (\ref{e=coh}) holds on an $f$-invariant set of full volume; for points in this set, we have 
     $$f^n_\phi (x,\Phi(x)) = (f^n(x), \Phi(f^n(x)) + cn),$$  for all $n$.
     
     Choose a compact set $C\subset M$ such that $\hbox{vol}(C) > .5\hbox{vol}(M)$, on which 
     $\Phi$ is uniformly continuous.  Ergodicity of $f$ and absolute 
     continuity of $\cW^s$ implies that for almost every $x\in M$,
     and almost every $x'\in \cW^s(x)$, the pair of points $x$ and $x'$ will visit $C$ 
     simultaneously for a positive density set of times. For such a pair of points $x,x'$ we have 
     $$\liminf_{n\to\infty} d(f_\phi^n(x,t),   f_\phi^n(x',t'))=$$
     $$\liminf_{n\to\infty} d((f^n(x), \Phi(f^n(x))) ,  (f^n(x'), \Phi(f^n(x'))) )= 0,$$
     and so $(x,\Phi(x)),(x',\Phi(x'))$ lie on the same $W^s_\phi$ leaf.  
     This implies that $\Phi$ is essentially $s$-saturated: one defines
     the $s$-saturate $\Phi^s$ of $\Phi$ at (almost every) $x$ to be equal to the almost-everywhere constant
     value of $\Phi$ on $\cW^s(x)$ (see \cite{NicP} for a version of this argument when
     $f$ is Anosov).
     
     Similarly $\Phi$ is essentially $u$-saturated, and hence bi essentially saturated. \end{proof}

\begin{proof}[Proof of Theorem~\ref{t=main}, part I]

Let $f$ be $C^1$ and accessible and let $\phi:M\to \RR$ be H{\"o}lder continuous. 
Part I of Theorem~\ref{t=main} asserts that there exists a continuous function $\Phi\colon M\to \RR$ and
$c\in\RR$ satisfying (\ref{e=livsic2}) if and only if
   $PCF_{\cC}(\phi)=0$, for every accessible cycle $\cC$.
 
We start with a lemma:
 
\begin{lemma}\label{l=pcfloop} Let $\gamma$ be an $su$-loop corresponding
to the accessible cycle $\cC$. Then
$PCF_\cC(\phi)=0$ if and only if {\em every} lift of $\gamma$ to an $su$-lift path in  $M\times \RR$ is an $su$-lift loop.
\end{lemma}
   
\begin{proof}[Proof of Lemma~\ref{l=pcfloop}] Let $x\in M$ Proposition~\ref{p=liftfol}, part (\ref{liftfol4})
implies that if $\cC=(x_0,\ldots,x_k = x_0)$ is an accessible cycle, then for any $t\in \RR$ 
$$H_\cC(t) - t = \sum_{i=0}^{k-1} PCF_{(x_i,x_{i+1})}(\phi) = PCF_\cC(\phi)
$$
Let $\gamma$ be an $su$-loop corresponding to $\cC$.  Then for any $t\in\RR$, $H_\gamma(t) - t = PCF_\cC(\phi)$

Fix $t\in \RR$, and let $\tilde\gamma_t = \tilde\gamma_{x_0,t}\colon [0,1]\to M\times \RR$ be the $su$-lift path projecting to $\gamma$, with $\tilde\gamma_t(0) = (x_0,t)$.  Then
$\tilde\gamma_t(1) = (x_0, H_\gamma(t)) = (x_0, t + PCF_\cC(\phi)=0)$. Thus $PCF_\cC(\phi) = 0$ if and
only if $\tilde\gamma_t(1) = t$ if and only if $\tilde\gamma_t$ is an $su$-lift loop.  Since
$t$ was arbitrary, we obtain that $PCF_\cC(\phi)=0$ if and only if every lift of $\gamma$ to an $su$-lift path 
is an $su$-lift loop. \end{proof}

   By Proposition ~\ref{p=closedloop} and Lemma~\ref{l=pcfloop}, if
    $PCF_\cC(\phi)=0$, for every accessible cycle $\cC$, then there exists a
   bisaturated function $\Phi:M\to M\times R$. Theorem~\ref{t=asv}, part (1), plus accessibility 
   of $f$ implies that $\Phi$ is continuous.   Proposition~\ref{p=satprop} implies that there exists
   a $c\in \RR$ such that (\ref{e=coh}) holds.
   
   On the other hand, if $\Phi$ is continuous and there exists a $c\in\RR$ such that  (\ref{e=coh})
   holds, then Proposition~\ref{p=satprop}, (part \ref{satpart1}) implies that $\Phi$ is bisaturated.  
   Proposition ~\ref{p=closedloop} and Lemma~\ref{l=pcfloop} imply that $PCF_\cC(\phi)=0$, for every 
   accessible cycle $\cC$. \end{proof}

\begin{proof}[Proof of Part III of Theorem B] Assume that $f$ is 
$C^2$, volume-preserving, center bunched and accessible. 
Let $\Phi$ be a measurable solution to (\ref{e=livsic2}), for some $c\in\RR$. We
prove that there exists a continuous function $\hat\Phi$ satisfying $\Phi=\hat\Phi$ almost everywhere.
     
Since $f$ is center bunched and accessible, it is ergodic, by (\cite{BWannals}, Theorem 0.1). 
Proposition~\ref{p=satprop}, part (\ref{satpart2}) implies that 
$\Phi$ is bi essentially saturated.  Theorem~\ref{t=asv}, part (2) then implies that $\Phi$ is 
essentially bisaturated, which means there exists a bisaturated function $\hat\Phi$, with 
$\hat\Phi = \Phi$ a.e. Since $f$ is accessible, Theorem~\ref{t=asv}, part (1) then implies that $\hat\Phi$ is continuous. \end{proof}

\section{H{\"o}lder regularity: proof of Theorem~\ref{t=main}, part II.}\label{s=Holder}

Let $f\colon M\to M$ be partially hyperbolic and let $\phi\colon M\to \RR$ be $\alpha$-H{\"o}lder continuous, for some 
$\alpha>0$.  
As above, define the skew product $f_\phi\colon M\times \RR\to M\times \RR$ by
$$f_\phi(p,t) = (f(p), t + \phi(p)).$$

We start with a standard proposition showing that the stable and unstable foliations for
$f$ lift to invariant stable and unstable foliations for $f_\phi$.

\begin{proposition}\label{p=liftfolHolder} There exist foliations $\W^u_\phi, \W^s_\phi$ of $M\times \RR$ with the following properties.
\begin{enumerate}
\item The leaves of $\W^u_\phi, \W^s_\phi$ are $\alpha$-H{\"o}lder continuous.
\item The leaves of $\W^u_\phi$ project to leaves of $\W^u$, and the leaves of $\W^s_\phi$ project to leaves of $\W^s$.
 Moreover, $(x',t')\in \W^s_\phi(x,t)$ if and only if $x'\in \cW^s(x)$ and 
        $$\liminf_{n\to\infty} d(f_\phi^n(x,t),   f_\phi^n(x',t'))=0.$$
\end{enumerate}
\end{proposition}

\begin{proof} This result is by now standard (see \cite{NT}), although strictly speaking,
the proof appears in the literature only under a stronger partial hyperbolicity assumption
(in which the functions $\nu,\hat\nu,\gamma,\hat\gamma$ are assumed to be constant).
We sketch the proof under the slightly weaker hypotheses stated here.

For $x\in M$, let $\G_x = \{g \colon \W^u(x,\delta)\to \RR :  g\in C^\alpha, g(x)=0\}$. 
The number $\delta>0$ is chosen so that for all $x\in M$, if $y\in \W^u(x,\delta)$,
then $d(f(x), f(y)) \geq \hat\nu(x)^{-1} d(x,y)$. 
Notice that the function $\psi(y) = \phi(y) - \phi(x)$ belongs to $\G_x$.
The $\alpha$-norm
of an element $g\in \G_x$ is defined:
$$\|g\|_\alpha = \sup_{y\in \W^u(x,\delta)}\frac{|g(y)|}{d(x,y)^\alpha}.$$
The bundle $\G$ over $M$ with fiber $\G_x$ over
$x\in M$ has the structure of a Banach bundle.  The fiber is modelled on the Banach space
$B = \{g\colon B_{\RR^u}(0,\delta) \to \RR :  g\in C^\alpha , g(0)=0\}$,
with the norm $$\|g\|_\alpha = \sup_{v\in B_{\RR^u}(0,\delta)} \frac{|g(v)|}{|v|^\alpha}.$$
The restriction of $f$ to $\W^u$-leaves
sends $\W^u(x,\delta)$ onto $\W^u(f(x),\hat\nu(x)^{-1}\delta)$, which contains $\W^u(f(x),\delta)$.  
On $\W^u(x)\times \RR$,
the  map
$f_\phi$ takes the form 
$f_\phi(p,t) = (f(p), t + \phi(p))$, and
the induced graph transform map $\cT_x: \G_x\to \G_{f(x)}$ takes the form:
$\cT_x(g)(y) =   g(f^{-1}(y)) + \phi(f^{-1}(y)) - \phi(f^{-1}(x))$.

Suppose that $\|g\|_\alpha \leq C$.  Then 
\begin{eqnarray*}
|\cT_x(g)|_\alpha  &=& \sup_{z\in \W^u(f(x),\delta)} \frac{|\cT_x(g)(z)|}{d(f(x),z)^\alpha}\\
&\leq & \sup_{y\in \W^u(x,\delta)} \frac{|g(y) + \phi(y) - \phi(x)|}{d(f(x),f(y))^\alpha}\\
&=&  \sup_{y\in \W^u(x,\delta)} \frac{|g(y)|}{d(f(x),f(y))^\alpha} + \frac{|\phi(x) - \phi(y))|}{d(f(x),f(y))^\alpha} \\
&\leq & \hat\nu(x)^{\alpha} \left( \sup_{y\in \W^u_\delta(x)} \frac{|g(y)|}{d(x,y)^\alpha} + \frac{|\phi(x) - \phi(y))|}{d(x,y)^\alpha}\right)\\
&\leq &  \hat\nu(x)^{\alpha} \left(\|g\|_\alpha  + |\phi - \phi(x)|_\alpha\right)\\
&\leq & \hat\nu(x)^{\alpha} ( C + K) \leq C,
\end{eqnarray*}
provided that $C$ is larger than $\sup_x K/(1-\hat\nu(x))$.

Hence the closed sets $\G_x(C) = \{g\in \G_x :  \|g\|_\alpha \leq C\}$ are preserved  by the maps $\cT_x$.
Next we show that $\cT_x$ is a contraction in the $\alpha$-norm. To this end, let $g,g'\in \G_x(C)$.
Then 
\begin{eqnarray*}
\|\cT_x(g) - \cT_x(g')\|_\alpha  & = & \sup_{z\in \W^u(f(x),\delta)} \frac{|\cT_x(g)(z) - \cT_x(g')(z)|}{d(f(x),z)^\alpha}
\end{eqnarray*}
\begin{eqnarray*}
\qquad\qquad\qquad&\leq & \sup_{y\in \W^u(x,\delta)} \frac{|g(y) + \phi(y) - \phi(x) - (g'(y) + \phi(y) - \phi(x))|}{d(f(x),f(y))^\alpha}\\
&=&  \sup_{y\in \W^u(x,\delta)} \frac{|g(y) - g'(y)|}{d(f(x),f(y))^\alpha}\\
&\leq & \hat\nu(x)^{\alpha} \|g-g'\|_\alpha.
\end{eqnarray*}
The invariant section theorem (\cite{HPS}, Theorem 3.1) now implies that there is a unique $\cT$-invariant section $\sigma:M\to \G_x(C)$.
It is easy to check that the set $\hW^u_\phi(p,t) = \{(y,t+ \sigma_p(y)) :  y\in \cW^u(p,\delta)\}$
is a local unstable manifold for $f_\phi$.  The rest of the proof is standard.
\end{proof}

Fix a foliation box $U$ for $\W^s$.  For any two smooth transversals
$\Sigma$, $\Sigma'$ in $U$, there is the $\W^s$-holonomy map from
$\Sigma$ to $\Sigma'$ that sends $x\in \Sigma$ to the unique
point of intersection $x'$ between $\cW^s(x)$ and $\Sigma'$.  For
any such $\Sigma,\Sigma'$ there is also a well-defined $\cW^s_\phi$-holonomy 
between $\Sigma\times \RR$ and $\Sigma'\times \RR$,
sending $(x,t)\in \Sigma\times \RR$ to the unique
point of intersection $(x',t')$ between $\cW^s_\phi(x,t)$ and $\Sigma'\times \RR$.
Since the $\W^s$ leaves lift to $\W^s_\phi$-leaves, the $\W^s_\phi$ holonomy
covers the $\W^s$ holonomy under the natural projection.

\begin{proposition}\label{p=Holder} Suppose that $f$ is $C^1$ and $\phi$
is $\alpha$-H{\"o}lder continuous, for some $\alpha\in(0,1]$. 
Then the $\W^s_\phi$ and $\W^u_\phi$ holonomy maps are uniformly H{\"o}lder continuous.
Any $\theta\in (0,\alpha]$ satisfying the pointwise inequalities:
\begin{eqnarray}\label{e=holderbunch}\nu < (\nu\hat\mu)^{\theta/\alpha} \quad\hbox{and}\quad \nu\gamma^{-1} < (\nu\hat\mu)^{\theta/\alpha}
\end{eqnarray}
is a H{\"o}lder exponent for the $\cW^s_\phi$ holonomy,
where $\nu, \gamma, \hat\mu:M\to \RR$ are any continuous functions satisfying,
for every $p\in M$ and any unit vector $v\in T_pM$:
$$v\in E^s(p) \,\Rightarrow \,\|T_pf v\| < \nu(p),\quad v\in E^s(p)\, \Rightarrow \,\gamma(p) < \|T_pf v\|,$$
and
$$ v\in E^u(p)\, \Rightarrow \,\|T_pf v\| \leq \hat\mu(p)^{-1},$$
for some Riemannian metric.
\end{proposition}

By considering the trivial (constant) cocycle, we also obtain:

\begin{corollary}\label{c=Holder} The  stable holonomy maps for a $C^1$ partially
hyperbolic diffeomorphism $f$ are uniformly H{\"o}lder continuous. Any $\theta\in (0,1]$ satisfying
$$\nu < (\nu\hat\mu^{-1})^\theta\quad\hbox{and}\quad \nu\gamma^{-1} < (\nu\hat\mu^{-1})^\theta$$
is a H{\"o}lder exponent for the stable holonomy,
where $\nu, \gamma, \hat\mu$ are defined as in Proposition~\ref{p=Holder}.
\end{corollary}

\begin{remark}  In  (\cite{PSW}, Theorem A)  it is shown that the holonomy 
maps for $\W^u$ and $\W^s$ are H{\"o}lder continuous
if $f$ is at least $C^2$ (or $C^{1+\alpha}$, for some $\alpha>0$).
The proof in \cite{PSW} uses a graph transform argument and an invariant section theorem to show that
the plaques of $\W^u$ and $\W^s$ form a H{\"o}lder continuous family.  Here in the proof of Proposition~\ref{p=liftfol}, 
as in the first part of the proof in \cite{PSW}, we have exhibited the plaques of $\W^u_\phi$
as an invariant section of a fiber-contracting bundle map $\cT$.
It is not possible, however, to carry over the rest of the proof in \cite{PSW} to this setting:
the low regularity of $\cT$ prevents one from using a H{\"o}lder section theorem to conclude
that the invariant section is H{\"o}lder continuous.

Hence we employ a different approach to prove that the holonomy maps are H{\"o}lder continuous.  The proof
here has some similarities with the proof that stable foliations are absolutely continuous.  We fix two
transversals $\tau$ and $\tau'$ to $\W^s_\phi$ and a pair of points $x,y\in \tau$.  We iterate the picture forward until
$f_\phi^n(\tau)$ and $f_\phi^n(\tau')$ are very close and then
push $f_\phi^n(x)$ and $f_\phi^n(y)$ across a short distance to 
points $f_\phi^n(x'),f_\phi^n(y')\in f_\phi^n(\tau')$. 
The points $x',y'$ are the images of $x,y$ under
$\W^s_\phi$-holonomy; the iterate $n$ is chosen carefully so that the distance between $x$ and $y$ can
be compared to some power of the distance between $x'$ and $y'$.  Unlike the proof of absolute continuity 
of stable foliations, in which $n$ is chosen arbitrarily large,
the choice of $n$ is delicate and depends on the distance between $x$ and $y$.
We will employ this type of argument again in later sections.

As a final remark, we note that for every partially hyperbolic diffeomorphism $f$ and every
H{\"o}lder continuous cocycle $\phi$, there is a choice of $\theta>0$ satisfying (\ref{e=holderbunch}),
for some Riemannian metric.
\end{remark}

\begin{proof}[Proof of Proposition~\ref{p=Holder}.]

In this proof, we will  use the  convention that if $q$ is a point in $M$ and $j$
is an integer, then $q_j$ denotes the point $f^j(q)$, with $q_0=q$.
If $\alpha:M\to {\bf R}$ is a positive function, and $j\geq 1$ is
an integer, we set
$$
\alpha_j(p) = \alpha(p)\alpha(p_1)\cdots\alpha(p_{j-1}), 
$$
and
$$
\alpha_{-j}(p) = \alpha(p_{-j})^{-1}\alpha(p_{-j+1})^{-1}\cdots\alpha(p_{-1})^{-1}.
$$
We set $\alpha_0(p) = 1$.
Observe that $\alpha_j$ is a
multiplicative cocycle; in particular, we have
$\alpha_{-j}(p)^{-1}= \alpha_j(p_{-j})$. Note also that $(\alpha\beta)_j = \alpha_j\beta_j$, and if 
$\alpha$ is a constant function, then $\alpha_n = \alpha^n$.

Fix $\theta\in(0,\alpha]$ satisfying (\ref{e=holderbunch}). Next, fix 
a continuous positive function $\rho\colon M\to \RR_+$ satisfying:
\begin{itemize}
\item $\rho < \min\{1,\gamma\}$, and
\item $\nu\rho^{-1} \leq (\nu\hat\mu^{-1})^{\theta/\alpha}$.
\end{itemize}
We say that a smooth transversal $\Sigma$ to $\W^s$ is {\em admissible} if the angle
between $T\Sigma$ and $E^s$ is at least $\pi/4$.

The next lemma follows from an elementary inductive argument and
continuity of the functions $\nu$, $\hat\mu$ and $\rho$ (cf. \cite{BWannals}, Lemma 1.1).
\begin{lemma}\label{l=pointwise} There exists $\delta_0>0$ such that
for any $p\in M$, and for any  $p'\in\W^s(p,\delta_0)$:
\begin{enumerate}
\item for any $i\geq 0$,
$$d(p_i, p_i') \leq	 \nu_{i}(p) d(p,p');$$
\item for any admissible transversal $ \Sigma'$ to $\W^s$ at $p'$,
and any point $q'\in  \Sigma'$, if $d(p_i', q_i')<\delta_0$, for $i=1,\ldots, n$, then
$$\rho_i(p) d(p',q') \leq d(p_i', q_i') \leq	 \hat\mu_{i}(p)^{-1} d(p',q'),$$ for $i=1,\ldots, n$.
\end{enumerate}
\end{lemma}

Let $\delta_0>$ be given by this lemma; by rescaling the metric, we may assume that $\delta_0=1$.
Fix $p\in M$ and $p'\in \W^s(p,1)$.  
Let $ \Sigma$ and $ \Sigma'$ be admissible transversals to $\W^s$, with $p\in \Sigma$ and
$p'\in \Sigma'$, so that the $\W^s$-holonomy $h^s: \Sigma\to  \Sigma'$, with $h^s(p)=p'$ is well-defined.  Let
$\tau =  \Sigma\times\RR$, and let $\tau' =  \Sigma'\times\RR$.  Fix $q\in  \Sigma$ with
$d(p,q) <1$, and let $q'=h^s(q)$. 

For $(z,t)\in M\times\RR$ and $n\geq 0$, write 
$(z_n,t_n)$ for $f_\phi^n(z,t)$.  We introduce the notation
$S_n\phi(z) = \sum_{i=0}^{n-1} \phi(z_i)$, and note that $S_1\phi(z) = \phi(z)$.  
With these notations, we have $(z_n, t_n) = (z_n, t + S_n\phi(z))$.
Denote by $h^s_\phi\colon  \Sigma\times\RR\to  \Sigma\times\RR$ the $\W^s_\phi$-holonomy,
which covers the map $h^s$. 
We first establish H{\"o}lder continuity of the base holonomy map $h^s\colon  \Sigma\to  \Sigma'$.

Since $\nu  < \hat\mu^{-1}$, there exists an $n$ 
so that $d(p,q) = \Theta(\nu_n(p) \hat\mu_{n}(p))$; fix such an $n$.  Lemma~\ref{l=pointwise} applied
in the transversal $ \Sigma$ implies that 
$d(p_i, q_i) \leq  \hat\mu_{i}(p)^{-1} d(p,q) \leq O(\nu_n(p))$, for $i=1,\ldots, n$.

On the other hand, since $p'\in \W^s(p,1)$, we have $d(p_i,p'_i) \leq O(\nu_i)$, for all $i$;
in particular, $d(p_n,p'_n) \leq O(\nu_n)$.  Similarly,  $(q_n,q'_n) \leq O(\nu_n)$.
By the triangle inequality, we have that
\begin{eqnarray*}
d(p_n',q_n')& \leq & d(p_n, q_n) + d(p_n, p_n') +  d(q_n, q_n')\\
&=& O(\nu_n(p)).
\end{eqnarray*}
Now applying $f^{-n}$ to the pair of points $p_n',q_n'$  we obtain the pair of points
$p',q'$, which lie in the admissible transversal $ \Sigma'$.  
Lemma~\ref{l=pointwise} then implies that $d(p',q') \leq \rho_n(p)^{-1} d(p_n',q_n')
\leq O(\rho_n(p)^{-1} \nu_n(p))$.  Since $\rho_n(p)^{-1} \nu_n(p) < (\nu_n(p)\hat\mu_n(p))^{\theta/\alpha}
= O(d(p,q)^{\theta/\alpha})$, we obtain that $d(p',q') \leq O(d(p,q)^{\theta/\alpha}) \leq O(d(p,q)^{\theta})$,
and so $h^s$ is $\theta$-H{\"o}lder continuous.

We next turn to the H{\"o}lder continuity of $h^s_\phi$.  Since $h^s_\phi$ covers $h^s$, it suffices to establish
H{\"o}lder continuity in the $\RR$-fiber. Fix a point $(p,r)\in  \Sigma\times\RR$ and write
$h^s_\phi(p,r) = (p', r')$ and $h^s_\phi(q,s) = (q',s')$.

H{\"o}lder continuity of $\phi$ with exponent $\alpha$ implies that 
\begin{eqnarray*}
|S_n\phi(p) -  S_n\phi(q)| &\leq& \sum_{i=0}^{n-1} O(d(p_i,q_i)^\alpha)\\
&\leq&  \sum_{i=0}^{n-1} O((\nu_n(p)\hat\mu_n(p)\hat\mu_i(p)^{-1})^\alpha)\\
&=& \nu_n(p)^\alpha \sum_{i=0}^{n-1} O (\hat\mu_{-i}(p_n)^{-\alpha})\\
&\leq& \nu_n(p)^\alpha \sum_{i=0}^{n-1} O (\bar \mu^{i\alpha}) = O(\nu_n(p)^\alpha)\\
\end{eqnarray*}
where $\bar \mu < 1$ is an upper bound for $\hat\mu$. This means that 
$|r_n - s_n|  \leq  |r-s| + O(\nu_n(p)^\alpha)$.

Note that $(p_n',r_n') \in \W^s_\phi(p_n,r_n)$. Proposition~\ref{p=liftfol} implies
that $\W^s_\phi(p_n,r_n)$ is the graph of an $\alpha$-H{\"o}lder continuous function
from $\W^s(p_n)$ to $\RR$. Hence
\begin{eqnarray*}
|r_n-r_n'| &\leq & O(d(p_n,p_n')^{\alpha}) = O(\nu_n(p)^\alpha),
\end{eqnarray*}
and similarly,  $|s_n-s_n'| = O(\nu_n(p)^\alpha)$.
Now, by the triangle inequality,
\begin{eqnarray}\label{e=r_n'-s_n'}
|r_n'-s_n'| & \leq & |r_n-s_n|  + |r_n - r_n'| +  |s_n - s_n'|\\
&\leq&  |r-s| + O(\nu_n(p)^\alpha);
\end{eqnarray}

Since $d(p_{n-i}',q_{n-i}') \leq O(\nu_n(p)\rho_{-i}(p_n))$, for $i=1,\ldots n$, the
$\alpha$-H{\"o}lder continuity of $\phi$ implies that 
$|S_n\phi (p')) -  S_n\phi (q')| \leq \sum_{i=1}^{n} O( (\nu_n(p)\rho_{-i}(p_n))^\alpha)
= O( (\nu_n(p)\rho_{n}(p)^{-1})^\alpha) $, since $\rho < 1$.
The inequality $(\nu\rho^{-1})^\alpha < (\nu\hat\mu)^\theta$ now implies that
\begin{eqnarray}\label{e=S_n}
|S_n\phi (p')) -  S_n\phi (q')| &\leq & O((\nu_n(p)\hat\mu_n(p))^\theta).
\end{eqnarray}

Combining (\ref{e=r_n'-s_n'}) and (\ref{e=S_n}), we obtain:
\begin{eqnarray*}
|r'-s'| &=& |(r_n'-s_n') - (S_n \phi (p')) -  S_n \phi(q'))| \\
&\leq & |r-s| + O(\nu_n(p)^\alpha) + O((\nu_n(p)\hat\mu_n(p))^\theta)\\
&\leq& |r-s| + O((\nu_n(p)\hat\mu_n(p))^\theta),
\end{eqnarray*}
since  $\nu^\alpha < (\nu\hat\mu)^\theta$.

We would like to compare $|r'-s'|$ to $d((p,r), (q,s))^\theta$; the latter quantity 
is equal to $(|r-s| + d(p,q))^\theta = (|r-s| + \Theta((\nu_n(p) \hat\mu_{n}(p))^\theta)$;
by the preceding calculation, $|r'-s'| \leq O(d((p,r), (q,s))^\theta)$. Hence
$h^s_\phi$ is $\theta$-H{\"o}lder continuous.
\end{proof}

Having completed this preliminary step, we turn to the proof of the main result in this section.
\begin{proof}[Proof of Theorem~\ref{t=main}, part II]

Suppose that $f$ is accessible and $\phi\colon M\to \RR$ is H{\"o}lder continuous.
Let $\Phi\colon M\to \RR$ be a continuous map satisfying $\phi = \Phi\circ f - \Phi + c$, for some $c\in\RR$.
We show that $\Phi$ is H{\"o}lder continuous.
The key ingredient in the proof is the following lemma.
\begin{lemma}\label{l=holderpaths} There exist $C>0$, $r_0>0$ and $\kappa\in (0,1)$ with the following properties.

For any pair of points $p,q\in M$, there exist functions $\alpha\colon B_M(p,r_0) \to B_M(q,1)$ 
and  $\beta\colon B_M(p,r_0) \to \RR $ with the following
properties:
\begin{enumerate}
\item $\alpha(p) = q$
\item for all $z,z'\in B_M(p,r_0)$,
$$d(\alpha(z), \alpha(z')) \leq C d(z,z')^\kappa,$$
and 
$$|\beta(z) - \beta(z')| \leq C d(z,z')^\kappa,$$
\item for all $z\in B_M(p,r_0)$, $\alpha(z)$ is the endpoint of an $su$-path in $M$ originating at $z$,
\item for all $z\in B_M(p,r_0)$, and $t\in \RR$,
$\Delta(z,t)$ is the endpoint of an $su$-lift path in $M\times \RR$ originating at $(z,t)$,
where $\Delta\colon B_M(p,r_0)\times \RR \to B_M(q,1)\times \RR$ is the map
$\Delta(z,t) = (\alpha(z), t+\beta(z))$.
\end{enumerate}
\end{lemma}

Assuming this lemma, the proof proceeds as follows. Let $C,r_0,\kappa$ be given
by Lemma~\ref{l=holderpaths}.  Fix $x_0,x_1\in M$ with $d(x_0,x_1)<r_0$. For $i\geq 1$, we construct
a sequence of points $x_i$ and maps $\alpha_i \colon B_M(x_0,r_0) \to B_M(x_{i},1)$, 
$\beta_i\colon B_M(x_0,r_0) \to \RR $ and $\Delta_i\colon B_M(x_0,r_0)\times \RR \to B_M(x_{i},1)\times \RR$
inductively as follows.  The point $x_1$ is already defined. 
Assume that $x_i$, for $i\geq 1$ has been defined. 
Let $\alpha_{i}$ and $\beta_i$
be given by the lemma, setting $p=x_0$ and $q=x_{i}$ (so that $h(x_0) = x_i$).
Define $\Delta_i$, as in Lemma~\ref{l=holderpaths},
by $\Delta_i(z,t) = (\alpha_i(z),t+ \beta_{i}(z))$.  We then set $x_{i+1} = \alpha_i(x_1)$.  

We next argue that, for any $i\geq 1$, the map $\Delta_i$ has the property
that, for all $z\in B_M(x_0, r_0)$, $$\Delta_i(z,\Phi(z)) = (\alpha(z), \Phi(z) + \beta_i(z)) = (\alpha(z), \Phi(\alpha(z))).$$
Since $\Phi$ is a continuous solution to (\ref{e=livsic2}), Proposition~\ref{p=satprop}
implies then the graph of $\Phi$ is bisaturated.  That is, for any $p,q\in M$,
if $(q,t)$ is the endpoint of any ${su}$-lift path 
originating at $(p,\Phi(p))$, then $t=\Phi(q)$.
But properties 3 and 4 of the maps $\Delta_i$ given by Lemma~\ref{l=holderpaths} imply
that $\alpha_i(z)$ is the endpoint of an $su$-path originating at $z$, and
$\Delta_i(z,\Phi(z))$ is the endpoint of an ${su}$-lift path originating at $(z,\Phi(z))$.
Hence we obtain that $\Delta_i(z,\Phi(z)) = (\alpha_i(z), \Phi(\alpha_i(z)))$, as claimed.

It now follows from the properties of $\Delta_i$ and the definition of $x_i$ that, for $i\geq 1$:
$$\Phi(x_0) + \beta_i(x_0) = \Phi(\alpha_i(x_0)) = \Phi(x_i),$$
and
$$\Phi(x_1) + \beta_i(x_1) = \Phi(\alpha_i(x_1)) = \Phi(x_{i+1}).$$
Thus:
\begin{eqnarray}\label{e=phidifference}
\Phi(x_1) - \Phi(x_0) &=& \left(\Phi(x_{i+1}) - \Phi(x_i)\right) + \left(\beta_i(x_0) - \beta_i(x_1)\right).
\end{eqnarray}
Summing equation (\ref{e=phidifference}) over $i\in\{1,\ldots,n\}$, we obtain:
\begin{eqnarray*}
n\left(\Phi(x_1) - \Phi(x_0)\right) &=& \left(\Phi(x_{n+1}) - \Phi(x_1)\right) + \sum_{i=1}^n\left(\beta_i(x_0) - \beta_i(x_1)\right),
\end{eqnarray*}
and so:
\begin{eqnarray*}
|\Phi(x_1) - \Phi(x_0)| &\leq& \frac{1}{n}\left|\Phi(x_{n+1}) - \Phi(x_1)\right| + \frac{1}{n}\sum_{i=1}^n\left|\beta_i(x_0) - \beta_i(x_1)\right|\\
&\leq& \frac{1}{n}\|\phi\|_{\infty} + \frac{1}{n}\sum_{i=1}^n C d(x_0, x_1)^{\kappa}\\
&\leq& \frac{1}{n}\|\phi\|_{\infty} + C d(x_0, x_1)^{\kappa}.
\end{eqnarray*}
Sending $n\to \infty$, we obtain that $|\Phi(x_1) - \Phi(x_0)| \leq  C d(x_0, x_1)^{\kappa}$;
since $x_0$ and $x_1$ were arbitrary points within distance $r_0$ of each other, this implies
that $\Phi$ is $\kappa$-H{\"o}lder continuous.  This completes the proof of Proposition~\ref{p=Holder}, assuming
Lemma~\ref{l=holderpaths}. \end{proof}

\begin{proof}[Proof of Lemma~\ref{l=holderpaths}] 
Let $\theta$ be given by Proposition~\ref{p=Holder}, and let $N_M$, $L_M$ be given by 
Lemma~\ref{l=unifaccess}.

We first describe how to construct the maps $\alpha$ and $\beta$ in the case where
$q\in \W^s(p,L_M)$. The analogous construction works for $q\in\W^u(p,L_M)$.
Lemma~\ref{l=unifaccess} implies that any $p$ and $q$ can be connected by an $(K_M,L_M)$-accessible sequence.
We can therefore construct $\alpha,\beta$ for a general pair of points $p$ and $q$ by composing
at most $K_M$ maps along stable and unstable segments.

Suppose then that $p'\in \W^s(p, L_M)$.  We define $\alpha = \alpha_{p,p'}$ as follows.  Fix a foliation
box $U$ of $\cW^s$ containing $\W^s(p, L_M)$, and let $\{\Sigma_x\}_{x\in U}$ be a (uniformly-chosen)
smooth foliation by admissible transversals to $\cW^s$ in $U$.  For $z\in U$, we define $\alpha_{p,p'}(z)$ to
be the unique point of intersection of $\cW^s(z,L_M)$ with $\Sigma_{p'}$ in $U$.
The map $\alpha_{p,p'}\colon U\to \Sigma_{p'}$ sends $p$ to $p'$ and is $\theta$-H{\"o}lder continuous
when restricted to any transversal $\Sigma_x$. Since $\{\Sigma_x\}_{x\in \cW^s(p)}$ is a smooth foliation,
it follows that $\alpha_{p,p'}$ is $\theta$-H{\"o}lder continuous, uniformly in $p'\in U$.

Similarly, for $(z,t)\in U\times \RR$, we define $\Delta_{p,p'}(z,t)$ to be the unique point of intersection of $\cW^s_\phi(z)$ with $\Sigma_{p'}\times\RR$ in $U\times \RR$. 
Proposition~\ref{p=liftfol} implies that $\Delta_{p,p'}$ takes the form
$$\Delta_{p,p'}(z,t) = (\alpha_{p,p'}(z), t+ \beta_{p,p'}(z)),$$
for some function $\beta_{p,p'}\colon U\to \RR$.   Proposition~\ref{p=Holder} implies
that $\Delta_{p,p'}$, and so $\beta_{p,p'}$, is $\theta$-H{\"o}lder continuous, uniformly in $p'\in U$.

The same construction defines $\alpha_{p,p'}$ and $\beta_{p,p'}$ for $p'\in \W^u(p,K_M)$.
Finally, for $p,q$ in $M$, we fix an $(K_M,L_M)$-accessible sequence $(y_0, y_1, \ldots, y_{K_M})$ connecting $p$ and $q$ and define 
$$\alpha_{p,q} = \alpha_{y_{K_M-1}, y_{K_M}} \circ  \alpha_{y_{K_M-2}, y_{K_M-1}}\circ\cdots\circ  \alpha_{y_{0}, y_{1}}.$$  
By construction, $\alpha_{p,q}(p) = q$.
Similarly define $\beta_{p,q}$.

Then there exists $r_0>0$ such that for every pair $p,q$, $\alpha_{p,q}$ and $\beta_{p,q}$ are
defined in the neighborhood $B_M(p,r_0)$ and $\alpha_{p,q}$ takes values in $B_M(q,1)$.  Furthermore, there exists
$C>0$ such that (1) and (2) in the statement of the lemma hold, for $\kappa = \theta^{K_M}$.
Finally, property (4) holds by construction.

\end{proof}

\begin{remark} The H{\"o}lder exponent for $\Phi$ obtained in this proof can be considerably smaller
than the exponent for $\phi$.  In particular, the largest possible exponent for
the $\cW^s_\phi$  or $\cW^u_\phi$ holonomy given by Proposition~\ref{p=Holder} is $\frac{1}{2}$.
Concatenating these holonomies along $K$ steps of an accessible sequence reduces this exponent
further to $\frac{1}{2^K}$.  In contrast, the exponents for $\Phi$ and $\phi$ in
Theorem~\ref{t=livsicold} are the same.  This is because the transverse H{\"o}lder
continuity of $\cW^s_\phi$ and $\cW^u_\phi$ does not play a role in the proof when $f$ is
Anosov, and so only the  H{\"o}lder exponent of the leaves, which is the same as for $\phi$,
determines the exponent for $\Phi$.

\end{remark}

\section{Jets}

In this section we review  basic facts about jets and jet bundles that will
be needed in subsequent sections. The reader is
referred to \cite{Hirsch, Kolar} for a more detailed account.

If $N_1$ and $N_2$ are $C^k$ manifolds and $\ell \leq k$, 
we denote by $\Gamma^\ell(N_1,N_2)$ the set of local $C^k$ maps from
$N_1, N_2$; each element of $\Gamma^\ell(N_1,N_2)$ is a triple
$(p,\phi, U)$, where $\phi$ is a $C^\ell$ map from a neighborhood $U$ of $p$ in $N_1$ to $N_2$.
For $p\in N_1$, we denote by $\Gamma^\ell_p(N_1,N_2)$ the set of
elements of $\Gamma^\ell(N_1,N_2)$ based at $p$.
We denote by
$J^\ell(N_1,N_2)$ the bundle of $C^\ell$ jets from $N_1$
into $N_2$: each element of $J^\ell(N_1,N_2)$ is an equivalence class of triples $(p,\phi,U)\in \Gamma^\ell_p(N_1,N_2)$,
where two triples  $(p,\phi,U)$ and  $(p', \phi', U')$ are equivalent if $p=p'$, and the partials of $\phi$ and $\phi'$ at $p$ up to order $\ell$ coincide.  

We denote by $[p,\phi,U]_\ell$ the equivalence
class containing $(p,\phi,U)$, which is called a {\em $\ell$-jet at $p$}.
Alternately, we use the notation $j^\ell_p\phi$.
The point $p$ is called the {\em source} of $(p,\phi,U)$
and $\phi(p)$ is the {\em target}. The source map $\sigma$ gives $J^\ell(N_1,N_2)$ the structure
of a $C^{k-\ell}$ bundle over $N_1$; we denote by $J^\ell_p(N_1,N_2)$ the $\ell$-jets with source 
$p\in N_1$.
We also denote by $J^\ell(N_1,N_2)_q$ the set of jets with target $q$.  

More generally one has the $\ell$-jet bundle associated to a fiber bundle.
If $\pi\colon \cB \to M$ is a $C^k$ fiber bundle, and $\ell\leq k$, 
we denote by $\Gamma^\ell(\pi\colon \cB \to M)$ the set of $C^\ell$ local sections
of $\cB$, and by $\Gamma^\ell_p(\pi\colon \cB \to M)$ the set of $C^\ell$ local sections whose 
domain contains $p\in M$. We then define the
$\ell$-jet bundle $J^\ell(\pi\colon \cB\to M)$ to be 
the set of pairs $(p,\phi)$, where $\phi\in \Gamma^r_p(\pi\colon \cB \to M)$, 
and two pairs  $(p,\phi)$ and  $(p', \phi')$ are 
equivalent if $p=p'$, and the partials of $\phi$ and $\phi'$ at $p$ up to 
order $\ell$ coincide. Then $J^\ell(\pi\colon M\to \cB)$ is a $C^{k-\ell}$ bundle over
$M$. Observe that $J^\ell(N_1,N_2) = J^\ell(\proj_{N_1}\colon N_1\times N_2 \to N_1)$
under the natural identification of sections of $N_1\times N_2$ with functions $\phi:N_1\to N_2$.

For $\ell'\leq \ell$, there is a natural projection $\pi_{\ell,\ell'}$
from the $\ell$-jet bundle to the $\ell'$-jet bundle
that sends $j^\ell_p\phi$ to $j^{\ell'}_p\phi$.   Under this projection,
$J^\ell$ has the structure of a $C^{k-\ell'}$ fiber bundle over $J^{\ell'}$.
Moreover, $J^{\ell-\ell'}(J^{\ell'}) = J^{\ell}$. 

The bundle
$J^\ell(\RR^m,\RR^n)$ is a trivial bundle over $\RR^m$.
The fiber space $J^\ell_v(\RR^m,\RR^n)$ is the $\ell+1$-fold 
product $P^\ell(m,n) = \Pi_{i=0}^\ell L^i_{sym}(\RR^m,\RR^n)$, where $L^i_{sym}(\RR^m,\RR^n)$ is the vector space of
of symmetric, $i$-multilinear maps from $\RR^m$ to $\RR^n$.   Each $\ell$-jet $[v,\phi,U]_\ell$ in
$J^\ell_v(\RR^m,\RR^n)$ has a canonical representative, which is the $\ell$th order Taylor polynomial 
of $\phi$ about $v$. To denote an element of $J^\ell(\RR^m,\RR^n)$, we sometimes use the notation 
$(v,\wp)$ with $v\in \RR^m$ and $\wp$ a degree $\ell$ polynomial (suppressing the neighborhood $U$,
since polynomials are globally defined).  These give $C^\infty$ global coordinates on $J^\ell(\RR^m,\RR^n)$;
in this way we regard $J^\ell(\RR^m,\RR^n)$ as a finite dimensional vector space with a Euclidean structure $|\cdot|$.

\subsection{Prolongations}

If $\phi:N_1\to N_2$ is a $C^\ell$ function, then $\phi$ gives rise to a section
of the bundle $J^\ell(N_1,N_2)$ over $N_1$ via the map $v\mapsto j^\ell_v\phi$.  This
section, denoted $j^\ell\phi$ is called the {\em $\ell$-prolongation} of $\phi$.  In the case
$\ell=0$, the jet bundle $J^0(N_1,N_2)$ is just the product $N_1\times N_2$, and
the image of $N_1$ under the prolongation $j^0\phi$ is the just the graph of $\phi$.  

The function $\phi:M\to M$ is $C^k$ if and only if the $\ell$-prolongation of $\phi$ is $C^{k-\ell}$.  
Not every continuous section of $J^\ell(M,N)$ is the prolongation of a $C^\ell$ function; however,
the set of prolongations of smooth functions is  closed:
\begin{proposition}\label{p=prolongation} If $f_n\in C^\ell(M,N)$ and $j^\ell f_n\to j^\ell f$ in the weak topology
on $C^0(M, J^\ell(M,N))$, then $f\in C^\ell(M,N)$.
\end{proposition}

More generally, if $\sigma\colon M\to \cB$ is a section (resp. local section) of a $C^k$ bundle $\pi\colon \cB\to M$,
then the $\ell$-prolongation $j^\ell\sigma\colon M\to J^\ell(\pi\colon M\to \cB)$  is a $C^{k-\ell}$ section
(resp. local section). The analogue of Proposition~\ref{p=prolongation} holds for prolongations
of sections.

\subsection{Isomorphism of jet bundles}

The next lemma is used extensively in various forms in this paper.

\begin{lemma}\label{l=smoothops} Let $N_1, N_2,$ and $N_3$ be $C^k$ manifolds.
\begin{enumerate}
\item Let $g:N_2\to N_3$ be a $C^k$ map.  Then for every $\ell\leq k$, the map $j_x^\ell\phi\mapsto j_{x}^\ell(g\circ \phi)$
is a $C^{k-\ell}$ map from $J^\ell(N_1, N_2)$ to $J^\ell(N_1, N_3)$.
\item Let $h\colon N_1\to N_2$ be a $C^k$ diffeomorphism.
Then for every $\ell\leq k$, the map $j_x^\ell\phi\mapsto j_{h(x)}^\ell(\phi\circ h^{-1})$ is a $C^{k-\ell}$ diffeomorphism
from $J^\ell(N_1,N_3)$ to  $J^\ell(N_2,N_3)$.
\end{enumerate}
\end{lemma}
\begin{remark} There is some subtlety in item 2.  If  $h\colon N\to N$ is  a $C^{k-\ell}$ diffeomorphism
other than the identity,
then neither of the following maps is even differentiable on $J^\ell(N,N)$:
$$j_x^\ell\phi\mapsto j_{h(x)}^\ell\phi \qquad \hbox{or}\quad j_x^\ell\phi\mapsto j_{x}^\ell(\phi\circ h^{-1}).$$
It is at first glance a fortuitous fact that the composition of these maps is $C^{k-\ell}$. What
item 2 expresses is the fact that the $\ell$-jet bundle is a $C^{k-\ell}$ invariant under
$C^k$-diffeomorphisms. More generally:

\begin{corollary}\label{c=bundleiso}(see, e.g. \cite{Kolar}, Chapter 14.4) 
If $\pi\colon \cB\to M$ and $\pi'\colon \cB'\to M'$ are $C^k$ fiber bundles,
and $H\colon \cB\to \cB'$ is a $C^k$ isomorphism of fiber bundles, covering
the $C^k$ diffeomorphism $h\colon M\to M'$, then for every $\ell\leq k$ there is a 
canonical $C^{k-\ell}$ isomorphism of fiber bundles 
$$H^\ell\colon J^\ell(\pi\colon \cB\to M)\to J^\ell(\pi'\colon \cB'\to M')$$ covering $h$. 
For $\ell'\leq \ell$, the map $H^\ell$ covers $H^{\ell'}$ under the natural projection.

The map $H^\ell$ is defined by:
$$H^\ell(j^\ell_x\sigma) = j_{h(x)}^\ell(H\circ\sigma \circ h^{-1}).$$
\end{corollary} 

\end{remark}

\subsection{The graph transform on jets}\label{s=jets}

In its local form, Corollary~\ref{c=bundleiso} tells us that for diffeomorphisms
of $\RR^m\times\RR^n$
of the form $H(x,y) = (h(x),g(x,y))$, the induced {\em graph transform}
on functions $\Phi\colon \RR^m\to \RR^n$ produces a map that is smooth on the level of jets.
By graph transform, we mean the map $\cT_H\colon \{\Phi\colon \RR^m\to \RR^n\}\to \{\Phi\colon \RR^m\to \RR^n\}$
defined by:
$$\cT_H(\Phi)(x) = g(h^{-1}(x),\Phi(h^{-1}(x))).$$
It is easy to see that if $H$ is $C^k$,
then $\cT_H(C^\ell(\RR^m,\RR^n)) = C^\ell(\RR^m,\RR^n)$,
for all $\ell\leq k$;
nonetheless, the restriction of 
$\cT_H$ to $C^\ell(\RR^m,\RR^n)$ is not smooth at all, even for $\ell=0$.
What is smooth, however, is the induced map $H^\ell\colon J^\ell(\RR^m,\RR^n) \to J^\ell(\RR^m,\RR^n)$:
$$H^\ell(j_x^\ell\psi) = j_{h(x)}^\ell(\cT_H(\psi)).$$
This map on $\ell$-jets is $C^{k-\ell}$.

More generally, whenever a graph transform is well-defined, it induces a 
continuous map on jets, which we now describe.
Suppose that $H(x,y) = (h(x,y),g(x,y))$ is a $C^k$ local diffeomorphism of $\RR^m\times\RR^n$.
Write 
$$D_v H =
\left(\begin{array}{cc}
A_v & B_v\\
C_v & K_v
\end{array}\right),
$$
where $A_v\colon \RR^m\to\RR^m$, $B_v\colon \RR^n\to\RR^m$,
$C_v\colon \RR^m\to\RR^n$ and $K_v\colon \RR^n\to\RR^n$.
Suppose that there exists $\rho_0>0$ such that for all $v\in B_{\RR^{m+n}}(0,\rho_0)$,
the map $A_v$ is invertible.

Then there exists $\rho_1>0$ such that, 
for every $\ell\leq k$, there exists a $C^{k-\ell}$ local diffeomorphism
$$H^\ell \colon J^\ell(\RR^m,\RR^n) \to J^\ell(\RR^m,\RR^n),$$
defined in the $\rho_1$-neighborhood of the $0$-section of  
$J^\ell_{B_{\RR^m}(0,\rho_0)}(\RR^m,\RR^n)$, given by:
$$H^\ell(j^\ell_x\psi) =  j^\ell_{h(x,\psi(x))} \left((g\circ (id, \psi)) \circ (h \circ (id,\psi))^{-1}\right).$$
The map $H^\ell$ has the defining property that for every $\psi\in \Gamma^\ell(\RR^m,\RR^n)$,
if $j_x^\ell\psi$ is in the domain of $H^\ell$, and $\psi'\in \Gamma^\ell(\RR^m,\RR^n)$
satisfies:
$$\hbox{graph}(\psi') = H(\hbox{graph}(\psi))
$$
in a neighborhood of $h(x,\psi(x))$, then $H^\ell(j^\ell_x\psi) = j_{h(x,\psi(x))}^\ell\psi'$.
This fact motivates the term ``graph transform.''

We explore the properties of these maps in more detail; this will be used in subsequent sections.
Writing $P^\ell(m,n) = \Pi_{i=0}^\ell 
 L^i_{sym}(\RR^m,\RR^n)$, we have coordinates
$$(x,\wp) \mapsto (x,\wp_0,\ldots, \wp_\ell)$$ on $\RR^m\times P^\ell(m,n)$, where
$\wp_i = D^i_x\wp \in  L^i_{sym}(\RR^m,\RR^n)$. 
Denote by $H^\ell(x,\wp)_i$
the $L^i_{sym}(\RR^m,\RR^n)$-coordinate of $H^\ell(x,\wp)$, so that
$$H^\ell(x,\wp) = (h(x,\wp_0), H^\ell(x,\wp)_0,\ldots, H^\ell(x,\wp)_\ell).$$
Clearly $H^0(x,\wp_0)_0  = g(x,\wp_0)$.  Because jets are natural, for $\ell'\leq \ell$, we have 
$$H^{\ell}(x,\wp_0,\ldots,\wp_\ell)_{\ell'} = H^{\ell'}(x,\wp_0,\ldots,\wp_{\ell'})_{\ell'}.$$
Furthermore,
$$H^1(x,\wp_0, \wp_1)_1 = \left(C_{(x,\wp_0)}+K_{(x,\wp_0)}\wp_1   \right) \left(A_{(x,\wp_0)} + B_{(x,\wp_0)}\wp_1 \right)^{-1}.$$
Differentiating this expression $\ell$ times (implicitly), we get, for $\ell\geq 1$:
\begin{eqnarray*}
H^{\ell}(x,\wp_0,\ldots,\wp_\ell)_{\ell}
& = &(K_{(x,\wp_0)} \wp_\ell - H^1(v,\wp_0, \wp_1)_1 B_{(x,\wp_0)} \wp_\ell \\
&& + S^\ell(x,\wp_0,\ldots ,\wp_{\ell-1})) \circ (A_{(x,\wp_0)}  +  B_{(x,\wp_0)}\wp_1 )^{-1},
\end{eqnarray*}
where $S^\ell$ is a polynomial in $(x,\wp_0,\ldots ,\wp_{\ell-1})$ and in the
partial derivatives of $H$ at $(x,\wp_0)$ up to order $\ell$. 

Notice that if $B_{(x,\wp_0)}=0$, then these expressions reduce to:
\begin{eqnarray*}
H^{\ell}(x,\wp_0,\ldots,\wp_\ell)_{\ell}
& = &(K_{(x,\wp_0)} \wp_\ell + S^\ell(x,\wp_0,\ldots ,\wp_{\ell-1})) \circ A_{(x,\wp_0)}^{-1}.
\end{eqnarray*}
In particular, if 
$B_{(x,\wp_0)}=0$, then there exists $\rho_2>0$ such that
for all $(x',\wp')$ lying in the $\rho_2$-neighborhood of $(x,\wp)$ in
$J^\ell(\RR^m,\RR^n)$, we have:
\begin{eqnarray}\label{e=skewnice}
|H^{\ell}(x,\wp)_{\ell} - H^{\ell}(x',\wp')_{\ell}|\qquad\qquad\qquad\qquad \qquad\qquad\qquad\qquad\qquad \\
\qquad\quad\leq Q^\ell_{(x,\wp_0)}(\wp_\ell - \wp_\ell') + O\left(|(x,\wp_0,\ldots,\wp_{\ell-1})- (x',\wp_0',\ldots,\wp_{\ell-1}')|\right),
\end{eqnarray}
where $Q^\ell_{(x,\wp_0)}\colon L^\ell_{sym}(\RR^m,\RR^n) \to L^\ell_{sym}(\RR^m,\RR^n)$ is the linear map:
$$Q^\ell_{(x,\wp_0)} (\overline\wp_\ell) = K_{(x,\wp_0)} \circ \overline\wp_\ell  \circ A_{(x,\wp_0)}^{-1}.$$
Observe that, because $\overline\wp_\ell$ is a symmetric map of order $\ell$, we have $\|Q^\ell_{(x,\wp_0)}\| \leq \|K_{(x,\wp_0)}\|/m(A_{(x,\wp_0)})^\ell$, where
$m(X) = \|X^{-1}\|^{-1}$ denotes the conorm of an invertible matrix $X$.

For $\ell\geq 1$, we may regard $J^\ell(\RR^m,\RR^n)$ as a 
vector bundle over $J^0(\RR^m,\RR^n)$ ($ = \RR^m\times\RR^n$) under the natural projection $\pi_{\ell,0}$;
the fiber is $\Pi_{i=1}^\ell  L^i_{sym}(\RR^m,\RR^n)$.
In a variety of contexts (see Section~\ref{ss=fake} ff.) we will consider the case where the map $H^\ell$ is
a fiberwise contraction on a neighborhood of the $0$-section of this bundle.
We assume that $\|K_{(x,\wp_0)}\| <  m(A_{(x,\wp_0)})$ and $\|K_{(x,\wp_0)}\| <  m(A_{(x,\wp_0)})^\ell$
(which together imply that  $\|K_{(x,\wp_0)}\| <  m(A_{(x,\wp_0)})^i$, for $1\leq i\leq \ell$).

Continuing to assume that $B_{(x,\wp_0)}=0$, we next construct in the standard way a norm $|\cdot|'$ 
on  $\Pi_{i=1}^\ell  L^i_{sym}(\RR^m,\RR^n)$  such that:
\begin{eqnarray}\label{e=lipbound}
|H^{\ell}(x,\wp) - H^{\ell}(x,\wp') |'  \qquad \qquad\qquad\qquad \qquad\qquad\qquad\qquad\qquad\quad\\
\qquad \qquad\qquad\qquad\leq \max\left\{ \frac{\|A_{(x,\wp_0)}\|}{m(K_{(x,\wp_0)})}, \frac{\|K_{(x,\wp_0)}\|}{m(A_{(x,\wp_0)})^\ell}\right\}\cdot
|(x,\wp)_{\ell} - (x,\wp')_{\ell}|',
\end{eqnarray}
for $(x,\wp),(x,\wp')$ lying in the set $\{(x,\wp_0, \overline\wp_1,\ldots,\overline\wp_\ell)\colon
| (\overline\wp_1,\ldots,\overline\wp_\ell) |' \leq  1\}$.
To do this,  fix $L>0$ and for $( \overline\wp_1,\ldots,\overline\wp_\ell)\in \Pi_{i=1}^\ell  
L^i_{sym}(\RR^m,\RR^n)$, define:
$$|( \overline\wp_1,\ldots,\overline\wp_\ell)|_L =  L^{\ell}|\overline\wp_1|+ \cdots + L|\overline\wp_\ell|.$$
It is not difficult to verify using (\ref{e=skewnice}) that if $L>0$ is sufficiently large, then (\ref{e=lipbound}) holds for $|\cdot|' = |\cdot|_L$ and all  $(x,\wp),(x,\wp')$ lying in the set  $\{(x,\wp_0, \overline\wp_1,\ldots,\overline\wp_\ell)\colon | (\overline\wp_1,\ldots,\overline\wp_\ell) |' \leq  1\}$.

The same holds true if $\|B_{(x,\wp_0)}\|$ is sufficiently small. Summarizing this
discussion, we have:
\begin{lemma}\label{l=graphjets} Fix $\ell\geq 1$.
For every $R>0$ and $\kappa\in (0,1)$ there exist $\eps>0$ and $L>0$ with the following
properties.

Let $H\colon B_{\RR^{m+n}}(0,1)\to \RR^{m+n}$ be a $C^\ell$ local diffeomorphism
such that:
\begin{itemize}
\item  $d_{C^\ell}(H,Id)\leq R$, and
\item  writing $D_v H =
\left(\begin{array}{cc}
A_v & B_v\\
C_v & K_v
\end{array}\right)
$, we have:
$$\inf_{v\in B_{\RR^{m+n}}(0,1)}m(A_v) > 0,$$
$$\kappa > \sup_{v\in B_{\RR^{m+n}}(0,1)}\max \left\{ \frac{\|K_{v}\|}{m(A_{v})}, \frac{\|K_{v}\|}{m(A_{v})^\ell}\right\},$$
and
$$\sup_{v\in B_{\RR^{m+n}}(0,1)} \|B_v\| < \eps.$$
\end{itemize}

Then for all $v = (v^m,v^n)\in\RR^{m+n}$ and all $j_{v^m}^\ell\psi, j_{v^m}^\ell\psi'\in \pi_{\ell,0}^{-1}(v)$, with $|j_{v^m}^\ell\psi|,|j^\ell_{v^m}\psi'|\leq 1$, we have:
$$| H^\ell(j^\ell_{v^m}\psi) - H^\ell(j^\ell_{v^m}\psi') |_L \leq \kappa |j^\ell_{v^m}\psi-j^\ell_{v^m}\psi'|_L.$$
\end{lemma}

\section{Proof of Theorem~\ref{t.Crsubmanif}}

Before proving our main higher regularity result (part IV of Theorem~\ref{t=main}), we give a 
proof of Theorem~\ref{t.Crsubmanif}, as the proof conveys some of the
basic techniques we will use later, but in a simpler setting.

Suppose that $N$ is an embedded $C^1$ submanifold of $\RR^{m+n}$ such that, for every $x,y$ in $N$,
there exist neighborhoods $U$ of $x$ and $V$ of $y$ and a $C^k$ diffeomorphism
$H:U\to V$ such that $H(U) = V$ and $H(U\cap N) = V\cap N$, where $k\geq 2$.  

We prove that $N$ is a $C^\ell$ submanifold of $\RR^{m+n}$, for all $\ell\leq k$, by induction on $\ell$. 
By assumption, $N$ is a $C^1$ submanifold. Suppose that $N$ is a $C^\ell$ submanifold, for some $\ell\leq k-1$.
We prove that $N$ is $C^{\ell+1}$ submanifold. As the problem is local, we may restrict attention to a small neighborhood in $N$.

Fix a point $x_0\in N$ and a neighborhood $V$ of $x_0$ in $N$. 
By a local $C^{k}$ change of coordinates in $N'$ sending $x_0$ to $0\in \RR^n\times \RR^m$,
we may assume that $N$ is the graph of a $C^\ell$ function $\Phi: \overline{B_{\RR^n}(0,1)}\to \RR^m$ satisfying
$j^k_0\Phi = 0$.  The first main step in the proof of Theorem~\ref{t.Crsubmanif} is the following lemma.

\begin{lemma}\label{l=smoothgraph} For every $u\in \overline{B_{\RR^n}(0,1)}$ there exists $\rho=\rho(u)>0$, and
for every $i\in\{0,\ldots, \ell\}$,  a $C^{k-i}$ local diffeomorphism
$$H_u^i\colon B_{J^i(\RR^n,\RR^m)}(0,\rho) \to J^i(\RR^n,\RR^m)$$
with the following properties:
\begin{enumerate}
\item  $H_u^{i}$ covers $H_u^{i-1}$ under the projection $J^i(\RR^n,\RR^m)\to J^{i-1}(\RR^n,\RR^m)$,
and
\item  writing $H_u^0(v,w) = (h_u(v,w), g_u(v,w))$, we have $h_u(0,\Phi(0)) = u$, and:
$$ H_u^\ell(j^\ell_v\Phi) = j_{h_u(v,\Phi(v))}^\ell \Phi,$$ 
for every $v$ such that $j^\ell_v\Phi\in  B_{J^\ell(\RR^n,\RR^m)}(0,\rho)$.
\end{enumerate}
\end{lemma}
\begin{proof}
For $i=0$, this follows immediately from $C^k$ homogeneity.
Given $u\in \overline{B_{\RR^n}(0,1)}$, select a $C^k$ local diffeomorphism 
$$H_u = (h_u, g_u)\colon B_{\RR^n\times\RR^m}(0,\rho_0)\to \RR^n\times\RR^m$$
sending $(0,0) = (0,\Phi(0))$ to 
$(u,\Phi(u))$ and preserving the graph of $\Phi$. 
Under the natural identification of $J^0(\RR^n,\RR^m)$ with
$\RR^n\times\RR^m$, this defines the map $H^0_u$:
$$H^0_u(v,w) = (h_u(v,w), g_u(v,w)).$$

Suppose $i\geq 1$, and fix a point $v'\in \RR^n$ near $0$,
and a function $\psi\in \Gamma^i_{v'}(\RR^n,\RR^m)$.
Consider the local map $h_u\circ(id, \psi)\in  \Gamma^{i}_{v'}(\RR^n,\RR^n)$ given by:
$$H_u\circ (id,\psi)(v) = h_u(v,\psi(v)).$$
Its derivative at $v'$ is 
\begin{eqnarray}\label{e=diffgraph}
D_{v'}\left(h_u\circ (id,\psi)\right) = \frac{\partial h_u}{\partial v}(v',\psi(v')) +  \frac{\partial h_u}{\partial w}(v',\psi(v'))D_{v'}\psi.
\end{eqnarray}
Since $DH^0_u$ preserves the tangent space to the graph of $\Phi$,
it follows that the map  ${\partial H_u}/{\partial v}(0, 0)$ is a diffeomorphism onto a neighborhood of $u$.
On the other hand, plugging in ${v'}=0$, $D_{v'}\psi = 0$ into equation (\ref{e=diffgraph}) we obtain that 
for any $\psi\in \Gamma^i_{0}(\RR^n,\RR^m)$ with $j^1_0\psi = 0$,
 $D_0\left(h_u\circ (id,\psi)\right) = \frac{\partial h_u}{\partial v}(0,0)$. 
 
Since $H^0$ is $C^1$, from this it follows that for $|j^i_{v'}\psi|$ and $|v'|$ sufficiently small,
the derivative $D_{v'}\left(h_u\circ (id,\psi)\right)$
is invertible. The inverse function theorem then implies that $h_u\circ (id,\psi)$
is a $C^i$ local diffeomorphism in a neighborhood of $v\in\RR^n$, 
provided $|j_v^i\psi|$ is sufficiently small; in particular, $(h_u\circ (id,\psi))^{-1}$
is defined.

For $i\geq 1$, we then set
$$H^i_u(j^i_v\psi) =  j^i_{h_u(v,\psi(v))} \left((g_u\circ (id, \psi)) \circ (h_u\circ (id,\psi))^{-1}\right).$$
Lemma~\ref{l=smoothops} implies that $H^i_u$ is a $C^{k-i}$ local diffeomorphism.  By construction,
the maps $H_u^i$ satisfy properties (1) and (2).
\end{proof}

\begin{remark} Notice that Lemma~\ref{l=smoothgraph} implies that the image of 
$\overline{B_{\RR^n}(0,1)}$ under $j^\ell\Phi$ is a $C^1$ homogeneous submanifold of $J^\ell(\RR^n,\RR^m)$. 
At this point, it is tempting to appeal to Theorem~\ref{t=rss} to finish the proof.  If one does so, one obtains that  
$j^\ell\Phi(\overline{B_{\RR^n}(0,1)})$ is a $C^1$ submanifold of $J^\ell(\RR^n,\RR^m)$.  
{\em However, this alone does not imply that $\Phi$ is $C^{\ell+1}$.}   To conclude
that $\Phi$ is $C^{\ell+1}$, for $\ell\geq 1$, it is in fact necessary
to show that the {\em function} $j^\ell\Phi$, and not just its image, is  $C^1$.

Here is an example that illustrates what can go wrong. Let $\Psi\colon\RR\to\RR$ be any $C^2$ function, and
let $\Phi(x) = \Psi(x) + x^{4/3}$.  Then the graph of $\Phi$ is a $C^1$ submanifold of $\RR^2$, but
not a $C^2$ submanifold, as is easily checked.  On the other hand, $j^1_x\Phi = (x,\Psi(x) + x^{4/3},  \Psi'(x)+ (4/3) x^{1/3})$, and the image of $j^1\Phi$ is the image of the following $C^1$ embedding:
$$\psi\colon \RR\to \RR^3;\quad \psi(t) = (t^3, \Psi(t^3) + t^4, \Psi'(t^3) + (4/3) t).$$
Hence the image of $j^1\Phi$ is $C^1$, but the graph of $\Phi$ is not $C^2$.

The situation can be arbitrarily bad: if $\Psi$ is real analytic, then 
the image of $j^1\Phi$ is analytic, while the graph of $\Phi$ still fails to be $C^2$.
(The construction we have described is a very simple example of the procedure in algebraic geometry
of resolving a singularity.) 
\end{remark}

\bigskip

Returning to the proof of Theorem~\ref{t.Crsubmanif}, our next step is to show:

\bigskip

{\em If $\Phi$ is $C^\ell$ and $j^\ell\Phi$ is a $C^1$-homogeneous function (in the
sense of} \\
{\em Lemma~\ref{l=smoothgraph}), then $j^\ell\Phi$ is  $C^1$, and so $\Phi$ is $C^{\ell+1}$.}

\bigskip

To this end, let $A\colon J^\ell(\RR^n,\RR^m)\to J^\ell(\RR^n,\RR^m)$ 
be an invertible linear transformation, and let $\rho>0$.
We next define a subset $\cG(A,\rho)\subset \overline{B_{\RR^n}(0,1)}$ consisting
of the set of all $u\in  \overline{B_{\RR^n}(0,1)}$ with the following properties:
\begin{itemize}
\item For each $i\in\{0,\ldots \ell\}$, there exists a bilipschitz embedding
$$\tilde H_u^i\colon B_{J^i(\RR^n,\RR^m)}(0,\rho) \to J^i(\RR^n,\RR^m)$$ such that:
\item  $\tilde H_u^{i}$ covers $\tilde H_u^{i-1}$ under the projection $J^i(\RR^n,\RR^m)\to J^{i-1}(\RR^n,\RR^m)$,
\item  writing $\tilde H_u^0(v,w) = (\tilde h_u(v,w), \tilde g_u(v,w))$, we have $\tilde h_u(0,\Phi(0)) = u$, and:
$$\tilde H_u^\ell(j^\ell_v\Phi) = j_{\tilde h_u(v,\Phi(v))} \Phi,$$ for every $v$ such that $j^\ell_v\Phi\in  B_{J^\ell(\RR^n,\RR^m)}(0,\rho)$, and
\item $\hbox{Lip}(A - \tilde H^\ell_u) \leq \frac{m(A)}{5}$ on $B_{J^\ell(\RR^n,\RR^m)}(0,\rho)$,
where $m(A) = \|A^{-1}\|^{-1}$ denotes the conorm of $A$.
\end{itemize}

Fix a countable dense subset $\{A_j\}_{j\in\ZZ_+}\subset GL(J^\ell(\RR^n,\RR^m))$ of invertible linear transformations.

\begin{lemma} For each $A\in GL(J^\ell(\RR^m,\RR^n))$, and $\rho>0$, the set $\cG(A,\rho)$ is
compact in $\overline{B_{\RR^n}(0,1)}$.  Moreover: 
$$\overline{B_{\RR^n}(0,1)} = \bigcup_{j_1,j_2\in\ZZ_+} \cG(A_{j_1},j_2^{-1}).$$
\end{lemma}

\begin{proof} 
Suppose that $\cG(A,\rho)$ is nonempty.  Let $u_j$ be a sequence in $\cG(A,\rho)$,
and for each $i\in\{0,\ldots,\ell\}$, let $\tilde H^i_{u_j}$ be the associated
sequence of bilipschitz embeddings.  Since the space of bilipschitz embeddings
is locally compact in the uniform topology, there exists a convergent subsequence 
$u_{j_\ell}\to u \in \overline{B_{\RR^n}(0,1)}$ with 
$\tilde H^i_{u_{j_\ell}}\to \tilde H^i_{u}$ uniformly for all $i$.
The maps $\tilde H^i_{u}$ are bilipschitz embeddings, with $\tilde H^i_{u}$ covering
$\tilde H^{i-1}_u$, and $\hbox{Lip}(H^i_{u}-A)\leq \frac{m(A)}{5}$.  Since the $\ell$-jet
$j^\ell\Phi$ is a closed subset of $J^\ell(\RR^n,\RR^m)$, the limiting
map $\tilde H^\ell_u$ preserves $j^\ell\Phi$.
Hence $u\in \cG(A,\rho)$, and so $\cG(A,\rho)$ is compact.

Lemma~\ref{l=smoothgraph} implies that for each $u$, and each $i$ there exists
a $C^{r-i}$ diffeomorphism $H^i_u$ satisfying the first two properties.
Let $\eps = m(D_{0} H^\ell_u)/11$.
Fix $A_{j_1}\in GL(J^\ell(\RR^n,\RR^m))$ such that $\|D_{0} H^\ell_u - A_{j_1}\|< \eps$.
A simple estimate shows that  $\|D_{0} H^\ell_u - A_{j_1}\| < \frac{m(A_{j_1})}{10}$.
Next, fix  $j_2$ such that $\hbox{Lip}(D_{0} H^\ell_u - H^\ell_u) < \frac{m(A_{j_1})}{10}$ on
$B_{J^\ell(\RR^m,\RR^n)}(0,j_2^{-1})$.  Then
$\hbox{Lip}(A_{j_1} - H^\ell_u) < \frac{m(A_{j_1})}{5}$ on $B_{j^\ell(\RR^n,\RR^m)}(0, j_2^{-1})$, which implies 
that $u\in \cG(A_{j_1},j_2^{-1})$.  Hence:
$$\overline{B_{\RR^n}(0,1)} = \bigcup_{j_1,j_2\in\ZZ_+} \cG(A_{j_1},j_2^{-1}),$$
completing the proof of the lemma.
\end{proof}

Since $\overline{B_{\RR^n}(0,1)}$ is a Baire space, there exist integers $j_1, j_2$ such that
$\cG(A_{j_1},{j_2}^{-1})$ has nonempty interior.  Let $U$ be an open ball contained
in $\cG(A_{j_1},j_2^{-1})$.  For each pair $u, u'\in U$ and $i\in\{0,\ldots,\ell\}$,
we set ${H}^i_{(u,u')} = \tilde {H_{u'}^i}\circ \tilde {H^i_{u}}^{-1}$,
which is defined on a neighborhood of $j_u^i\Phi$ in $J^i(\RR^n,\RR^m)$.
We thus obtain:
\begin{lemma} There exists $\rho>0$ such that, for every pair $z = (u,u')\in U\times U$, the following hold:
\begin{itemize}
\item  for each $i\in\{0,\ldots \ell\}$,  ${H}^i_{z}$ is a bilipschitz
homeomorphism, defined on a $\rho$-neighborhood of $j^i_u\Phi$,
\item  $H_z^{i}$ covers $H_z^{i-1}$ under the projection $J^i(\RR^n,\RR^m)\to J^{i-1}(\RR^n,\RR^m)$,
\item  writing $H_{z}^0(v,w) = (h_{z}(v,w), g_{z}(v,w))$, we have $h_{z}(u,\Phi(u)) = u'$, and:
$$ H_{z}^\ell(j^\ell_v\Phi) = j_{h_{z}(v,\Phi(v))} \Phi,$$ for every $v$ such that $j^\ell_v\Phi\in  B_{J^\ell(\RR^n,\RR^m)}(j_u\Phi,\rho)$, and
\item $\hbox{Lip}(I - H^\ell_{z}) \leq \frac{1}{2}$ on $B_{J^\ell(\RR^n,\RR^m)}(j^\ell_u\Phi,\rho)$.
\end{itemize}
\end{lemma}

Let $K=3/2$, which is a bound, over all $z=(u,u')\in U\times U$, for the Lipschitz norm of $H^\ell_{z}$ on
$B_{J^\ell(\RR^n,\RR^m)}(j^\ell_u\Phi,\rho)$. Since $\Phi$ is assumed to be at least 
$C^1$, there exists a constant $C>0$ such that,
for all $u,u'\in U$,
$$|j^0_u\Phi - j^0_{u'}\Phi| \leq C|u-u'|.$$ 
Fix a point $u_0\in U$, and
let $\alpha = d(u_0, \RR^n\setminus U)$ (which depends uniformly on $u_0$). 
Since $j^\ell\Phi$ is continuous, if $u$ is sufficiently close to $u_0$ (uniformly in $u_0$), 
we will have $j^\ell_{u}\Phi\in B_{J^{\ell}(\RR^m,\RR^n)}(j^{\ell}_{u_0}\Phi,\rho)$. 

Let $u_1\in U$ be such a point. 
Fix $N\in \ZZ_+$ such that:
$$\frac{\alpha}{CK(N+1)}\leq |u_1-u_0| < \frac{\alpha}{CKN}.$$

We construct a sequence of points $u_0, u_1, u_2,\ldots, u_N$ in $U$ inductively as follows. The
points $u_0$ and $u_1$ have already been defined.
For $i\in\{1,\ldots,n-1\}$, we set $z_i = (u_0,u_i)\in U\times U$ and  
$u_{i+1} = h_{z_i}(u_1,\Phi(u_1))$. We need to check that if $u_i$ is contained
in $U$, then $u_{i+1}$ is also contained in $U$.

To see this, note that, for $i\leq N$, we have:
\begin{eqnarray*}
|u_{i} - u_{i-1}| &=& |h_{z_i}(u_1,\Phi(u_1)) - h_{z_i}(u_0,\Phi(u_0))|\\ 
&\leq& K |j^{0}_{u_1}\Phi - j^{0}_{u_0}\Phi|\\
&\leq & KC |u_1 - u_0| 
\end{eqnarray*}
Hence, for $i\leq N$, this implies that 
$|u_i-u_0| \leq K Ci |u_1 - u_0|< \alpha$, so
that $u_i\in U$, for all $i\in\{1,\ldots,N\}$.

Then, for each $i$:
\begin{eqnarray*}
j^\ell_{u_i}\Phi - j^\ell_{u_{i-1}}\Phi &=& H^\ell_{z_i}(j_{u_1}^\ell\Phi) - H^\ell_{z_i}(j_{u_0}^\ell\Phi)\\
&=& j_{u_1}^\ell\Phi-j_{u_0}^\ell\Phi +  (H^\ell_{z_i} - Id)(j_{u_1}^\ell\Phi) -
(H^\ell_{z_i} - Id)(j_{u_2}^\ell\Phi)
\end{eqnarray*}

Summing these equations
from $i=1,\ldots, N$, and taking the norm, we obtain:
\begin{eqnarray*}
|j^\ell_{u_N}\Phi - j^\ell_{u_{0}}\Phi| &\geq  & N|j_{u_1}^\ell\Phi-j_{u_0}^\ell\Phi|\\
\qquad\qquad\qquad\qquad\qquad\quad &&
- \sum_{i=1}^N \left| (H^\ell_{z_i} - Id)(j_{u_1}^\ell\Phi) -
(H^\ell_{z_i} - Id)(j_{u_0}^\ell\Phi)\right|\\
&\geq& \frac{N}{2}|j_{u_1}^\ell\Phi-j_{u_0}^\ell\Phi|,
\end{eqnarray*}
since $\hbox{Lip}(H^\ell_{z_i} - Id) < \frac{1}{2}$, for $i=1,\ldots, N$.

Since $j^\ell\Phi$ is continuous, by assumption, there exists a constant
$M>0$ such that $|j^\ell_{v}\Phi|\leq M$, for all $v\in U$.
Then:
\begin{eqnarray*}
|j^\ell_{u_1}\Phi - j^\ell_{u_{0}}\Phi| &\leq  & \frac{2}{N}|j_{u_N}^\ell\Phi-j_{u_0}^\ell\Phi|\\
&\leq& \frac{4M}{N}\\
&=& \frac{4MCK(N+1)}{n\alpha} \frac{\alpha}{CK(N+1)}\\
&\leq& \frac{12MC}{\alpha} |u_1-u_0|.
\end{eqnarray*}

From this it follows that $u\mapsto j_u^\ell\Phi$ is Lipschitz at $u_0$; since $u_0$ was arbitrary,
the map is locally Lipschitz on $U$.  Hence $j^\ell\Phi$ is differentiable almost everywhere
 on $U\subset V$. $C^{\ell+1}$-homogeneity of $V$ now implies that $j^\ell\Phi$ is
differentiable everywhere on $V$.  Taking
a point of continuity for the derivative of $j^\ell\Phi$, and applying
$C^{\ell+1}$-homogeneity one more time, we obtain that $j^\ell\Phi$ is $C^1$,
and so $V$ is a $C^{\ell+1}$ submanifold
of $\RR^n\times \RR^m$.  This completes the inductive step of our proof, and so
completes the proof that $N$ is a $C^k$ submanifold of $\RR^{m+n}$.

\section{Journ\'e's theorem, re(re)visited.}\label{s=journe}

Journ\'e's theorem \cite{J} is widely used in rigidity theory to show that a continuous function is smooth. The theorem states
that any function that is uniformly smooth along leaves of two transverse foliations with uniformly smooth leaves
is smooth. This theorem is typically
applied in the Anosov setting as follows:  according to Proposition~\ref{p=satprop},
the graph of a continuous transfer function $\Phi$ for a smooth coboundary $\phi$ is bisaturated, i.e.
saturated by leaves of the unstable and stable foliation for the skew product $f_\phi$.  
Since $f_\phi$ is smooth, the leaves of these foliations are smooth graphs over the corresponding foliations
for $f$.  This implies that the function $\Phi$ is smooth along leaves of the stable and unstable foliations $\W^s$ and $\W^u$ for $f$.  In the Anosov setting, these foliations are transverse, so applying Journ\'e's theorem, we obtain that $\Phi$ is smooth (see \cite{NT}).

Here in the partially hyperbolic setting, we reproduce this argument in part.  Indeed, by the same argument, any continuous transfer function $\Phi$ of a smooth coboundary $\phi$ 
is smooth along leaves of $\W^s$ and $\W^u$.  Since the stable and unstable foliations are not necessarily transverse,
we cannot apply Journ\'e's theorem at this point. The idea is to use accessibility and center bunching to show that the restriction of $\Phi$ to leaves of a center foliation is also smooth.  
One then applies Journ\'e's theorem twice, first to the pair of 
foliations $\W^c$ and $\W^u$, and then to the pair $\W^{cu}$ and $\W^s$, to conclude that $\Phi$ is smooth.

If one assumes that $f$ is dynamically coherent, then it is possible to turn this idea into a rigorous argument,
as we outlined above in Section~\ref{s=techniques}. Here are a few more details on how one can
show that $\Phi$ is smooth along leaves of $\W^c$ in the dynamically coherent setting.
Bisaturation of $\Phi$ implies that the graph of $\Phi$ when restricted to the $\W^c$-manifolds
is invariant under the $\W^{s}_{\phi}$ and $\W^{u}_{\phi}$-holonomy maps between lifted 
$\W^c_{\phi}$-manifolds.  The strong bunching hypothesis on $f$ implies that these 
holonomy maps are smooth when restricted to
center manifolds of $f_\phi$.  Each center manifold $\W^c_{\phi}(p,t)$ of $f_\phi$ is the product 
$\W^c(p)\times \RR$ of a center manifold for $f$ with $\RR$, and the $\W^{s}_{\phi}$ and $\W^{u}_{\phi}$-holonomies
between $\W^c_{\phi}$-manifolds covers the corresponding $\W^{s}$ and $\W^{u}$-holonomies between $\W^c$-manifolds.  
Since $f$ is accessible and $\Phi$ is bisaturated, any two points on the graph of $\Phi$
can be connected by an $su$-lift path. Corresponding to any such $su$-lift path is
a composition of $\W^{s}_{\phi}$ and $\W^{u}_{\phi}$-holonomy diffeomorphisms between $\W^c_{\phi}$-manifolds
that preserves the graph of $\Phi$. Putting all of this together, we get that the graph of $\Phi$ over any 
given center manifold $\W^c(p)$ is  a smoothly 
homogeneous submanifold of $\W^c(p)\times\RR$ and so by Corollary~\ref{c=homog} is a smooth submanifold.  
Hence the restriction of $\Phi$ to $\W^c$ leaves is also uniformly smooth. 

If we do not assume dynamical coherence, then this argument fails. One can attempt to use in place
of a center foliation a local ``fake'' center foliation $\hW^c_x$, as is done in \cite{BWannals} to
prove ergodicity.
However, the fake center foliation $\hW^c_x$ available to us is not sufficiently canonical  
to allow a dynamical proof that the graph of $\Phi$ is smoothly homogeneous over $\hW^c_x$ leaves.  Another
difficulty is that the fake center foliation and the unstable foliation $\W^u$ are not 
jointly integrable, and so we cannot apply Journ\'e's theorem in the two steps outlined above.  
Fortunately, both problems can be overcome, and it is possible to employ the fake foliations
of \cite{BWannals} to prove Theorem~\ref{t=main}.  The key observations that allow is to do this are:
\begin{enumerate}
\item the fake center foliation $\hW^c_x$ and the {\em fake} unstable foliation
$\hW^u_x$ are jointly integrable,
\item one can show that $\Phi$ has continuous ``approximate jets'' along 
leaves of  $\hW^u_x$ and $\hW^c_x$, and
\item Journ\'e's theorem has a stronger formulation in terms of ``approximate jets''.
\end{enumerate}
We detail the argument in the next section.  In this section, we describe the stronger formulation
of Journ\'e's theorem and what we mean by ``approximate jets.''

\begin{dfinition} Let $D$ be a domain in $\RR^m$, $C\geq 1$, $\alpha>0$ and $\ell\in \ZZ_+$.  A function $\psi\colon D\to \RR^n$ {\em has an $(\ell,\alpha,C)$-expansion at $z$} if there exists a polynomial $\wp_z$ of degree $\leq \ell$
such that:
$$|\psi(z')- \wp_z(z')| \leq C |z-z'|^{\ell+\alpha},$$
for all $z'\in D$.
\end{dfinition}

The following theorem was proved by Campanato (in a more general context):

\begin{theorem}\label{t=campanato}\cite{campanato}  For $\ell\in \ZZ_+$ and 
$\alpha \in (0,1]$, a function
$\psi:\RR^m\to \RR^n$ is $C^{\ell,\alpha}$ if and only if, for every compact set $D\in \RR^m$,
there exists $C>0$  such that $\psi$ has an $(\ell,\alpha,C)$-expansion at every $z\in D$.

Furthermore, $\psi$ is a polynomial of degree $\leq \ell$ if and only if there exists
$\alpha > 1$ such that, for every compact set $D\in \RR^m$,
there exists a $C>0$ such that $\psi$ has an $(\ell,\alpha,C)$-expansion at every $z\in D$.
\end{theorem}

\begin{dfinition} A {\em parametrized $C^{\ell,\alpha}$ transverse pair of plaque families} is
a pair of maps $(\omega^H,\omega^V)$, with
$$\omega^H\colon  I^{m+n} \times I^m \to \RR^{m+n},\quad\hbox{and}\quad\omega^V \colon I^{m+n}\times I^n \to \RR^{m+n},$$
of the form:
$$\omega^H_z(x) = z + (\beta^H_z(x),x),\quad\hbox{and}\quad\omega^V_z(y) = z + (y, \beta^V_z(y)),$$
for $z\in  I^{m+n}$, where $\beta^H_z \in C^{\ell,\alpha}(I^m,\RR^n)$ and $\beta^V_z \in C^{\ell,\alpha}(I^n,\RR^m)$
have the following additional properties:
\begin{enumerate}
\item $\beta_z^H(0) = 0$ and  $\beta_z^V(0) = 0$, for all $z\in I^{m+n}$,
\item $\beta^H_{(0,0)}(x)  = 0$  for every $x\in I^m$,
and $\beta^V_{(0,0)}(y) =  0$,  for every $y\in I^n$,
\item The maps $z\mapsto \beta^H_z \in C^{\ell,\alpha}(I^m,\RR^n)$
and
$z\mapsto \omega^V_z \in C^{\ell,\alpha}(I^n,\RR^m)$ are continuous.
\end{enumerate}

\medskip

\noindent
If $(\omega^H,\omega^V)$ is a parametrized $C^{\ell,\alpha}$ transverse pair of plaque families,
we define the norm  $\|(\omega^H,\omega^V)\|_{\ell,\alpha}$ as follows:
$$\|(\omega^H,\omega^V)\|_{\ell,\alpha} := \sup_{z\in I^{m+n}} \|\beta^H_z\|_{C^{\ell,\alpha}(I^m,\RR^n)}
+ \|\beta^V_z\|_{C^{\ell,\alpha}(I^n,\RR^m)}.$$

\end{dfinition}

\begin{remark} A pair of transverse foliations with uniformly $C^{\ell,\alpha}$ leaves, after a  $C^{\ell,\alpha}$
local change of coordinates, becomes a parametrized transverse pair of plaque families.  
Similarly, a pair of continuous plaque families (where the plaques depend continuously on the their center point
in the $C^{\ell,\alpha}$ topology) transverse at every point gives a transverse pair of plaque families.
\end{remark}

\begin{theorem}\label{t=journe} Fix $\ell\in\ZZ_+$ and $\alpha\in(0,1)$.
Let $(\omega^H,\omega^V)$ be a parametrized $C^{\ell,\alpha}$ transverse pair of  plaque families in  $I^{n+m}\subset \RR^n\times \RR^m$.  For every $C>0$ there exist
 $C' =  C'(C,\|(\omega^H,\omega^V)\|_{\ell,\alpha})$  and $\rho = \rho(C,\|(\omega^H,\omega^V)\|_{\ell,\alpha})$ such that the following holds.

Suppose that $\psi:I^{n+m}\to \RR$ has the properties:
\begin{itemize}
\item[(1)] for every $z\in I^{m+n}$, there exists a polynomial $\wp_z^H: I^m\to \RR$ of degree $\leq \ell$ such that,
for all $x\in I^m$:
$$|\psi(\omega^H_z(x))- \wp_z^H(x)| \leq C (|x|^{\ell+\alpha} + |z|^{\ell+\alpha}),$$
\item[(2)] for every $z\in I^{m+n}$, there exists a polynomial $\wp_z^V: I^n\to \RR$ of degree $\leq \ell$ such that,
for all $y\in I^n$:
$$|\psi(\omega^V_z(y))- \wp_z^V(y)| \leq C  (|y|^{\ell+\alpha} + |z|^{\ell+\alpha}),$$
\end{itemize}
Then $\psi$ has an $(\ell,\alpha, C')$-expansion at $(0,0)$ in $B_{\RR^{m+n}}(0,\rho)$.
\end{theorem}

\begin{remark} Note that the hypotheses of Theorem~\ref{t=journe} are  weaker than requiring that 
$\psi\circ \omega^H_z$ and $\psi\circ \omega^H_z$ be $C^{\ell,\alpha}$ for every $z\in I^{m+n}$.
They are also weaker than requiring that $\psi\circ \omega^H_z$ and 
$\psi\circ \omega^H_z$ have $(\ell,\alpha,C)$-expansions about $0$ for every $z$.
This latter condition corresponds to the stronger conditions:
$$|\psi(\omega^H_z(x))- \wp_z^H(x)| \leq C |x|^{\ell+\alpha},\quad\hbox{and}\quad
|\psi(\omega^V_z(y))- \wp_z^V(y)| \leq C  |y|^{\ell+\alpha},$$
for every $(x,y)$.
Note also that the conclusion of Theorem~\ref{t=journe} is in some aspects very weak: it 
does not even imply that $\psi$ is continuous (except at the origin). 

One can recover Journ\'e's original result from Theorems~\ref{t=journe} and \ref{t=campanato}
as follows. Suppose that $\psi$ is uniformly $C^{\ell,\alpha}$ along the leaves of
two transverse foliations with uniformly $C^{\ell,\alpha}$ leaves.  Fix an arbitrary point $x$;
in local coordinates sending $x$ to $0$,
the transverse foliations give a  parametrized $C^{\ell,\alpha}$ transverse pair of plaque families.
In the coordinates given by this parametrization, $\psi$ has a Taylor expansion at every point
with uniform remainder term on the order $\ell +\alpha$.  This implies that conditions (1) and (2)
in Theorem~\ref{t=journe} hold, for some $C$ that is uniform in the point $x$.  Theorem~\ref{t=journe}
implies that $\psi$ has an $(\ell,\alpha,C')$ expansion (in these coordinates) at $x$, where
$C'$ is uniform in $x$. Since $x$ was arbitrary, Theorem~\ref{t=campanato} then implies that $\psi$
is $C^{\ell,\alpha}$.

We also remark that whereas Theorem~\ref{t=campanato} holds for $\alpha=1$, Theorem~\ref{t=journe} is
false for $\alpha=1$, if $\ell>1$ (see \cite{PSW} for an example with $\alpha=1$, $\ell=1$).
\end{remark}

\begin{proof}[Proof of Theorem~\ref{t=journe}] The proof amounts to a careful inspection of the main result in \cite{J}.  We follow the
format in \cite{NicT}, where the structure of the original treatment in \cite{J} has been clarified.
We retain as much as possible the notation from \cite{J, NicT}, though there are some small changes.
The two differences in the way the result is stated here and the way it is stated in \cite{J} are the following:
\begin{enumerate}
\item In \cite{J}, the transverse plaque family arises from a transverse pair of local foliations $\cF_s$ and $\cF_u$; this is not assumed here. An extra lemma (Lemma~\ref{l=goodgrid}) deals with this.
\item In \cite{J}, it is assumed that $\psi$ is $C^{\ell,\alpha}$ along leaves of the foliations $\cF_s$ and $\cF_u$. This is replaced by (1) and (2).  A slight adaptation of the proof of Lemma~\ref{l=gunk}, part 1, deals with this.
\end{enumerate}

As in \cite{J} and \cite{NicT}, we give the proof for $m=n=1$; the proof for general $m,n$ is
completely analogous.  We first reduce Theorem~\ref{t=journe}
to the following lemma. 

\begin{lemma}[cf. \cite{NicT}, Lemma 4.4]\label{l=cone}  Under the hypotheses of
Theorem~\ref{t=journe}, there is a polynomial $\wp=\wp(\psi)$
of degree $\leq \ell$ with the following property.  Given $\kappa>0$ and the cone
$\cK_\kappa = \{(u,v)\in\RR^2 :  |v|\leq \kappa|u|\}$, there exist positive constants $C_1=C_1(\kappa,C,\|(\omega^H,\omega^V)\|_{\ell,\alpha})$
and  $\rho=\rho(\kappa,C,\|(\omega^H,\omega^V)\|_{\ell,\alpha})$ such that:
\begin{eqnarray}\label{e=conepolynomial}
|\psi(z) - \wp(z)| \leq C_1|z|^{\ell +\alpha},\quad\hbox{for}\,z\in \cK\cap B(0,\rho_1).
\end{eqnarray}
\end{lemma}

We first prove Theorem~\ref{t=journe} using Lemma~\ref{l=cone}. Fix $\kappa>2$.
Applying Lemma~\ref{l=cone}
to the cones $\cK = \{(u,v)\in\RR^2 :  |v|\leq \kappa|u|\}$ and $\cK' = \{(u,v)\in\RR^2 :  |u|\leq \kappa|v|\}$
(with the roles of $u$ and $v$ switched), we obtain polynomials $\wp$ and $\wp'$
of degree $\leq \ell$ and constants $C',\rho$ such that
$$
|\psi(z) - \wp(z)| \leq C'|z|^{\ell +\alpha},\quad\hbox{for}\,z\in \cK\cap B(0,\rho),
$$
and
$$
|\psi(z) - \wp'(z)| \leq C'|z|^{\ell +\alpha},\quad\hbox{for}\,z\in \cK'\cap B(0,\rho).
$$
Note that $V=B(0,\rho)\cap \cK\cap\cK'$ has nonempty interior.
But then $\wp$ and $\wp'$ must agree because they have contact higher than $\ell$ on $V$.
Hence $\psi$ has an $(\ell,\alpha,C')$ jet on $B_{\RR^2}(0,\rho)$.
This completes the proof of Theorem~\ref{t=journe}, assuming Lemma~\ref{l=cone}.
\end{proof}

\begin{proof}[Proof of Lemma~\ref{l=cone}]
Replacing $\psi$ by $\psi(x,y) - \psi(x,0) - \psi(0,x) + \psi(0,0)$, we may
assume that $\psi$ vanishes along the $x$- and $y$-axes.  
For $z\in I^{m+n}$, let $\cF^H(z) = \omega^H_z(I^m)$
and let $\cF^V(z) = \omega^V_z(I^n)$.

The structure of the 
proof is as follows.  We construct a sequence of degree $(\ell+1)^2$ polynomials $\wp_m$ on
$I^2$ that interpolate the values of $\psi$ on a carefully chosen collection $S_m$ of
$(\ell+1)^2$ points in $\RR^2$. The terms of degree
$\leq \ell$ in $\wp_m$ converge to a degree $\ell$ polynomial $\wp$ that 
satisfies (\ref{e=conepolynomial}) on a cone $\cK_\kappa$.
\begin{figure}[h]\label{f=grids}
\psfrag{sk}{$S_{2k}$}
\psfrag{sk+1}{$S_{2k+1}$}
\psfrag{sk+2}{$S_{2k+2}$}
\begin{center}
\includegraphics[scale=1.0]{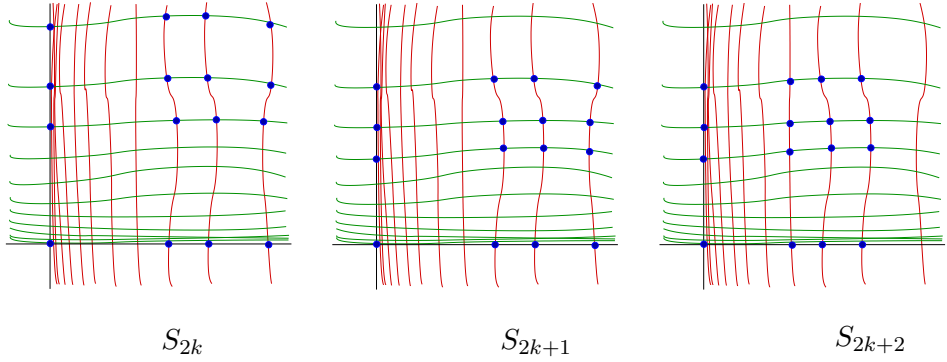}
\caption{The geometry of the sets $S_m$, when $\ell=3$.}
\end{center}
\end{figure}

We say more about the selection of sets $S_m$.  Each set $S_m$ is the union of
four subsets $S_m = \{(0,0)\}\cup (H_m\times\{0\}) \cup (\{0\} \times V_m) \cup J_m$,
where $H_m$ and $V_m$ each contain $\ell$ distinct real positive numbers.
The sets $S_m$ are chosen with several properties:
\begin{itemize}
\item the minimum and maximum distance between any two points in $S_m$ are comparable by
a fixed factor $B\geq 1$ and are both $O(r^{m/2})$, for some fixed $r\in(0,1)$, 
\item  $J_m$ is approximately the cartesian product $H_m\times V_m$, with error $o(r^{m/2})$,
\item  any ``vertical'' collection of $\ell+1$ points in $S_m$ lies on a vertical $\cF^V$-plaque,
and any ``horizontal'' collection of $\ell+1$ points in $S_m$ lies on a horizontal $\cF^H$-plaque,
\item  $S_m$ and $S_{m+1}$ agree on $\ell$ (horizontal or vertical, depending on the parity of $m$) 
collections of $\ell+1$ points.
\end{itemize}
These properties, combined with properties (1) and (2) of $\psi$ 
ensure both that the degree $\leq\ell$ terms in the polynomials $\wp_m$ converge  
and that the limiting polynomial is a good approximation
to $\psi$ on any cone $\cK_\kappa$. We will say more about the construction of $S_m$ shortly;
we note that it will be necessary to construct more than one such sequence,
in order to prove that $\wp$ is a good approximation at all points in $\cK$,
and not just those points on which $\psi$ was interpolated.

The starting point in Journ\'e's  argument is to prove a higher dimensional version of the
following interpolation lemma.

\begin{lemma}[Basic interpolation lemma. \cite{J}]\label{l=basicinterp} Fix $\ell \geq 1$. For each $B\geq 1$,
there exist $\eps_0=\eps_0(B)>0$ and $C_0=C_0(B)>0$ with the following property.  
If the collection of points $\{z_0,z_1, \ldots, z_\ell\}\subset \RR$ satisfies
$R/\eta < B$, where
$$R = \sup_{j} |z_{j}|\quad\hbox{and}\quad \eta = \inf_{j\ne j'}| z_{j} - z_{j'} |,$$
Then for any values $\{b_0,\ldots, b_\ell\}\subset \RR$,
there exists a unique polynomial
$$\wp(x) =  \sum_{p=0}^{\ell} c_{p}x^p$$
such that $\wp(z_{j}) = b_{j}$, for $0\leq j \leq \ell$. Moreover,
$$\sum_{p} |c_{p}|R^{p} \leq C\sup_{j}|b_{j}|.
$$
\end{lemma}

Journ\'e's generalization of Lemma~\ref{l=basicinterp} allows
one to interpolate values of a function on a collection of $(\ell+1)^2$ points in $\RR^2$ 
that lie in a rectangle-like configuration -- like
the sets $S_m$ described above -- by a degree $(\ell+1)^2$ polynomial whose
$C^0$ size is controlled on the scale of the grid:

\begin{lemma}[Rectangle interpolation lemma. \cite{J}, Lemma 1; cf. \cite{NicT}, Lemma 4.5]\label{l=interp} Fix $\ell \geq 1$. For each $B\geq 1$,
there exist $\theta_0=\theta_0(B)>0$ and $C_0=C_0(B)>0$ with the following property.  If the collections of
points $\{z_{j,k} :  0\leq j\leq \ell,\, 0\leq k\leq \ell\}\subset \RR^2$,
$\{x_{j} :  0\leq j\leq \ell \}\subset \RR$ and $\{y_{k} :  0\leq k\leq \ell \}\subset \RR$
satisfy:
$$R/\eta < B,\quad\hbox{and}\quad |z_{j,k} - (x_{j},y_{k})| \leq \theta_0\eta,$$ where
$$R = \sup_{j,k} |z_{j,k}|\quad\hbox{and}\quad \eta = \inf_{(j,k)\ne (j',k')}| z_{j,k} - z_{j',k'} |,$$
Then for any values $\{b_{j,k} :  0\leq j\leq \ell, \,0\leq k\leq \ell\}\subset \RR$,
there exists a unique polynomial
$$\wp(x,y) =  \sum_{0\leq p,q\leq \ell} c_{pq}x^p y^q$$
such that $\wp(z_{j,k}) = b_{j,k}$, for $0\leq j,k \leq \ell$. Moreover,
$$\sum_{p,q} |c_{p,q}|R^{p+q} \leq C_0\sup_{j,k}|b_{j,k}|.
$$
\end{lemma}
As mentioned above, to create the sets $S_m$, we will intersect plaques
of our transverse plaque families.  The next lemma gives control over
the location of the intersection of two transverse plaques.
\begin{figure}[h]\label{f=grids3}
\psfrag{K}{$\cK$}
\psfrag{(x,y)}{$(x,y)$}
\psfrag{(0,y')}{$(0,y')$}
\psfrag{(x,y')}{$(x,y')$}
\psfrag{[(x,y),(0,y')]}{$[(x,y),(0,y')]$}
\begin{center}
\includegraphics[scale=.85]{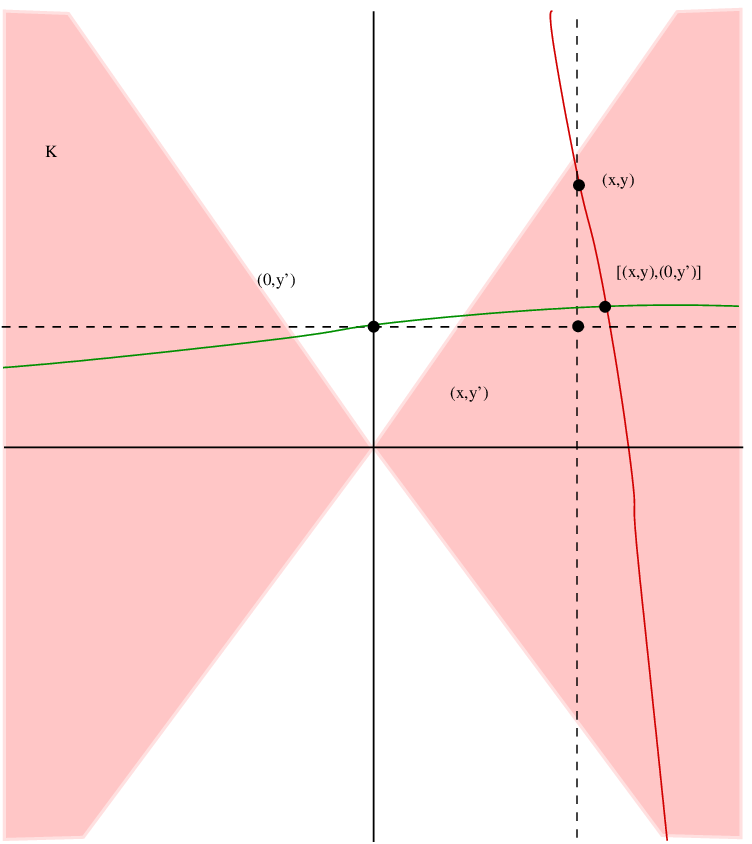}
\caption{Lemma~\ref{l=lps}}
\end{center}
\end{figure}
\begin{lemma}[Local product structure]\label{l=lps} There exists $\rho_0 = \rho_0(K)>0$, 
and for every $K, \theta >0$, a $\rho_1 = \rho_1(K,\theta)>0$ with $\rho_1<\rho_0$
such that, for any  parametrized $C^{\ell,\alpha}$ transverse pair of
plaque families $(\omega^H,\omega^V)$ with $\|(\omega^H,\omega^V)\|_{1}\leq K$,
and any $z_1,z_2\in B_{\RR^{m+n}}(0,\rho_0)$, the manifolds $\omega^V_{z_1}(I^m)$
and $\omega^H_{z_2}(I^n)$ intersect transversely in a single point $[z_1,z_2]\in I^{m+n}.$
Moreover, if $|(x,y)| < \rho_1$, and $|(x',y')|<\rho_1$ then  
$$|[(x,y),(0,y')] - (x,y') | < \theta (|(x,y)|+|y'|),$$
and
$$|[(x',0),(x,y)] - (x',y)|<\theta (|(x,y) + |x'|).$$
\end{lemma}
\begin{proof} This is a simple consequence of the fact that the transverse plaque families are
continuous in the $C^1$ topology.
\end{proof}

Fix $K>0$ and $\kappa \geq 1$ and let $\rho_0 = \rho_0(K)$. Fix  $(\omega^H,\omega^V)$ such that $\|(\omega^H,\omega^V)\|_{\ell,\alpha}\leq K$ .
We now define the {\em base grid:} 
$$\cG_0 = \cG_0(\omega^H,\omega^V) = 
(\{\cF^V_j\}_{j\in \ZZ_+\cup\{\infty\}},\{ \cF^H_k \}_{k\in \ZZ_+\cup\{\infty\}})$$ 
of horizontal and vertical plaques from which  we will eventually
construct the sets $S_m$.  We fix $r\in (0,1)$, and
let $\cF^H_\infty = \cF^H(0,0)$ and $\cF^V_\infty = \cF^V(0,0)$,   and for $j, k\geq 1$ set
$\cF^V_j = \cF^V(r^j,0)$ and $\cF^V_k = \cF^V(0,r^j)$.

For each (nonzero) $w\in B_{\RR^{m+n}}(0,\rho_0)$, we also define a new grid  $\cG_w$ as follows.
We choose $j = j(w)\in\ZZ_+$ such that the quantity
$$|[w, (0,0)] - r^j|$$ 
is minimized.  The grid $\cG_w$ is then the same as $\cG_0$, except that the 
vertical leaf $\cF_j^V$ in $\cG_0$ is redefined: $\cF_j^V = \cF^V(w)$.
This is illustrated in Figure 3.
\begin{figure}[h]\label{f=grids2}
\psfrag{z}{$w\,$}
\begin{center}
\includegraphics[scale=1.0]{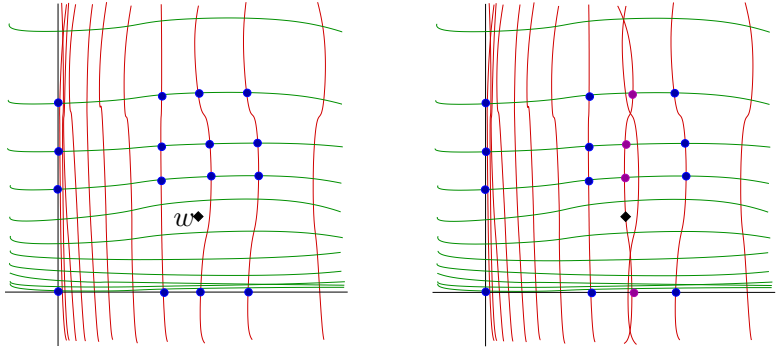}
\caption{Grid substitution}
\end{center}
\end{figure}

Each grid $\cG = (\{\cF^V_j\},\{ \cF^H_k \}) $ defines  sequences of points 
$\{z_{j,k}\}_{j,k\in\ZZ_+}\subset \RR^2$, and $\{x_j\}, \{y_k\}\subset \RR$
via:
$\{z_{j,k}\} = \cF^V_j\cap \cF^H_k$, $\{(x_j,0)\} =  \cF^V_j\cap \cF^H_\infty$,
and $\{(0,y_k)\} =  \cF^V_\infty\cap \cF^H_k$.
For each pair $(j,k)$ with $|j-k|\leq 1$, 
we then define 
$$H_{j,k} = H_{j,k}(\cG) = \{x_{j'} :  j\leq j'\leq j+ \ell \},
V_{j,k} =  V_{j,k}(\cG) =\{y_{k'} :  k\leq k'\leq k + \ell \}$$
and $$J_{j,k} =J_{j,k}(\cG) =  \{z_{j',k'} :  j\leq j'\leq j+\ell,\, k\leq k'\leq k+\ell\}.$$

\begin{lemma}[Grids are good]\label{l=goodgrid} For every $K>0$ and $\kappa>1$, 
there exists $\rho_2 = \rho_2(K,\kappa)>0$ such that
if $\|(\omega^H,\omega^V)\|_{1}\leq K$, then 
for every $\theta>0$, there exists an integer $k_0 = k_0(K,\kappa,\theta)>0$ such that: 
for all $k\geq k_0$, for all $j$ with $|j-k|\leq 1$, and for all
$w\in B_{\RR^{m+n}}(0,\rho_2)\cap \cK_\kappa$,
the grid $\cG_w$ has the following properties.
$$R_{j,k}/\eta_{j,k} \leq 6 r^{\ell-2},\quad\hbox{and}\quad \sup_{z_{j',k'}\in J_{j,k}}|z_{j',k'} - (x_{j'},y_{k'})| \leq \theta\eta_{j,k},$$
where 
$$R_{j,k} =  \sup_{z\in J_{j,k}} |z|  \quad\hbox{and}\quad \eta_{j,k} = \inf_{z,z'\in J_{j,k},\, z\ne z'}| z - z' |.$$
Moreover, $R_{j,k} \leq 3 r^{k-1}$.
\end{lemma}

\begin{proof}[Proof of Lemma~\ref{l=goodgrid}]
Let $K>0$ and $\kappa>1$ be given, and
suppose that $\|(\omega^H,\omega^V)\|_{1}\leq K$.

We choose $\rho_2$ such that for all $w\in B_{\RR^2}(0,\rho_2)\cap \cK_\kappa$,
and for $j$ sufficiently large (greater than some $j_0$),
if $j$ minimizes the quantity $|[w, (0,0)] - r^j|$, then $|w|\leq  2(1+\kappa)r^{j}$. 
This is possible, by Lemma~\ref{l=lps}.

Let $\theta>0$ be given;  we will describe below how to choose a constant
$\theta_1 = \theta_1(K,\kappa,\theta)$.
Assuming this choice has been made, let $\rho_1 =  \rho_1(K,\theta_1)$ be given by Lemma~\ref{l=lps}.
We  choose $k_0>j_0$ such that $\max\{2(1+\kappa)r^{k-1},  R_{j,k} \}  < \rho_1$, for all $|j-k|\leq 1$
and $k\geq k_0$.

Let $w\in B_{\RR^{m+n}}(0,\rho_2)\cap \cK_\kappa$,
and consider the grid $\cG_w$.
For $j,k \in \ZZ_+$ satisfying $|j-k|\leq 1$, and $k\geq k_0$, fix a point $z\in J_{j,k}$,
which by definition is the point of intersection of $\cF^V_{j'}$ and $\cF^H_{k'}$,
for some $k-1 \leq j',k'\leq  k+ \ell + 1$.  Write $z= (x,y)$ and $w=(x',y')$.
There are two possibilities.  Either $\cF^V_{j'}$  
is in the base grid $\cG_0$,  or 
$\cF^V_{j'} = \cF^V(w)$.

In the first case, since $z\in \cF^V_{j'}\cap \cF^H_{k'}$, we have $|z| < \rho_1$.
Lemma~\ref{l=lps} implies that 
$$|[(0,0),(x,y)] - (0,y)| = |y_{k'} - y| = |r^{k'}  - y | < \theta_1|(x,y)|$$
and
$$|[(x,y),(0,0)] - (x,0)| = |x_{j'} - x| = |r^{j'} - x|  < \theta_1|(x,y)|.$$
and so  $|z - (x_{j'},y_{k'})| \leq  |x_{j'} - x|  + |y_{k'} - y|  \leq  2\theta_1|z|$.
Since $|(x_{j'},y_{k'})|\leq 2r^{k-1}$, we therefore have, for $\theta_1$ sufficiently small:
\begin{eqnarray}\label{e=Rbound1}
|z| &\leq & 3 r^{k-1},
\end{eqnarray}
and
\begin{eqnarray}\label{e=gridbound1}
|z - (x_{j'},y_{k'})| \leq 6\theta_1 r^{k-1}.
\end{eqnarray}

Suppose, on the other hand, that $\cF^V_{j'} = \cF^V(w)$.  
Then the point $(x_{j'},0) = [w,(0,0)]$ has the property that
$$|x_{j'} - r^{j'}|\leq \frac{1}{2}|r^{j'} - r^{j'+1}| =  \frac{( 1-r )}{2}r^{j'} < \frac{r^{j'}}{2}.$$
Since $w\in B_{\RR^2}(0,\rho_2)\cap \cK_\kappa$, and $j' \geq k_0$, we have that $|w|< 2(1+\kappa)r^{j'-1} < \rho_1$.
Hence Lemma~\ref{l=lps} implies that $|x_{j'} - x'| = |[w,(0,0)] - (x',0)|\leq \theta_1 (|w| + |x'|)$;
This implies that  $|x_{j'} - x'| = \leq \theta_1 (|w| + |x'|) \leq   2 \theta_1 |w| \leq  4\theta_1 (1+\kappa)r^{j'-1}$.

Now $z = [w,(0,r^{k'})]$ and $|[w,(0,r^{k'})] - (x',r^{k'})| \leq \theta_1(|w| + r^{k'}) \leq \theta_1(3+2\kappa)  r^{k-1}$. Using the triangle inequality, we conclude that, for $\theta_1$ sufficiently small, we have
\begin{eqnarray}\label{e=Rbound2}
|z| &\leq & 3r^{k-1}
\end{eqnarray}
and
\begin{eqnarray}\label{e=gridbound2}
|z - (x_{j'}, y_{k'})| \leq |z - (x', r^{k'})| + |x_{j'} - x'| \leq \theta_1 (7+6\kappa) r^{k-1}.
\end{eqnarray}
Hence, in either case, we conclude that
\begin{eqnarray}\label{e=Rbound3}
R_{j,k} &\leq & 3r^{k-1}
\end{eqnarray}
and
\begin{eqnarray}\label{e=gridbound3}
\sup_{z_{j',k'}\in J_{j,k}} |z_{j',k'} - (x_{j'}, y_{k'})| \leq \theta_1(7 + 6\kappa) r^{k-1}.
\end{eqnarray}

On the other hand,
\begin{eqnarray}\label{e=etabound}
\eta_{j,k} &\geq &\inf_{j'\ne j''} |y_{j'} - y_{j''}| - \sup_{z_{j',k'}\in J_{j,k}} |z_{j',k'} - (x_{j'}, y_{k'})|\\
& > & r^{\ell+k+1} - \theta_1(7 + 6\kappa) r^{k-1},
\end{eqnarray}
and for $\theta_1$ sufficiently small, we get
$\eta_{j,k} \geq r^{\ell+k+1}/2$.  Combining this with (\ref{e=Rbound3}), we have $R_{j,k}/\eta_{j,k} \leq 6 r^{\ell-2}$.
Combining (\ref{e=etabound}) with (\ref{e=gridbound3}) we also get:
\begin{eqnarray*}
\sup_{z_{j',k'}\in J_{j,k}}|z_{j',k'} - (x_{j'},y_{k'})| \leq \eta_{j,k} \frac{\theta_1(7 + 6\kappa) r^{k-1}}
{ r^{\ell+k + 1} - \theta_1(7 + 6\kappa) r^{k-1}}.
\end{eqnarray*}
Choosing $\theta_1 = \theta_1(K,\kappa,\theta)$ small enough, we obtain that
\begin{eqnarray*}
\sup_{z_{j',k'}\in J_{j,k}}|z_{j',k'} - (x_{j'},y_{k'})| \leq \theta \eta_{j,k},
\end{eqnarray*}
which finishes the proof.
\end{proof}

Let $B =  6 r^{\ell-2}$ and let $\rho_2 = \rho_2(K,\kappa)>0$ be given by 
Lemma~\ref{l=goodgrid}. Let $\theta_0=\theta_0(B)>0$ and $C_0=C_0(B)>0$
be given by Lemma~\ref{l=interp}. Now let $k_0 = k_0(K,\kappa,\theta_0)>0$ be given
by Lemma~\ref{l=goodgrid}.

Fix $w\in B_{\RR^{m+n}}(0,\rho_2)$. We now define the sequence $S_m$ of rectangles associated
to the grid $\cG_w$.  For $|j-k|\leq 1$,  we set:
$$S_{j,k} = \{0,0\}\cup (H_{j,k}\times\{0\}) \cup (\{0\} \times V_{j,k}) \cup J_{j,k}.$$
Now, let $S_{2k} = S_{k,k}$ and let $S_{2k+1} = S_{k,k+1}$.
Define the sets $H_m$, $V_m$, and $J_m$ analogously, for $m\in \ZZ_+$.
Let $R_m = \sup_{z\in J_{m}} |z| $ and let $\eta_m =  \inf_{z,z'\in J_{m},\, z\ne z'}| z - z' |$.
Lemma~\ref{l=goodgrid} implies that for $m\geq  2k_0$, we have $|R_m| \leq 3r^{(m-1)/2}$, and  
$R_m/\eta_m \leq B$.

By Lemma~\ref{l=interp}, there exists a constant $C_0=C_0(B)>0$ such that
for each $m\geq 2k_0$, and any function $\psi$, 
there exists a unique (degree $(\ell+1)^2$) polynomial 
$\wp_m = \wp_m((\omega^H,\omega^V),w,\psi)$:
$$\wp_m(x,y) = \sum_{0\leq p,q\leq\ell} c_{p,q}^m x^py^q$$
that interpolates $\psi$ on the rectangle $S_m$.
Furthermore:
\begin{eqnarray}\label{e=sumbound}
\sum_{p,q}|c^m_{p,q}| R_{m}^{p+q} \leq C_0\sup \{\psi(z) :  z\in S_m\},
\end{eqnarray}
where $R_m$ is defined above.

\begin{lemma}\label{l=convergence} 
For every $K,C>0$, there exist  constants  $C_1 = C_1(K,C)>0$ and $\rho = \rho(K,C)>0$,
such that for all $(\omega^H,\omega^V)$ with $\|(\omega^H,\omega^V)\|_{\ell,\alpha}\leq K$,
for all $w\in B_{\RR^2}(0,\rho_2)\cap \cK $ and for all $\psi$ satisfying
hypotheses (1) and (2) of Theorem~\ref{t=journe} for this value of $C$, 
the sequence $c_{p,q}^m = c_{p,q}^m(\cG_w,\psi)$
has the following property.

Let $\overline{\wp}_m(x,y) = \sum_{p + q \leq \ell} c^m_{p,q} x^p y^q$.  Then there exists
a polynomial $\overline{\wp}$
such that $\overline\wp = \lim_{m\to\infty} \overline{\wp}_m$ (uniformly on compact sets).  Furthermore:
$$|\overline\wp(z) - \psi(z)| \leq C_1 |z|^{\ell+\alpha}\quad\hbox{for}\quad z\in \cK\cap \bigcup_{k\geq k_0} \cF^V_k \cap B_{\RR^{m+n}}(0,\rho).$$ 
\end{lemma}

We first finish the proof of Lemma~\ref{l=cone}, assuming Lemma~\ref{l=convergence}.  
Let
$C>0$ and  $\psi$ be given satisfying
hypotheses (1) and (2) for this value of $C$.
Let  $C_1 = C_1(K,C)>0$ and $\rho = \rho(K,C)>0$ be given by Lemma~\ref{l=convergence}.
Given $w\in B_{\RR^{2}}(0,\rho)\cap \cK$, let $\overline{\wp} = \overline{\wp}(\cG_w,\psi)$ be given
by  Lemma~\ref{l=convergence}.  By construction of the grid $\cG_w$, we
have that $w\in \bigcup_{k\geq k_0} \cF^V_k$.  This implies in particular that
$$|\overline\wp(w) - \psi(w)| \leq C_1 |w|^{\ell+\alpha}.$$
Let $w'\in B_{\RR^{2}}(0,\rho)\cap \cK$ be another point, and let $\overline{\wp}' = \overline{\wp}(\cG_{w'},\psi)$.
By the same reasoning,
$$|\overline\wp'(w') - \psi(w')| \leq C_1 |w'|^{\ell+\alpha}.$$
Note that the sequences $c^{m}_{p,q}(\cG_w,\psi)$ and $c^{m}_{p,q}(\cG_{w'},\psi)$
differ in only finitely many places.  This implies that $\overline{\wp}'=\overline{\wp}$.
The polynomial $\wp = \overline\wp$ satisfies the conclusions of Lemma~\ref{l=cone}.
This completes the proof of Lemma~\ref{l=cone}.
\end{proof}

\begin{proof}[Proof of Lemma~\ref{l=convergence}]
The proof follows the proof of Lemma 4.4 in \cite{NicT} very closely;
the only slight change occurs in the proof of Lemma~\ref{l=gunk}, part (1) below,
which corresponds to Lemma 4.8 in \cite{NicT}.  We outline the proof
and refer the reader to \cite{NicT} or \cite{J} for the details.

Fix $k$ and let $\wp = \wp_{2k}$ and $\wp' = \wp_{2k+1}'$ be the interpolating polynomials on 
$S_{2k} = S_{k,k}$ and $S_{2k+1} = S_{k,k+1}$, respectively. Denote their coefficients by
$c_{p,q}$ and $c_{p,q}'$ respectively. Let $T_k = 3r^{k-1}$.
We will show that
$$|c_{p,q} - c_{p,q}'| = O(T_k^{\ell+\alpha - p - q}).$$

By Lemma~\ref{l=interp}, it is enough to consider the polynomial $\wp - \wp'$
and find an upper bound for $|\wp-\wp'|$ on $S_{k,k+1}$.
Note that $\wp$ and $\wp'$ agree on $S_{k,k+1}$, except at the $\ell$ points
$z_{j,k+\ell}$, $k\leq j\leq k+\ell$. On these points we have 
$\wp'(z_{j,k+\ell}) = \psi(z_{j,k+\ell})$.  Hence we need only estimate
$|\psi(z_{j,k+\ell}) - \wp(z_{j,k+\ell})|$, for $k\leq j\leq k+\ell$.
For such a $j$, write $\cF^V_j$ as a graph of a function of the second coordinate:
$\cF^V_j = \{(x_j(y),y)\colon y\in I\}$, and let $z_j(y) = (x_j(y),y)$.
Notice that, in the case where $j=j(w)$, we have 
$z_j(y) = \omega^V_{w}(y-y_w)$, where $w = (x_w,y_w)$;
otherwise, $z_j(y) = \omega^V_{(x_j,0)}(y)$.
Note that in either case, $x_{j}(0)=x_j$, and the function $x_j(y)$ would be constant if the curve
$\cF^V_{j,k}$ were truly vertical. The following estimates would be trivial if $x_j$ were
a constant function. The hypothesis that $(\omega^H,\omega^V)$ is uniformly $C^{\ell,\alpha}$ will be
used as in \cite{J, NicT} to estimate the $C^{\ell,\alpha}$ size of $x_j(y)$.

Choose a constant $C_2>0$ 
so that $\{z_j(y) :  y\in I_k\}$ contains all the points in 
$\omega^V_{(x_j,0)}(I)\,\cap \, S_{k,k}\,\cap \,\cK$, for all $k\geq k_0$ and
$k\leq j\leq k+\ell$, where $I_k$ is the interval $I_k:=[-C_2T_k,C_2T_k]$.
We next show that $|\psi(z_j(y)) - \wp(z_j(y))| = O(T_k^{\ell+\alpha})$,
for $k\leq j\leq k+\ell$ and any $y\in I_k$.  
Fix such a $j$.  For $h\colon I^2\to \RR$, write $\tilde h(y)$
for $h(z_j(y))$.  We will restrict attention to the domain $I_k$.

\begin{lemma}\label{l=gunk} There exists $C_3>0$ such that if $k\geq k_0$, $k\leq j\leq j+\ell$,
and $y\in I_k$, then:
\begin{enumerate}
\item $$|(\tilde\psi - \tilde \wp)(y)|\leq C_3 \left\|\frac{d^\ell}{dy^\ell}(\tilde\wp)\right\|_\alpha T^{\ell + \alpha}_k + C_3T^{\ell+\alpha}_k,$$
\item if $p,q\leq \ell$ and $p+q>\ell$, then
$$\left\|\frac{d^\ell}{dy^\ell}x_j^p(y)y^q\right\|_\alpha \leq C_3T_k^{p+q-\ell-\alpha}\|x_{j}\|_{C^{\ell,\alpha}(I_k)}, $$
\item if $p+q\leq \ell$, then
$$\left\|\frac{d^\ell}{dy^\ell}x_j^p(y)y^q\right\|_\alpha \leq C_3,$$
\item and therefore
$$\left\|\frac{d^\ell}{dy^\ell} \tilde\wp \right\|_\alpha \leq C_3\|x_j\|_{C^{\ell,\alpha}(I_k)}\sum_{p+q>\ell} |c_{p,q}| T_k^{p+q-\ell-\alpha} + C_3\sum_{p+q\leq \ell}|c_{p,q}|.$$
\end{enumerate}
\end{lemma}

\begin{proof} To prove (1), recall that 
$z_j(y) = \omega^V_{w}(y-y_w)$, if $j=j(w)$, and $z_j(y) = \omega^V_{(x_j,0)}(y)$ otherwise.
The hypotheses of Theorem~\ref{t=journe} imply that 
$$\tilde\psi(z_j(y)) = \psi(\omega^V_{z_0}(y-y_0)) = \wp_{z}^V(y - y_0) + r_j^V(y-y_0),$$
where $z_0\in\{w,(x_j,0)\}$ and $y_0\in\{0,y_w\}$, 
and $|r_j^V(y-y_0)| \leq C(|z|^{\ell+\alpha} + |y-y_0|^{\ell+\alpha})$. 
Now $|z_0| = O(T_k)$ and $|y_0| = O(T_k)$ (since $w\in \cK$), which implies that
$|r_j^V(y)| = O(T_k^{\ell+\alpha})$, for $y\in I_k$.

Writing the Taylor expansion of of the $C^{\ell,\alpha}$ function $\tilde\wp$ about $0$, we have
$$\tilde\wp(y) = Q(y) + R_j(y),$$
where $Q$ is a degree $\ell$ polynomial and 
$|R_j(y)| = O(|y|^{\ell+\alpha} \left\|\frac{d^\ell}{dy^\ell}\tilde\wp\right\|_\alpha) = O(T_k^{\ell+\alpha} \left\|\frac{d^\ell}{dy^\ell}\tilde\wp\right\|_\alpha)$, for $y\in I_k$.
Recall that, since $k\leq j \leq k+\ell$,
the polynomial $\wp$ interpolates $\psi$ on the $\ell+1$ points in 
$S_{k,k+1}\cap \cF^V(x_j,0)$.
Therefore the degree $\ell$ polynomial $\overline Q(y) = Q(y) - \wp_{z_0}^V(y-y_0)$ on $I_k$
takes the value $r_j^V(t_i) + R_j(t_i)$ at the $\ell+1$ points $$\{0=t_0, t_1, \ldots, t_\ell\} =
(\omega^V_{(x_j,0)})^{-1}\left(S_{k,k+1}\cap \cF^V(x_j,0)\right).$$  
Lemma~\ref{l=lps} implies the points $\{0, t_1, \ldots, t_\ell\}$ in $I_k$ are spaced 
$\Theta(T_k)$ apart.  Since $|\overline Q(t_i)| \leq |r_j^V(t_i) + R_j(t_i)| = O(T_k^{\ell+\alpha} +  T_k^{\ell+\alpha}\left\|\frac{d^\ell}{dy^\ell}\tilde\wp\right\|_\alpha)$, for $i\in\{0,\ldots,\ell\}$, 
Lemma~\ref{l=basicinterp} then gives the desired inequality in (1).

The last three parts are proved in \cite{NicT} (part (4) follows from (2) and (3)).
\end{proof}

Given $\delta>0$, we may assume that $k_0>0$ was chosen sufficiently large so that
$\|x_j\|_{C^{\ell,\alpha}(I_k)} < \delta$. Then we have  
$$|(\psi - \wp)(z_{j}(y))| \leq C_3T^{\ell+\alpha}_k + C_3\delta\sum_{p+q>\ell}|c_{p,q}| T_k^{p+q}
+ C_3\sum_{p+q\leq \ell} |c_{p,q}|T^{\ell +\alpha}_k,$$
for all $y\in I_k$.
Plugging $y= z_{j,k+\ell}$ into this equation (and recalling that $\wp'(z_{j,k+\ell}) = \psi(z_{j,k+\ell})$), 
and using (\ref{e=sumbound}) for $\wp-\wp'$ on $S_{k,k+1}$, we get:
$$\sum_{p,q} |c'_{p,q} - c_{p,q}| T^{p+q}_k \leq C_4\left(T_k^{\ell+\alpha} + \delta \sum_{p+q>\ell}|c_{p,q}|T_k^{p+q}
\sum_{p+q\leq\ell} |c_{p,q}|T^{\ell+\alpha}_k \right).$$
(cf. equation (4.11) in \cite{NicT}).

Now the proof proceeds exactly as in \cite{NicT}, and we obtain a polynomial $\overline \wp$
satisfying the conclusions of Lemma~\ref{l=convergence}.
\end{proof}

\section{Saturated sections of partially hyperbolic extensions}\label{s=bundle}
We recast Theorem~\ref{t=main}, part IV  as a more general statement about
saturated sections of partially hyperbolic extensions.

\begin{dfinition} Let $f:M\to M$ be $C^{k}$ and partially hyperbolic. 
A {\em $C^k$ partially hyperbolic extension of $f$} is a tuple $(N,\cB,\pi, F)$, where
$N$ is a $C^\infty$ manifold, $\pi:\cB\to M$ is a $C^\infty$ fiber bundle over $M$ with fiber $N$,
and $F:\cB\to \cB$ is a $C^k$, partially hyperbolic diffeomorphism satisfying:
\begin{enumerate}
\item $\pi\circ F = f\circ \pi$,
and
\item $E^c_F = T\pi^{-1}(E^c_f)$.
\end{enumerate}

We say that $(N,\cB,\pi, F)$ is {\em an $r$-bunched extension} if there exists a Riemannian metric $<\cdot,\cdot>$ on
$\cB$ and functions
$\nu, \hat\nu, \gamma$, and $\hat\gamma$ on $\cB$ satisfying (4)--(6) such that, for every $x\in M$:
$$\sup_{z\in \pi^{-1}(x)} \nu(z) < \inf_{z\in \pi^{-1}(x)} \{\gamma(z), \gamma^{r}(z)\}, \, \sup_{z\in \pi^{-1}(x)} \hat\nu(z) < \inf_{z\in \pi^{-1}(x)} \{\hat\gamma(z),\hat\gamma^{r}(z)\},$$
$$\frac{\sup_{z\in \pi^{-1}(x)} \nu(z)}{\inf_{z\in \pi^{-1}(x)}\gamma(z)} <  \inf_{z\in \pi^{-1}(x)}\hat\gamma^{r}(z), \quad  \hbox{and}\quad \frac{\sup_{z\in \pi^{-1}(x)}  \hat\nu(z)}{\inf_{z\in \pi^{-1}(x)}\hat\gamma(z)} < \inf_{z\in \pi^{-1}(x)}\gamma^{r}(z).
$$
\end{dfinition}

If $(N,\cB,\pi, F)$ is an $r$-bunched extension of $f$, then $f$ is $r$-bunched.  To see this,
we construct a Riemannian metric on $M$ in which the inequalities in (8) and (9) hold. This is achieved
by fixing a horizontal distribution $\hbox{Hor}\subset T\cB$, transverse to $\ker T\pi$ and containing
$E^u_F\oplus E^s_F$, and defining, for $v\in T_xM$, the metric $<\cdot,\cdot>'$  by
$<v_1,v_2>_x' = \sup <w_1,w_2>_z$,
where the supremum is taken over all $w_i\in T\pi^{-1}(v_i)\cap \hbox{Hor}(z)$, with $z\in \pi^{-1}(x)$.
In this metric, the $r$-bunching inequalities hold for $f$, with
$\nu(x) = \sup_{z\in \pi^{-1}(x)} \nu(z)$, $\hat\nu(x) = \sup_{z\in \pi^{-1}(x)} \hat\nu(z)$,
$\gamma(x) = \inf_{z\in \pi^{-1}(x)} \gamma(z)$, and  $\hat\gamma(x) = \inf_{z\in \pi^{-1}(x)} \hat\gamma(z)$.

If $(N,\cB,\pi, F)$ is a partially hyperbolic extension of $f$, it follows that 
$\cB\to M$ is an admissible bundle with $\W^s_{\hbox{\tiny lift}} = \W^s_F$ and
$\W^u_{\hbox{\tiny lift}} = \W^u_F$.
We say that a section $\sigma:M\to \cB$ is {\em bisaturated} if it is bisaturated with respect
to these lifted foliations (see Definition~\ref{d=satdef}).  We have the following theorem.

\begin{main}\label{t=C^rsec} Let $f:M\to M$ be $C^{k}$, partially hyperbolic and accessible, for some integer $k\geq 2$. Let $(N,\cB,\pi, F)$ be a $C^{k}$ partially hyperbolic extension of $f$ that is
$r$-bunched, for some $r<k-1$ or $r=1$.

Let $\sigma: M\to \cB$ be a bisaturated section. Then $\sigma$ is $C^{r}$.
\end{main}

\begin{remark} One might ask whether the same conclusion holds
if $\sigma$ is instead assumed to be a continuous $F$-invariant section.  The answer
is no.  De la Llave has constructed examples of an $r$-bunched extension
of a linear Anosov diffeomorphism with a continuous $F$-invariant section 
that fails to be $C^1$. What is more, this section is $C^{(1/r) - \eps}$, for
all $\eps>0$, but fails to be $C^{1/r}$ (see \cite{NT}, Theorem 4.1).

What is true is the following. Suppose that $(N,\cB,\pi, F)$ is an $r$-bunched partially hyperbolic
extension of $f$.
Then there exists a critical  H{\"o}lder exponent $\alpha_0\geq 0$ such that,
if $\sigma$ is an $F$-invariant section of $N$ that is  H{\"o}lder continuous with exponent $\alpha_0$,
then $\sigma$ is bisaturated, and hence $C^r$. The exponent $\alpha_0$
is determined by $\nu,\hat\nu$ and the norm and conorm of the action of
$TF$ on fibers of $N$.  When $F$ is an isometric extension of $f$ (as with 
abelian cocycles, or cocycles taking values in a compact Lie group), then
$\alpha_0 = 0$, and any continuous invariant section is bisaturated.
In general, if $F$ is an $r$-bunched extension, then $\alpha_0\leq 1/r$,
but it can be smaller, as is the case with isometric extensions.
The proof of these assertions is similar to the proof of Proposition~\ref{p=satprop};
see also (\cite{NT}, Theorem 2.2).
\end{remark}

\subsection{Proof of Theorem~\ref{t=main}, Part IV from Theorem~\ref{t=C^rsec}}
Suppose that $f$ is $C^k$, accessible and strongly $r$-bunched and
that $\phi$ is $C^k$, for some $k\geq 2$ and $r<k-1$ or $r=1$. Then the skew product
$f_\phi\colon M\times \RR/\ZZ\to \RR/\ZZ$ is a $C^k$,  $r$-bunched, partially
hyperbolic extension of $f$.  If $\Phi$ is a continuous solution to 
(\ref{e=livsic2}), then Proposition~\ref{p=satprop} implies
that $\Phi$ is bisaturated.  Then the map $x\mapsto (x,\Phi(x)\, (\mod 1))$ is
a bisaturated section of  $M\times \RR/\ZZ$.  Theorem~\ref{t=C^rsec} implies
that this section is $C^r$.  This implies that $\Phi$ is $C^r$.

\section{Tools for the proof of Theorem~\ref{t=C^rsec}}

We finally delve into the details of the proof of Theorem~\ref{t=C^rsec}, which is the heart of this paper.

\subsection{Fake invariant foliations}\label{ss=fake}

Recall that to prove Theorem~\ref{t=main}, part IV, when $f$ is dynamically coherent,
one can make use of the stable and unstable holonomy maps for $f$ and $F$ between center manifolds;
more generally this strategy can be used to prove Theorem~\ref{t=C^rsec} when $f$ is dynamically coherent.
Since we do not assume that $f$ is dynamically coherent, 
we use in place of the center foliation a locally-invariant family
of center plaques (see \cite{HPS}, Theorem 5.5). The stable holonomy between center-manifolds
is replaced by holonomy along locally-invariant, ``fake'' stable foliations, first introduced as a 
tool in \cite{BWannals}.  These foliations are defined in the next proposition.

\begin{proposition}[cf. \cite{BWannals}, Proposition 3.1]\label{p=localfol} Let $f:M\to M$ be a $C^r$
partially hyperbolic diffeomorphism. For any $\eps > 0$, there exist
constants $\rho$ and $\rho_1$ with  $\rho > \rho_1 >0$ such that, for every $p\in M$, 
the neighborhood $B_M(p,\rho)$ is foliated by
foliations $\hW^u_p$, $\hW^s_p$, $\hW^c_p$, $\hW^{cu}_p$ and $\hW^{cs}_p$ with the
following properties:
\begin{enumerate}
\item {\bf Almost tangency to invariant distributions:}  
For each $q\in B_M(p, r)$ and for each $\ast \in \{u,s, c, cu, cs\}$, the leaf $\hW^\ast_p(q)$ is $C^1$ and
the tangent space $T_q\hW^\ast_p(q)$ lies in a cone of radius
$\eps$ about $E^\ast(q)$.
\item {\bf Local invariance:} for each $q\in B_M(p, \rho_1)$ and $\ast \in \{u,s, c, cu, cs\}$,
$$f(\hW^\ast_p(q,\rho_1)) \subset 
\hW^\ast_{f(p)}(f(q)),\,\hbox{ and }f^{-1}(\hW^\ast_p(q,\rho_1)) \subset 
\hW^\ast_{f^{-1}(p)}(f^{-1}(q)).$$
\item {\bf Exponential growth bounds at local scales:} The following hold for all $n\geq 0$. 
\begin{enumerate}
\item Suppose that $q_j\in B_M(p_j,\rho_1)$ for $0 \leq j \leq n-1$.  

If
$q' \in \hW^{s}_{p}(q,\rho_1)$,
then $q_n' \in \hW^{s}_{p}(q_n,\rho_1)$, and
$$
d(q_n,q_n') \leq \nu_n(p) d(q,q').
$$
If $q_j' \in \hW^{cs}_{p}(q_j,\rho_1)$ for $0 \leq j \leq n-1$, then 
$q_n' \in \hW^{cs}_{p}(q_n)$, and
 $$d(q_n,q_n') \leq \hat\gamma_n(p)^{-1} d(q,q').$$

\item Suppose that $q_{-j}\in B_M(p_{-j},\rho_1)$ for $0 \leq j \leq n-1$. 

If $q' \in \hW^{u}_{p}(q,\rho_1)$,
then $q_{-n}' \in \hW^{u}_{p}(q_{-n},\rho_1)$, and
$$d(q_{-n},q_{-n}') \leq \hat\nu_{-n}(p)^{-1} d(q,q').$$
If $q_{-j}' \in \hW^{cu}_{p}(q_{-j},\rho_1)$ for $0 \leq j \leq n-1$, then
 $q_{-n}' \in \hW^{cu}_{p}(q_{-n})$, and
$$d(q_{-n},q_{-n}') \leq \gamma_{-n}(p) d(q,q').$$
\end{enumerate}

\item {\bf Coherence:}  $\hW^s_p$ and $\hW^c_p$ subfoliate  $\hW^{cs}_p$; 
 $\hW^u_p$ and $\hW^c_p$ subfoliate  $\hW^{cu}_p$.

\item {\bf Uniqueness:}
$\hW^s_p(p)= \W^s(p,\rho)$, and $\hW^u_p(p)= \W^u(p,\rho)$.

\item {\bf Leafwise regularity:} The following regularity statements hold:
\begin{enumerate}

\item The leaves of $\hW^u_p$ and $\hW^s_p$ are uniformly $C^r$, and for $\ast\in\{u,s\}$,
the leaf $\hW^\ast_p(x)$ depends continuously in the $C^r$ topology on the pair $(p,x)\in M\times B_M(p,\rho_1)$.
 
\item If $f$ is $r$-bunched, then the leaves of $\hW^{cu}_p$, $\hW^{cs}_p$ and $\hW^{c}_p$ are uniformly $C^r$, and for $\ast\in\{cu,cs, c\}$,
the leaf $\hW^\ast_p(x)$ depends continuously in the $C^r$ topology on $(p,x)\in M\times B_M(p,\rho_1)$.

\end{enumerate} 

\item{\bf Regularity of the strong foliation inside weak leaves:}
If $f$ is $C^k$ and $r$-bunched, for some $r<k-1$ or $r=1$, and $k\geq 2$,
then each leaf of $\hW^{cs}_p$ is $C^r$ foliated by leaves of the foliation $\hW^s_p$, and each leaf of $\hW^{cu}_p$ is 
$C^r$ foliated by leaves of the foliation $\hW^u_p$.  

Furthermore, the distribution $\hE^s_p(x)= T_x\hW^s_p$ is $C^r$ in $x\in \hW^{cs}_p(p)$,  and the map $x\mapsto \hE^s_p(x)$ on $\hW^{cs}_p(p)$
depends continuously on $p\in M$ in the $C^r$ topology. The distribution $\hE^u_p(x)= T_x\hW^u_p$ is $C^r$ in $x\in \hW^{cu}_p(p)$,  and the map $x\mapsto \hE^u_p(x)$ on $\hW^{cu}_p(p)$
depends continuously on $p\in M$ in the $C^r$ topology.

\end{enumerate}

\end{proposition}

\begin{proof} The proof of parts (1)--(5) is contained in \cite{BWannals}. We review the proof
there, as we will use the same method to prove parts (6) and (7). Some of the discussion 
below is taken from \cite{BWannals}.

Suppose that $f$ is $C^r$, for some $r\geq 1$. After possibly reducing $\eps$, we can assume that inequalities
     (\ref{e=defn})--(6)  
hold for unit vectors in the $\eps$-cones around the spaces in the partially hyperbolic splitting.

The construction is performed in two steps. The first step is to construct 
foliations of each tangent space $T_pM$. In the second step, we use the
exponential map $\exp_p$ to
project these foliations from a neighborhood of the origin in $T_pM$ to 
a neighborhood of $p$.

\medskip  

\noindent{\bf Step 1.} \
In the first step of the proof, we choose $\rho_0>0$ such that $\exp_p^{-1}$ is defined
on $B_M(p,2\rho_0)$.  For $\rho\in (0,\rho_0]$, we define,
in the standard way, a continuous map ${\mathbf f}_\rho\colon TM\to TM$
covering $f$, which is uniformly $C^r$ on fibers, satisfying:
\begin{enumerate}
\item ${\mathbf f}_\rho(p,v) = \exp_{f(p)}^{-1}\circ f \circ \exp_p(v)$, for $\|v\|\leq \rho$;                                                           
\item ${\mathbf f}_\rho(p,v) = T_pf(v)$, for $\|v\| \geq 2\rho$; 
\item  
$\|{\mathbf f}_\rho(p,\cdot) - T_p f(\cdot)\|_{C^1}\to 0$ as $\rho\to 0$, uniformly in $p$;
\item $p\mapsto {\mathbf f}_\rho(p,\cdot)$ is continuous in the $C^r$ topology.
\end{enumerate}

Endowing $M$ with the discrete topology,
we regard $TM$ as the disjoint union of its fibers. 
if $\rho$ is small enough, then ${\mathbf f}_\rho$
is partially hyperbolic, and each bundle in the partially
hyperbolic splitting for ${\mathbf f}_\rho$ at $v\in T_pM$ lies within
the $\eps/2$-cone about the corresponding subspace of $T_pM$
in the  partially
hyperbolic splitting for $f$
at $p$ (we are making the usual identification of $T_vT_pM$ with $T_pM$).
If $\rho$ is small enough, the equivalents  of inequalities (\ref{e=defn})
will hold with $Tf$ replaced by $T{\mathbf f}_\rho$. Further,
if $f$ is $r$-bunched, then ${\mathbf f}_\rho$ will also be $r$-bunched,
for $\rho$ sufficiently small.

If $\rho$ is sufficiently small, 
standard graph transform arguments give stable, unstable, center-stable,
and center-unstable foliations for
${\mathbf f}_\rho$ inside
each $T_pM$. These foliations are uniquely determined by 
the extension ${\mathbf f}_\rho$.
and the requirement that their leaves be graphs of functions with bounded derivative.
We obtain a center foliation by intersecting the leaves of the
center-unstable and center-stable foliations. 
Since the restriction of ${\mathbf f}_\rho$ to $T_pM$ depends
continuously in the $C^r$ topology on $p$, the foliations
of $T_pM$ depend continuously on $p$.

The uniqueness of the stable and unstable foliations implies, via a standard
argument (see, e.g. \cite{HPS}, Theorem 6.1 (e)), that
the stable foliation subfoliates the center-stable, and the 
unstable subfoliates the center-unstable. 

We now discuss the regularity properties of these foliations
of $TM$.  Recall the standard 
method for determining the regularity of 
invariant bundles and foliations. 
\begin{theorem}[cf. $C^r$ Section Theorem (\cite{HPS}, Theorem 3.2)]\label{t=crsect} 
Let $X$ be a $C^r$ manifold, let $\pi\colon E\to X$ be a $C^r$ Finslered Banach
bundle, and let $g\colon E\to E$ be a $C^r$ bundle map covering the $C^r$ diffeomorphism $h:X\to X$.
Assume that the image of the $0$-section under $g$ is bounded.

Assume that for every $x\in X$ there exists a constant $\kappa_x$  
such that $$\sup_{x\in X}|\kappa_x|<1,$$ and for every $y,y'\in \pi^{-1}(x)$,
$\|g(y) - g(y') \|_{\pi^{-1}(h(x))} \leq \kappa_x \|y-y'\|_{\pi^{-1}(x)}$.
Then there is a unique bounded  section $\sigma\colon X\to M$ such 
that $g(\sigma(X)) = \sigma(X)$, and $\sigma$ is continuous.

Moreover, if 
$$\sup_{x\in X} \frac{\kappa_x}{\lambda_x^r} < 1,\quad\quad\hbox{where}\quad \lambda_x = m(T_xh)$$
then $\sigma$ is $C^r$.
\end{theorem}

This theorem is used to prove the $C^r$ regularity of
the stable and unstable foliations for a $C^r$ partially hyperbolic
diffeomorphism $f$, once the $C^1$ regularity has been established
(via Lipschitz jets, or some similar method).
We review this argument, as it is prototypical.

Assume that the leaves of $\cW^u$ are $C^1$.  Note that since
the leaves of $\cW^u$ are tangent to the continuous distribution
$E^u$, this automatically implies that the map $x\mapsto \cW^u(x)$
is continuous in the $C^1$ topology.

To prove that the leaves of $\cW^u$ are uniformly $C^r$ for $r>1$, one fixes a $C^\infty$ approximation 
$TM=\widetilde E^s\oplus \widetilde E^c \oplus \widetilde E^u$ to the partially
hyperbolic splitting. One then takes the $C^1$ manifold $X$  to be the disjoint 
union of the leaves of the unstable foliation and
the fiber of the bundle $E$ over $x$ to be the space
$L_x(\widetilde E^u, \widetilde E^{cs})$
of linear maps from $\widetilde E^u(x)$ to $\widetilde E^{cs}(x)$.
The linear graph transform on the bundle $E$ covers
the original partially hyperbolic diffeomorphism 
$f\vert_{X}$, contracts the fiber over $x$ by 
$\kappa_x = \|T_xf\vert_{E^{cs}}\|/m(T_xf\vert_{E^u})<1$, and
expands $X$ at $x$ by at least $\lambda_x = m(T_xf\vert_{E^u}) > 1$.

Since the ratio 
$$\frac{\kappa_x}{\lambda_x} =  \frac{\|T_xf\vert_{E^{cs}}\|/m(T_xf\vert_{E^u})}{m(T_xf\vert_{E^u})}$$
is bounded away from $1$, Theorem~\ref{t=crsect} implies that
the unique invariant bounded section of $\sigma\colon X\to E$ is $C^1$.  But at the
point $x\in X$, the graph of the map $\sigma(x) \colon \widetilde E^u(x) \to \widetilde E^{cs}(x) $
is precisely the bundle $E^u(x)$. Since $E^u$ is $C^1$ along $X$,
the manifold $X$ is $C^2$.  

Repeating the argument, using $2$-jets
of maps from $\widetilde{E}^u$ to $\widetilde{E}^{cs}$
instead of $1$-jets, shows that $X$ is $C^3$.
An inductive argument using the $\ell-1$ jet bundle shows that $X$ is $C^\ell$,
for every integer $\ell\leq r$ To obtain that $X$ is $C^r$,
one applies Theorem~\ref{t=crsect} in its H{\"o}lder formulation
to show that the $\lfloor r \rfloor$ jet bundle is $C^{r - \lfloor r \rfloor}$.
The leaves of $\cW^u$ vary continuously in the $C^r$ topology
because the jets of $E^u$ along $\cW^u(x)$ are found as the
fixed point of a fiberwise contraction that depends continuously on $x$.
This fiberwise contraction preserves sections that depend
continuously on $x$, and so the invariant section depends continuously on
$x$ as well.

Returning to the map ${\mathbf f}_\rho$, we see that the
stable and unstable foliations for this map have uniformly
$C^r$ leaves, and for each $p\in M$ the leaves vary continuously inside of
$T_pM$ in the $C^r$ topology.  Moreover, since $p\mapsto {\mathbf f}_\rho(p,\cdot)$
is continuous in the $C^r$ topology the 
leaves of unstable foliation for ${\mathbf f}_\rho$ also depend continuously on $p$ in the $C^r$ topology.

When $f$ is $r$-bunched, a similar argument shows that the center-stable, center-unstable and
center leaves for ${\mathbf f}_\rho$ are uniformly $C^r$.  The condition
$\hat\nu < \hat\gamma^r$ is an $r$-normal hyperbolicity condition for the
center-unstable foliation, which implies that the leaves of this foliation
are uniformly $C^r$ (see Corollary 6.6 in \cite{HPS}).  In this application
of Theorem~\ref{t=crsect}, the base manifold $X$ is the disjoint
union of center-unstable manifolds, and the bundle $E$ consists
of jets of maps between the approximate center-unstable and approximate stable bundles.
The fiber contraction on $\ell-1$-jets is $\kappa = \hat\nu/\hat\gamma^{\ell-1}$
and the base conorm of the bundle map on $X$ is $\lambda = \hat\gamma$.
The condition $\kappa/\lambda = \hat\nu/\hat\gamma^{\ell} < 1$ implies that the invariant
section on $\ell-1$ jets is $C^1$, and so the center unstable
leaves are $C^\ell$, for all $\ell <r$.  As above, one obtains that the center-unstable
leaves are uniformly $C^r$.

Similarly the condition $\nu<\gamma^r$
implies that the leaves of the center-stable foliation are uniformly $C^r$;
intersecting center-unstable with center-stable leaves, one
obtains that the leaves of the center foliation for
${\mathbf f}_\rho$ are uniformly $C^r$.  
The leaves of the center, center-stable and center-unstable foliations for ${\mathbf f}_\rho$ 
along $T_pM$ also depend continuously on $p\in M$ in the $C^r$ topology.

When $k\geq 2$, and $f$ is $r$-bunched, for $r<k-1$ or $r=1$, another argument using Theorem~\ref{t=crsect}
proves the $C^r$ regularity of the unstable bundle along the leaves of
the center-unstable foliation.  The manifold $X$ is the disjoint
union of the leaves of the center-unstable foliation for ${\mathbf f}_\rho$, 
and the bundle $E$ consists of linear maps from the approximate
unstable into the approximate center-stable bundles.
Note that $X$ is uniformly $C^r$ by the previous arguments,
and the first $\lfloor r \rfloor$ derivatives of
${\mathbf f}_\rho$ vary $(r-\lfloor r \rfloor)$-H{\"o}lder continuously
from leaf to leaf. Since $X$ and $E$ are $C^r$,
we may apply the $C^r$ section theorem directly (without inductive arguments).

In this case, the graph transform bundle map has fiber constant 
$\kappa = \hat\nu/\hat\gamma$
and the base conorm $\lambda$ of ${\mathbf f}_\rho$ restricted to center-unstable
leaves is bounded by $\gamma$.
The $r$-bunching hypothesis $\hat\nu < \hat\gamma\gamma^r$ implies that
$\kappa/\lambda^r < 1$, and so 
the unstable bundle for ${\mathbf f}_\rho$ is $C^r$ when restricted to $X$.
Moreover the jets of the unstable bundle along the center-unstable
leaf vary $(r-\lfloor r \rfloor)$-H{\"o}lder continuously. Notice
that we need $k-1\geq r$ to carry out this argument, because the
bundle map we consider is only $C^{k-1}$ (in the fiber it
is a linear graph transform determined by the derivative 
of ${\mathbf f}_\rho$, and we lose a derivative in this argument).

Similarly, this argument shows
that the  bunching hypothesis $\nu < \gamma\hat\gamma^r$ implies that
the stable bundle for ${\mathbf f}_\rho$ is a $C^r$ bundle over the leaves of the
center-stable foliation, and we have (H{\"o}lder) continuous dependence of
the appropriate jets on the basepoint.  The details are described in 
\cite{PSW, PSWc} in the case $r=1$ and $k=2$.  The argument for
general $r,k$ is completely analogous.

\medskip

\noindent{\bf Step 2.} \
We now have foliations of $T_pM$, for
each $p\in M$.  We obtain the foliations 
$\hW_p^u, \hW_p^c, \hW_p^s, \hW_p^{cu}$, and $\hW_p^{cs}$ by applying the 
exponential map $\exp_p$ to the corresponding foliations of $T_pM$ inside the ball around the origin of radius $\rho$.

If $\rho$ is sufficiently small, then
the distribution $E^\ast(q)$ 
lies within the angular $\eps/2$-cone about the parallel translate of
$E^\ast(p)$, for every $\ast \in \{u,s, c, cu, cs\}$ and all $p, q$
with  $d(p,q)\leq \rho$.  Combining this fact with the preceding discussion, we obtain that property 1. holds if $\rho$ is sufficiently small.

Property 2. --- local invariance --- follows from invariance under ${\mathbf f}_\rho$ of the foliations of $TM$ and the fact that $\exp_{f(p)}({\mathbf f}_\rho(p,v)) = f(\exp_p(p,v))$ provided $\|v\| \leq \rho$.

Having chosen $\rho$, we now choose $\rho_1$ small enough so that 
$f(B_M(p,2\rho_1)) \subset B_M(f(p),\rho)$ and 
$f^{-1}(B_M(p,2\rho_1)) \subset B_M(f^{-1}(p),\rho)$, and so that, for all
$q\in B_M(p,\rho_1)$,
\begin{eqnarray*}
q' \in \hW_p^s(q,\rho_1)&\Longrightarrow& d(f(q),f(q')) \leq \nu(p)\, d(q,q'),\\
q' \in \hW_p^u(q,\rho_1)&\Longrightarrow&d(f^{-1}(q),f^{-1}(q')) \leq \hat\nu(f^{-1}(p))\, d(q,q'),\\
q' \in \hW_p^{cs}(q,\rho_1)&\Longrightarrow& d(f(q),f(q')) \leq \hat\gamma(p)^{-1} \,d(q,q'),\qquad and\\
q' \in \hW^{cu}_{p}(q,\rho_1)&\Longrightarrow& d(f^{-1}(q),f^{-1}(q')) \leq \gamma(f^{-1}(p))^{-1}\, d(q,q').
\end{eqnarray*}

Property 3. --- exponential growth bounds at local scales --- is now
proved by an inductive argument.

Properties 4.-- 7. --- coherence, uniqueness, leafwise regularity and regularity of the strong foliation inside weak leaves --- follow immediately from the corresponding properties of the foliations of $TM$ discussed above.
\end{proof}

Since there is no ambiguity in doing so, we write $\hW^{cs}(x), \hW^{cu}(x),$ and $\hW^c(x)$ for the corresponding  manifolds $\hW^{cs}_x(x), \hW^{cu}_x(x),$ and $\hW^c_x(x)$.
If $f$ is $C^k$ and $r$-bunched, for $k\geq 2$ and $r<k-1$ or $r=1$, 
then the collection of all $\hW^{\ast}(x)$-manifolds forms a uniformly continuous $C^r$ 
plaque family in $M$, but not in general a foliation.

Henceforth {\em we shall assume that $\cB$ is the trivial bundle $\cB=M\times N$.}
All of the definitions and arguments that follow can be made for a general bundle $\cB$ by
fixing a connection on $\cB$, at the expense of more cumbersome notation and the
need to localize some of the objects, such as the fake foliations for $F$ in
the following lemma.  Since Theorem~\ref{t=C^rsec} concerns the 
local property of smoothness, this simplifying  
assumption is benign.

\begin{lemma}\label{l=liftedfake} Let $k\geq 2$ and $r=1$ or $r<k-1$.
If $F$ is a $C^k$, $r$-bunched extension of $f$, then we can construct the fake foliations $\hW^s_{F,z}, \hW^u_{F,z}, \hW^{cs}_{F,z}, \hW^{cu}_{F,z}$ and $\hW^c_{F,z}$ for $F$ 
and  $\hW^s_{p}, \hW^u_{p}, \hW^{cs}_{p}, \hW^{cu}_{p}$ and $\hW^c_{p}$ for $f$ 
so that:
\begin{itemize}
\item for each $p\in M$ and $z\in \pi^{-1}(p)$, the fake foliations  $\hW^\ast_{F,z}$ for $F$ are defined in the 
entire neighborhood $\pi^{-1}(B_M(p,\rho))$ of $\pi^{-1}(p)$ and are independent of $z\in \pi^{-1}(p)$;
\item for $\ast\in \{cs,cu,c\} $, we have:
$$\hW^\ast_{F,z}(w) = \pi^{-1}\left(\hW^\ast_{p}(\pi(w))\right),$$
for all $p\in M$, all $z\in \pi^{-1}(p)$, and all $w\in \pi^{-1}(B_M(p,\rho))$; 
\item for $\ast\in \{s,u\}$, we have:
$$\pi\left(\hW^\ast_{F,z}(w)\right) = \hW^\ast_{p}(\pi(w)),$$
for all $p\in M$, all $z\in \pi^{-1}(p)$, and all $w\in \pi^{-1}(B_M(p,\rho))$; and
\item the conclusions of Proposition~\ref{p=localfol} hold for the fake foliations of $F$ and $f$. 
\end{itemize}
\end{lemma}

\begin{proof}

Let $N$ be the fiber of $\cB$.
Fix $\rho_0>0$ such that the exponential map $\exp_p$ is a diffeomorphism
from $B_{T_pM}(0,\rho_0)$ to $B_M(p,\rho_0)$, for every $p\in M$.  
Note that $\pi^{-1}(B_M(p,\rho_0))$ is a trivial bundle over $B_M(p,\rho_0)$,
for each $p\in M$.
Denote by $B_{TM}(0,\rho_0)$ the $\rho_0$-neighborhood of the $0$-section of $TM$.  
The bundle $\cB$ pulls back via the exponential map $\exp\colon B_{TM}(0,\rho_0) \to M$
to a  $C^r$ bundle $\tilde\pi_0\colon \widetilde{\cB}_0 \to B_{TM}(0,\rho_0)$ with fiber $N$.
The bundle $\widetilde{\cB}_0$ is
trivial over each fiber $\cB_{T_pM}(0,\rho_0)$ of $B_{TM}(0,\rho_0)$ and
pulls back to the original bundle $\cB$ under the inclusion
$M\hookrightarrow B_{TM}(0,\rho_0)$ of $M$ into the $0$-section of $TM$.
Elements of $\widetilde{\cB}_0$ are of the form $(p,v,z) \subset B_{TM}(0,\rho_0) \times \cB$ such that
$\pi(z) = \exp_p(v)$, and the projection $\tilde\pi_0$ sends $(p,v,z)$ to $(p,v)$.
Extend $\widetilde{\cB}_0$ to a $C^r$ bundle $\tilde\pi\colon \widetilde{\cB}\to TM$ over $TM$
in such a way that  $\widetilde{\cB}$ is also a $C^r$ bundle over $M$ (with fiber $\RR^m\times N$), 
and the restriction of $\widetilde{\cB}$ to $T_pM$ is a trivial bundle, for every $p\in M$.

In the proof of Proposition~\ref{p=localfol}, we define  ${\mathbf F}_r$ slightly differently,
using the bundle $\widetilde{\cB}$, rather than $T\cB$. Fix $\rho_1<\rho_0$ such 
that $\overline{f(B_{M}(p,\rho_1))} \subset B_M(f(p),\rho_0)$, for all $p\in M$.
Let ${\mathbf f}\colon B_{TM}(0,\rho_1)\to B_{TM}(0,\rho_0)$ be the map:
$${\mathbf f}(p,v) = \exp_{f(p)}^{-1}\circ f\circ \exp_p(v).$$
The map $F:\cB\to \cB$ induces a map ${\mathbf F}\colon \tilde\pi^{-1}(B_{TM}(0,\rho_1))\to \tilde\pi^{-1}(B_{TM}(0,\rho_0))$,
covering ${\mathbf f}$, defined by:
$${\mathbf F}(p,v,z) = ({\mathbf f}(p,v), F(z)).$$

Since $\widetilde{\cB}\vert_{TM}$ is a trivial bundle,  
we can write elements of $\tilde\pi^{-1}(T_pM)$ as triples
$(p,v,y)$, where $v\in T_pM$ and $y\in \pi^{-1}(p) \cong N$;
we can choose this trivialization to depend smoothly on $p$.
We also metrically trivialize the fibers  $\widetilde{\cB}\vert_{T_pM}$  of this bundle,  
using the product of  the sup metric $<\cdot,\cdot>'_p$ on $T_pM$ defined 
at the beginning of this section
with the induced metric $<\cdot,\cdot>$
on the fiber $\pi^{-1}(p)$.  If $F$ is an $r$-bunched extension
of $f$, then the $r$-bunching inequalities
hold for this family of metrics on $\widetilde{\cB}\vert_{B_{TM}(0,\rho)}$,
if $\rho$ is sufficiently small.

Then for each $\rho>0$ there exists a $C^r$ bundle isomorphism 
$${\mathbf F}_\rho\colon \widetilde{\cB}\to \widetilde{\cB},$$
covering the map ${\mathbf f}_\rho\colon TM\to TM$ 
constructed in the proof of Proposition~\ref{p=localfol}, with the following properties:
\begin{itemize}
\item  ${\mathbf F}_\rho(p,v,y) = {\mathbf F}(p,v,y)$ if $\|v\|\leq \rho$; in particular, we have
 ${\mathbf F}_\rho(p,0,y) = (f(p),0,F(y))$,
\item   ${\mathbf F}_\rho(p,v,y) =    (f(p), T_pf(v), {\mathbf F}(p,\rho v/\|v\|, y))$  if $\|v\|\geq 2\rho$,
\item $\sup_{v\in T_pM} d_{C^r}({\mathbf F}_\rho(p,v,\cdot), {\mathbf F}(p,0,\cdot)) \to 0$ as $\rho\to 0$,
\item the $C^r$ diffeomorphism ${\mathbf F}_\rho(p,\cdot,\cdot)$  depends continuously on $p$ in the $C^r$ topology. 
\end{itemize}

The construction of ${\mathbf F}_\rho$ is straightforward, once one has proven the following lemma,
and we omit the details.
\begin{lemma}\label{l=extend}
Let $N$ be a compact manifold and let $\{F_v\colon N\to N\}_{v\in B_{\RR^n}(0,2)}$ be a family 
of diffeomorphisms of $N$ such that $(v,y)\mapsto F_v(y)$ is $C^r$.

Then for every $\rho\in (0,1)$, there exists a family $\{F_{\rho,v}\colon N\to N\}_{v\in B_{\RR^m}(0,\rho)}$ 
of diffeomorphisms with the following properties:
\begin{itemize}
\item  $(v,y)\mapsto F_{\rho,v}(y)$ is $C^r$;
\item $F_{\rho,v} = F_v$, if $\|v\|\leq \rho$;
\item   $F_{\rho,v}  =    F_{\rho v/\|v\|} $,  if $\|v\|\geq 2\rho$; and
\item $\sup_{v\in \RR^n} d_{C^r}(F_{\rho,v} , F_0 ) \to 0$ as $\rho\to 0$.
\end{itemize}
\end{lemma}
\begin{proof}[Proof of Lemma~\ref{l=extend}]  We construct $F_{\rho,v}$ as follows. 
Consider the family of vector fields $\{X_v\}_{v\in B_{\RR^m}(0,2)}$ on $N$ defined by
$$ X_v(y) = \frac{d}{dt}\vert_{t=0} F_{v+tv}(y),$$
and let $\varphi_{v,t}$ be the flow generated by $X_v$.
For $v\in \RR^n$, let $v_\rho =\rho v/\|v\|$.  

For $\rho\in (0,1)$, let $\beta_\rho\colon\RR^m\to [0,1]$
be a  smooth radial bump function vanishing outside of $B_{\RR^m}(0, 2\rho)$ and
identically $1$ on $\overline{B_{\RR}(0, \rho)}$ with derivative $|D\beta_\rho|$ bounded by
$3\rho$.   We then define:
$$F_{\rho,v} = \begin{cases} 
F_v &\hbox{if} \quad \|v\|\leq \rho\\
\varphi_{v_\rho, \beta(v) (\|v\|-\rho)}\circ F_{v_\rho}&\hbox{if}\quad  \|v\| > \rho.
\end{cases}
$$
Then the family $\{F_{\rho,v}\}_{v\in \RR^m}$ has the desired properties.
\end{proof}

Having constructed ${\mathbf F}_\rho$, the proof 
then proceeds exactly as in Proposition~\ref{p=localfol}, except to construct the fake foliations
for $F$, we consider the bundle $\widetilde{\cB}$ over $M$ (rather than $TM$ over $M$) 
and take the disjoint union of its fibers. For $\rho$ sufficiently small, ${\mathbf F}_\rho$ is
partially hyperbolic and $r$-bunched, if $F$ is an $r$-bunched extension of $f$.
The fake foliations for $F$ are constructed by first finding invariant foliations for 
${\mathbf F}_\rho$ on $\widetilde{\cB}$.  One verifies as in Proposition~\ref{p=localfol}
that these foliations have the required regularity properties.  To construct
the fake foliations for $F$, we first restrict these foliations to the bundle
$\tilde\pi^{-1}(B_{TM}(0,\rho)) \subset \widetilde{\cB}$.  Fix $p\in M$.
On $\tilde\pi^{-1}(B_{T_pM}(0,\rho))$, the projection
$(p,v,z)\mapsto z$ is a diffeomorphism onto $\pi^{-1}(B_{M}(p,\rho))$;
the image of the invariant foliations for ${\mathbf F}_\rho$ under
this projection gives the fake invariant foliations for $F$ on  $\pi^{-1}(B_{M}(p,\rho))$.

To construct the fake invariant foliations for $f$, we take instead the image of
the invariant foliations for ${\mathbf F}_\rho$ in $\tilde\pi^{-1}(B_{T_pM}(0,\rho))$
under the map $(p,v,z)\mapsto \exp_p(v)$. This construction ensures that the desired properties hold.
\end{proof}

Fix $\eps>0$ small and let the fake foliations for $f$ and $F$ be defined by the preceding
lemmas.

Since it does not depend on $z\in\pi^{-1}(p)$ we write $\hW^{\ast}_{F,p}(w)$ for
$\hW^{\ast}_{F,z}(w)$, for $\ast\in\{s,u,cs,cu,c\}$. As
with the fake foliations for $f$, for  $\ast\in\{cs,cu,c\}$ and $p\in M$, we will denote by
$\hW^\ast_F(p)$ the plaque $\hW^\ast_F(p) = \pi^{-1}(\hW^\ast(p))$ in $\cB$; it is the 
$\hW^\ast_{F}$-leaf through any $z\in \pi^{-1}(p)$.
 
By rescaling the Riemannian metric on $M$, we may assume that $\rho_1\gg 1$,
so that all of the objects used in the sequel are well-defined on any ball of
radius $1$ in $M$.

\subsection{Further consequences of $r$-bunching}\label{ss=further}

Here we explore in greater depth the properties of an $r$-bunched partially hyperbolic
diffeomorphism.  The goal is to bound the deviation between the fake foliations $\hW^\ast_p$ and $\hW^\ast_q$ for $q\in \hW^\ast(p)$.  In the dynamically coherent case, $\hW^\ast_p(q)$ and $\hW^\ast_q(q)$
coincide for $q\in \hW^\ast(p)$.  In a sense, the results in this section
tell us that $r$-bunched systems are dynamically coherent ``on the level of $r$-jets.''

Throughout this and the following subsections, 
we continue to assume that $F$ is a $C^k$, $r$-bunched extension of $f$, where
$k\geq 2$ and $r<k-1$ or $r=1$. In the statements of some of the lemmas, we will remind the reader
of these hypotheses. We fix as above a choice of fake foliations and fake lifted foliations (we will
not specify here the choice of $\eps>0$, but will indicate where it is relevant).
Let $m=\dim(M)$, $s=\dim{E^s}$, $u=\dim{E^u}$, and $c=\dim{E^c}$,
so that $m=s+u+c$. 

Fix a point $p\in M$.  We introduce $C^r$ local $\RR^u\times \RR^s\times \RR^c$ - coordinates  $(x^u,x^s,x^c)$ in 
the $\rho$-neighborhood of 
$p$, sending $p$ to $0$, $\hW^{cs}(p)$ into the subspace $x^u=0$,  $\hW^{cu}(p)$ into  $x^s=0$,
$\W^s(p)$ to $x^u=x^c=0$, $\W^s(p)$ to $x^u=x^c=0$, and $\W^u(p)$ to  $x^s=x^c=0$.  This is possible
because all of the submanifolds in question are $C^r$.  Since $\hW^u_p$ is a $C^r$ subfoliation of
$\hW^{cu}(p)$, and $\hW^s_p$ is a $C^r$ subfoliation of $\hW^{cs}(p)$, we may also choose these coordinates
so that each leaf $\hW^u_p(q)$, for $q\in \hW^{cu}(p)$ is sent into an affine space 
$x^s=0, x^c\equiv x^c_0$ and each leaf $\hW^s_p(q')$, for $q'\in \hW^{cs}(p)$ is sent into an affine space 
$x^u=0, x^c\equiv {x^c_0}'$.

\begin{figure}[h]
\psfrag{hW^s}{$\hW^u_p$}
\psfrag{hW^u}{$\hW^s_p$}
\psfrag{hW^c}{$\hW^c(p)$}
\psfrag{W^s}{$\W^u(p)$}
\psfrag{W^u}{$\W^s(p)$}
\psfrag{x^u}{$x^s$}
\psfrag{x^c}{$x^c$}
\psfrag{x^s}{$x^u$}
\begin{center}
\includegraphics[scale=1.0]{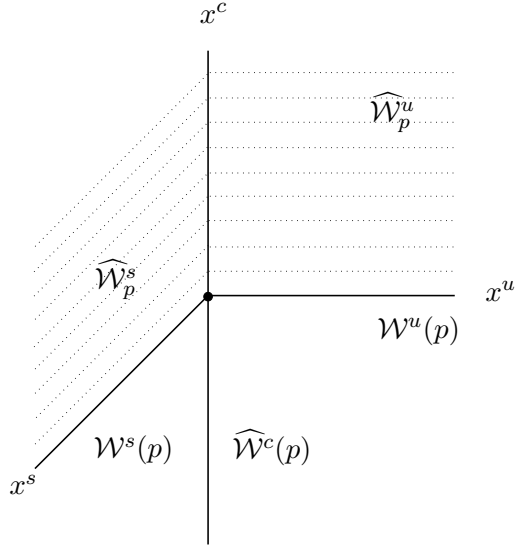}
\caption{Coordinates adapted to the fake foliations at $p$.}
\end{center}
\end{figure}
We can choose these coordinates to depend uniformly on $p$.  We call these coordinates {\em adapted coordinates at $p$}. Whenever we refer to adapted coordinates at a point $p$, we implicitly assume that they are chosen with a uniform
bound on their $C^r$ size.

According to Proposition~\ref{p=localfol} the leaves of the fake center,
center-stable  and center-unstable manifolds  
at each point $z$ can be expressed using parametrized $C^r$ plaque families:
$$\hat\omega^{cs}\colon I^m\times I^{c+s} \to \RR^m,\qquad \hat\omega^{cu}\colon I^m\times I^{c+u} \to \RR^m, $$
and
$$\hat\omega^{c}\colon I^m\times I^{c} \to \RR^m,$$
where $\hW^{cu}(z) = \hat\omega^{cu}_z(I^{c+u})$, $\hW^{cs}(z) = \hat\omega^{cs}_z(I^{c+s})$
and $\hW^{c}(z) = \hat\omega^{c}_z(I^{c})$.
The map $\hat\omega^c$ is obtained from $\hat\omega^{cs}$ and $\hat\omega^{cu}$ using the implicit function theorem.
We may assume these maps take the form:
$$\hat\omega^{cs}_z(x^{c},x^{s}) = z+ (\hat\beta^{cs}_z(x^c,x^s),x^s, x^c)\qquad \hat\omega^{cu}_z(x^{c},x^{u}) = z+ (\hat\beta^{cu}_z(x^u, x^c),x^u, x^c),$$ and  
$$\hat\omega^{c}_z(x^{c}) = z+ (\hat\beta^{c}(x^c),x^c),$$
where $\hat\beta_z^{cu}\in C^r(I^{c+u},\RR^s)$, $\hat\beta_z^{cs}\in C^r(I^{c+s},\RR^u)$,
and $\hat\beta_z^{c}\in C^r(I^{c},\RR^{s+u})$,
and $z\mapsto \hat\beta^{\ast}_z$ is continuous in the $C^r$ topology.  
Moreover, we have $\hat\beta^\ast_z(0) = 0$ and
$\hat\omega^\ast_0 \equiv 0$ for $\ast\in\{cs, cu,c\}$.

We now derive further consequences of the $r$-bunching hypothesis on $f$.
The first concerns the behavior of the plaque families $\hW^{\ast}(y)$
for $y\in \hW^\ast(x)$, for $\ast\in\{cs,cu,c\}$.  

\begin{lemma}\label{l=cspinch} For each $v= (0,v^s,v^c)\in\hW^{cs}(0)$, $w = (w^u,0, w^c)\in\hW^{cu}(0)$, and $z= (0,0,z^c)\in\hW^c(0)$, and for every positive integer $\ell\leq r$,
we have:
$$|j_{0}^\ell \hat\beta^{cs}_v| = o(|v^c|^{r-\ell}),
\quad  |j_{0}^\ell \hat\beta^{cu}_w| = o(|w^c|^{r-\ell}),\quad\hbox{and}\quad
|j_{0}^\ell \hat\beta^{c}_z| = o(|z^c|^{r-\ell}).$$
All of these statements hold uniformly in the coordinate system based at $p$.
\end{lemma}

\begin{proof} We prove the assertion for $\hat\beta^{cu}$; the argument for $\hat\beta^{cs}$ is the same
but with $f$ replaced by $f^{-1}$.  The assertion for $\hat\beta^c$ follows from the first two.

As in the proof of Proposition~\ref{p=Holder}, 
we will  use the  convention that if $q\in M$ and $j\in \ZZ$,
then $q_j$ denotes the point $f^j(q)$, with $q_0=q$.
For a positive function $\alpha\colon M\to \RR_+$ we also use the cocycle notation described there.

Endow the disjoint union $\hat M_p = \bigsqcup_{n\geq 0} B(p_{-n},\rho)$ with
the $C^r$ adapted coordinate system based at $p_{-n}$ in the ball $B(p_{-n},\rho)$.
We thereby identify $\hat M_p$ with the disjoint union
$\bigsqcup_{n\geq 0} (I^m)_{-n}$.  This coordinate system is
not invariant under $f$, but certain aspects of it are;
in particular, the planes $x^u=0$ and $x^s=0$ are invariant,
as are the families $x^u=0, x^c\equiv x_0^c$ and $x^s=0, x^c\equiv x_0^c$.
Moreover, we may assume (having chosen $\eps>0$ small enough in the application
of Proposition~\ref{p=localfol}) that for any point of the form $(0, x^s, x^c)\in  B(p_{i},\rho)$,
writing $f(0, x^s, x^c) = (0, x^s_1, x^c_1)$, we have that
$|x^s_1| \leq \nu(p_i)|x^s|$ and  $\gamma(p_i)|x^c|\leq |x^c_1| \leq \hat\gamma(p_i)^{-1}|x^c|$.
Similarly for any point of the form $(x^u, 0, x^c)\in  B(p_{i+1},\rho)$,
writing $f^{-1}(x^u, 0, x^c) = (x^u_{-1}, 0 , x^c_{-1})$, we have that
$|x^u_{-1}| \leq \hat\nu(p_i)|x^u|$ and  $\hat\gamma(p_i)|x^c|\leq |x^c_{-1}| \leq \gamma(p_i)^{-1}|x^c|$.

Let $\hat M_p(1) = \bigsqcup_{n\geq 1} B(p_{-n},1)$,
and note that $f(\hat M_p(1)) \subset \hat M_p$. 
Let $\varphi$ be the change of coordinate $\varphi(x^u,x^s,x^c) = (x^c, x^u, x^s)$,
and let $\tilde{f} = \varphi\circ f \circ \varphi^{-1}$.
Now write, for $x\in \hat M_p(1)$: 
$$D\tilde{f}(x) = \left(\begin{array}{cc}
A_x & B_x\\
C_x & K_x
\end{array}\right),
$$ 
where $A_x\colon \RR^{c+u}\to \RR^{c+u}$, $B_x\colon \RR^{s}\to \RR^{c+u}$,
$C_x\colon \RR^{c+u}\to \RR^{s}$ and $K_x\colon \RR^{s}\to \RR^{s}$.
We may assume that $\eps>0$ was chosen small enough in the application of Proposition~\ref{p=localfol}
that for every $x\in f^{-1}(B(p_{-n+1},1))\cap B(p_{-n},1)$, we have
that $m(A_x) \geq \gamma(p_{-n})$ and $\|K_x\| \leq \nu(p_{-n})$ ,
and $\|B_x\|$ and $\|C_x\|$ are very small.
The partial hyperbolicity and $r$-bunching hypotheses $\nu < \gamma$ and
$\nu < \gamma^r$ then imply that, for all $\ell\leq r$:
$$\sup_{x\in \hat M_p}\max \left\{ \frac{\|A_{x}\|}{m(K_{x})}, \frac{\|K_{x}\|}{m(A_{x})^\ell}\right\} < 1.$$
Fix $0\leq \ell\leq r$, and let  $\kappa = \max\{\nu\gamma^{-\ell}, \nu\gamma^{-1}\}$.
Also fix a continuous function $\delta < \min\{1,\gamma\}$ such that
$\kappa < \delta^{r-\ell}$; this is possible since $f$ is $r$-bunched.

Consider the $C^{k-\ell}$ induced map 
$$\cT_{f}^\ell\colon \hat M_p(1) \times J_0^\ell(\RR^{c+u},\RR^{s})_0 \to \hat M_p \times J_0^\ell(\RR^{c+u},\RR^{s})_0$$
defined by:
$$\cT_{f}^\ell(x, j_0^\ell\psi) = (f(x), j_0^\ell\psi'), $$
where $\psi'\in \Gamma^\ell_0(\RR^{c+u},\RR^{s})_0$ satisfies:
$$\tilde{f}(x + \hbox{graph}(\psi)) = \tilde{f}(x) + \hbox{graph}(\psi')$$ 

Lemma~\ref{l=graphjets} implies that there is a metric $|\cdot|_L$
on $J_0^\ell(\RR^{c+u},\RR^{s})_0$
such that for all $n\geq 0$, all $x\in B(p_{-n-1},1)\subset \hat M_p(1)$
and all $j_0\psi, j_0\psi' \in  J_0^\ell(I^{c+u},\RR^{s})_0$, 
with $|j_0\psi|_L, |j_0\psi'|_L \leq 1$, we have:
\begin{eqnarray}\label{e=cscontract}
| \cT^\ell_{f}(x,j_0\psi) - \cT^\ell_{f}(x,j_0\psi) |_L \leq \kappa(p_{-n})|j_0\psi - j_0\psi'|_L.
\end{eqnarray}

Given a point $w = (w^u,0,w^c) \in \hW^{cu}(p)$, 
we choose $n\in \ZZ_+$ such that $|w^c| = \Theta(\delta_{-n}(p)^{-1})$. This is possible,
since $\delta<1$ is a continuous function (remember that $\delta_{-n}$ is the product
of reciprocal values of $\delta$, and so $\delta_{-n}(p)^{-1}$ is less than $1$).
The planes $x^s=0, x^c \equiv x^c_0$ lie in an $\eps$-cone about the center-stable
distribution for $f$. Hence under iteration by $f^{-1}$, the part of 
$x^s=0, x^c \equiv x^c_0$ that remains inside of $\hat M_p(1)$ for $n$ iterates 
is a smooth plane that remains in the $\eps$-cone about the center-stable distribution.
Write $w_{-n} = f^{-n}(w) = (w^u_{-n},0, w^c_{-n})$.
Since $|w^c| = \Theta(\delta_{-n}(p)^{-1})$ and $|w^u| = O(1)$, and
$\hat\nu < \delta\gamma < 1$, Proposition~\ref{p=localfol},
parts (1)-(3) imply that $|w^u_{-n}| = O(\hat\nu_{-n}(p)^{-1}) = o(1)$ and 
$|w^c_{-n}| = O(\delta_{-n}(p)^{-1}\gamma_{-n}(p))= o(1)$; in particular,
we have that $w_{-i}\in B(p_{-i},1)$, for $i=1,\ldots, n$.

Now consider the orbit of $(w_{-n}, j^\ell_0\hat\beta^{cu}_{w_{-n}}) \in \hat M_p(1) \times J_0^\ell(\RR^{c+u},\RR^{s})_0$ under $\cT^\ell_{f}$. Local invariance of the $\hW^{cu}_p$ plaque family
implies that $${\left(\cT^\ell_{f}\right)}^n(w_{-n},j^\ell_0\hat\beta^{cu}_{w_{-n}}) = (w, j^\ell_0\hat\beta^{cu}_w).$$
On the other hand, since $f$ leaves invariant the planes $x^s=0$, we have that
${\left(\cT^\ell_{f}\right)}^n(w_{-n},0) = (w, 0)$ .
But now (\ref{e=cscontract}) implies that 
\begin{eqnarray*}
|j^\ell_0\hat\beta^{cu}_w|_L& \leq &\kappa_{-n}(p)^{-1}|j^\ell_0\hat\beta^{cu}_{w_{-n}}|_L\\
& = &  O( \kappa_{-n}(p)^{-1})
\end{eqnarray*}
On the other hand, $\kappa < \delta^{r-\ell}$,
and $|w^c| = \Theta( \delta_{-n}(p)^{-1})$. This implies that $|j^\ell_0\hat\beta^{cu}_w| = o(|w^c|^{r-\ell})$,
completing the proof of Lemma~\ref{l=cspinch}.
\end{proof}

The next consequence of $r$-bunching we derive concerns the discrepancy between
the leaves of the real and fake stable (or unstable) foliation originating at a given point.
To state these results, we introduce a parametrization of the fake stable
and unstable foliations as follows.  We are interested in the restriction
of the fake stable foliation $\hW^{s}_x$ to the center-stable leaf $\hW^{cs}(x)$.

As above, fix an adapted coordinate system at $p$.  Proposition~\ref{p=localfol} implies

that $\hW^s_p$ is a $C^r$ subfoliation when restricted to $\hW^{cs}(p)$. We
are going to give a different parametrization of $\hW^{cs}(p)$ to reflect this fact.
Recall our definition above: 
$\hat\omega^{cs}_z(x^{c},x^{s}) = z+ (\hat\beta^{cs}_z(x^c,x^s),x^s,x^c)$, and  $\hat\omega^{cu}_z(x^{c},x^{u}) = z+ (x^u, \hat\beta^{cu}_z(x^c,x^u), x^c)$. 
Using the implicit function theorem, we can write instead:
$$\hat\omega^{cs}_z(x^{c},x^s) = z+ (\hat\beta^{s,u}_z(x^c,x^s), x^s, \hat\beta^{s,c}_z(x^c,x^s)),$$
and
$$\hat\omega^{cu}_z(x^{u},x^c) = z+ (x^u \hat\beta^{u,s}_z(x^c,x^u), \hat\beta^{u,c}_z(x^c,x^u)),$$ 
with the property that for fixed $x^c\in I^c$:
$$ \hat\omega^{cs}_z(x^c,I^s) = \hW^s_z(\hat\omega^{cs}(x^c,0)),\quad\hbox{and}\quad \hat\omega^{cu}_z(x^c,I^u) = \hW^u_z(\hat\omega^{cu}_z(x^c,0)),
$$
and such that $z\mapsto \hat\beta^s_z = (\hat\beta^{s,u}_z, \hat\beta^{s,c}_z) \in C^r(I^c\times I^s,\RR^{u+c})$ and
$z\mapsto \hat\beta^u_z = (\hat\beta^{u,s}_z, \hat\beta^{u,c}_z)\in C^r(I^c\times I^u,\RR^{s+c})$ 
are all continuous in the $C^r$ topologies.
We may further assume that $\hat\beta^{s,c}_z(x^c,0) = x^c = \hat\beta^{u,c}_z(x^c,0)$.
Our choice of coordinates also implies that
$\hat\beta^{s}_0 \equiv 0$ and $\hat\beta^{u}_0\equiv 0$.
Finally, note that $\hat\omega^{cs}_z(0,I^s) = \hW^s_z(z) = \cW^s(z,\rho)$ and 
$\hat\omega^{cu}_z(0,I^u) = \hW^u_z(z) =\cW^u(z,\rho)$.

\begin{figure}[h]
\psfrag{hW^s}{$\hW^u_p$}
\psfrag{hW^u}{$\hW^s_p$}
\psfrag{hW^c}{$\hW^c(p)$}
\psfrag{W^s}{$\W^u(p)$}
\psfrag{W^u}{$\W^s(p)$}
\psfrag{x^u}{$x^s$}
\psfrag{x^c}{$x^c$}
\psfrag{x^s}{$x^u$}
\psfrag{(0,0,z^c)}{$z=(0,0,z^c)$}
\psfrag{B}{$\hW^{cu}(z)\cap\{x^u=0\}$}
\psfrag{A}{$\hW^c(z)$}
\psfrag{E}{$\{x^u=x^u_0\}$}
\psfrag{D}{$\hW^{cu}(z)\cap\{x^u=x_0^u\}$}
\psfrag{C}{$(x^u_0, \hat\beta^u_{z}(0,x^u_0))$}
\psfrag{F}{$(0,\hat\beta^{u,s}(x^c,0),z^c+x^c)$}
\psfrag{G}{$(x^u_0, \hat\beta^u_{z}(x^c,x^u_0))$}
\psfrag{H}{$\hW^{cu}(z)$}
\begin{center}
\includegraphics[scale=0.8]{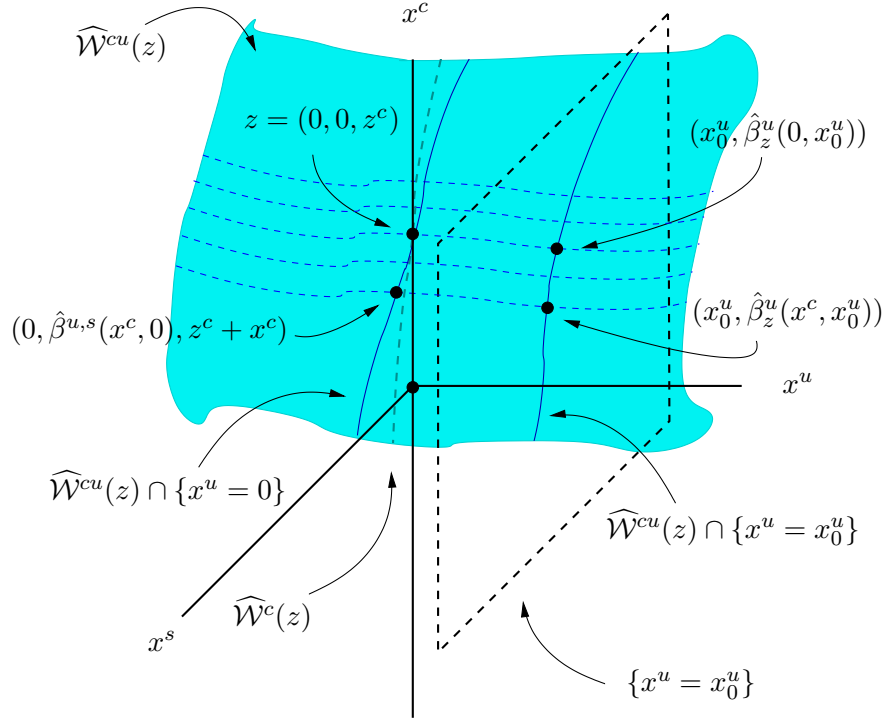}
\caption{Parametrizing the fake unstable foliations at $(0,0,z^c)$.}
\end{center}
\end{figure}

Fix $z^c\in I^c$. We are interested in the deviation between the true stable leaf
$\hat\omega^{cs}_{(0,0,z^c)}(\{0\}\times I^s)$ and the fake stable leaf
$\hat\omega^{cs}_{0}(\{z^c\} \times I^s)$; this is measured by the distance between
the functions
$\hat\beta^{s}_{(0,0,z^c)}(0,\cdot)$ and $\hat\beta^{s}_{0}(z^c,\cdot)$ at a point
$x^s\in I^s$.  We are interested not only in the $C^0$-distance between
these functions, but in the distance between their transverse jets.
By our choice of coordinate system, we have that
$\hat\beta^{s}_{0}$ is identically $0$; hence we will estimate
just the jets of $\hat\beta^{s}_{(0,0,z^c)}$ in the $x^c$ direction at $x^c=0$ and a fixed value of $x^s$.

\begin{lemma}\label{l=spinch} For $z^c\in I^c$, $x^s\in I^s$ and $x^u\in I^u$ we have:
$$ \left|j^\ell_{0}\left(x^c\mapsto \hat\beta^s_{(0,0,z^c)}(x^c,x^s)\right)\right| = |x^s|\cdot o(|z^c|^{r-\ell}),$$
and
$$\left|j^\ell_{0}\left(x^c\mapsto \hat\beta^u_{(0,0,z^c)}(x^c,x^u)\right)\right| = |x^u|\cdot o(|z^c|^{r-\ell}),$$
for every $\ell\leq r$.
\end{lemma}

\begin{remark} Consider the transversals $x^u=0$ and $x^u = x^u_0$ to the
foliations $\hW^u_0$ and $\hW^u_{(0,0,z^c)}$.  
If we restrict to the space $x^u=x^s = 0$ inside the first transversal (which corresponds
to the center manifold $\hW^c(p)$), then the holonomy map for $\hW^u_p\vert_{\hW^{cu}(p)}$ 
to the second transversal is trivial in these coordinates, sending $(0,0,x^c)$ to
$(x^u_0,0,x^c)$.  If we consider instead the holonomy map for $\hW^u_{(0,0,z^c)}\vert_{\hW^{cu}(0,0,z^c)}$
between these transversals,  then the point $(0,\hat\beta^{u,s}(x^c,0), z^c+x^c)$ is
sent to $(x^u_0,\hat\beta^{u}_{(0,0,z^c)}(x^c,x^u_0))$
The $\ell$-jet of this holonomy at $(0, 0, z^c)$ (measured in the $x^c$ coordinate) is precisely 
the quantity $j^\ell_{0}\left(x^c\mapsto \hat\beta^u_{(0,0,z^c)}(x^c,x^u_0)\right)$ estimated by Lemma~\ref{l=spinch}.
\end{remark}

\begin{proof}[Proof of Lemma~\ref{l=spinch}]
We continue to adopt the conventions and notations in the proof of Lemma~\ref{l=cspinch},
we define $\hat M_p$ and $\hat M_p(1)$ as in that proof, and use the same coordinate
system defined there. We prove the assertion for $\hat\beta^u$; the proof for $\hat\beta^s$
is the same, but with $f$ replaced by $f^{-1}$.

Denote by $f_0$ the restriction of $f$ to
$\bigsqcup_{n\geq 1} \hW^c(p_{-n})  \subset \hat M_p(1)$,
which we regard locally as a map from $I^c$ to $I^c$. We now focus
attention on a single neighborhood $B(p_{-n},1)$, for some fixed $n\geq 1$,
and regard $x^c\in I^c$ as coordinatizing $x^u=0, x^s=0$
and $(x^u,x^{s+c})\subset I^u\times I^{s+c} = I^m$ as coordinatizing
points in this neighborhood.

In local coordinates respecting the decomposition $I^m = I^u\times I^{s+c}$,
write:
$$f(x^u, x^{s+c}) = (f_u(x^u, x^{s+c}), f_{sc} (x^u, x^{s+c})).$$
In a neighborhood of each point, this map acts on graphs of $C^1$ functions from 
$I^u$ to $\RR^{s+c}$ by the
usual graph transform, which is a contraction on the fibers of 
$\pi^{1,0}\colon J^1(I^u,\RR^{s+c}) \to J^0(I^u,\RR^{s+c}) = I^u\times \RR^{s+c}$.
Unstable manifolds for $f$ are sent to unstable manifolds under this graph transform, and, locally, fake
unstable manifolds are sent to fake unstable manifolds.
For each point $(0,0,z^c)\in I^m$, we will consider a $C^\ell$ family of such $1$-jets,
expressed as a function of the coordinate $x^c$ transverse to the fake unstable
foliation in $\hW^{cu}(p_{-n}) = \{x^s = 0\}$; we study 
the variation of such graphs through points $(0,0,z^c+x^c)$ near $x^c=0$.

The space of all such $\ell$-jets of $1$-jets at the point $x^c=0$ is the bundle $J^\ell_0(J^1_{I^u}(I^u,\RR^{s+c}))$.
Elements of this ``mixed jet bundle'' are of the form $j^{\ell}_0(j^1_{x^u}\beta)$,
where $\beta(x^c,x^u)\colon I^c\times I^u\to \RR^{s+c}$ is defined in a neighborhood of
$\{0\}\times I^u$, the map $\beta(x^c,\cdot)$ is $C^1$, and the map
$x^c\mapsto j^1_{x^u}\beta(x^c,\cdot)$ is $C^\ell$.  In particular, if $\beta$ is $C^{\ell+1}$,
then this property is satisfied. We denote this space $\Gamma^\ell_0(I^c, \Gamma^1_{I^u}(I^u, \RR^{s+c}))$
of such local functions by $\Gamma^{\ell,1}_{\{0\}\times I^u}(I^c\times I^u, \RR^{s+c})$.
We also denote $j^{\ell}_0(j^1_{x^u}\beta)$ by
$j^{\ell,1}_{0,x^u}\beta$ and the bundle $J^\ell_0(J^1_{I^u}(I^u,\RR^{s+c}))$ by
$J^{\ell,1}_{\{0\}\times I^u}(I^c\times I^u,\RR^{s+c})$.

Note that in our parametrization $\hat\beta^u \colon I^m \times I^c \times I^u \to I^{s+c}$
of the fake unstable subfoliations, the set $\hat\beta^u_z(x^c,I^u)$ is the leaf
of $\hW^u_z$ through the point $\omega^{cu}_z(x^c,0) = z + (0,\hat\beta^{cu}_z(x^c))$; if
$z = (0,0,z^c)$, then the unique point of $\hW^u_z$ intersecting $x^u=0$ is of the
form $(0, x^s, z^c+x^c)$.  Because the sets $\{x^u=0, x^s = \hbox{const}\}$
are invariant under $f$ in our coordinate system, the image of the
point $(0, x^s, z^c+x^c)$ is of the form $(0, {x^s}', f_0(z^c+x^c))$.
This is the unique point on the leaf of $\hW^{u}_{f(z)}$ intersecting
$x^u=0$, which in turn lies in the set $\hat\beta^u_{f(z)}(\{f_0(z^c+x^c) - f_0(z^c)\}\times I^u)$.
We will thus define the natural action of $f$ on $I^c \times \Gamma^{\ell,1}_{\{0\}\times I^u}(I^c\times I^u, I^{s+c})$
so that it sends  $(z_0, \hat\beta^u_{(0,0,z_0)}(\{x^c\}\times I^u))$ to $(f_0(z_0), \beta_{f(0,0,z^c)}( \{f_0(z^c+x^c) - f_0(z^c)\}\times I^u)) $. 

For $(z^c, \beta) \in I^c \times \Gamma^{\ell,1}_{\{0\}\times I^u}(I^c\times I^u, \RR^{s+c})$, we would like
to define the map $\cT(z^c, \beta)\in \Gamma^{\ell,1}_{\{0\}\times I^u}(I^c\times I^u, \RR^{s+c})$ implicitly
by the equation
\begin{eqnarray}\label{e=betadef}
\cT(z^c, \beta)\left(f_0(z^c+x^c) - f_0(z^c), f_u(x^u, \beta(x^c,x^u) + (0,z^c))\right)\\
= f_{sc}(\beta(x^c,x^u) + (0,z^c)) - (0, f_0(z^c));
\end{eqnarray}
if such a map exists, then we will have:
$$ \cT(z^c, \hat\beta^u_{(0,0,z^c)}(x^c,I^u)) = \hat\beta^u_{(0,0,f_0(z^c))}(f_0(x^c+z^c) - f_0(x^c),I^u). 
$$
To check local invertibility, we must check that the map
$$g_{z^c}(x^c,x^u) = (f_0(z^c+x^c) - f_0(z^c), f_u(x^u, \beta(x^c,x^u) + (0,z^c)))$$
on $I^c\times I^u$ is invertible in a neighborhood of $(0,x^u)$.
The derivative of this map at $(0,x^u)$ is
$$Dg_{z^c}(0,x^u) = \left(\begin{array}{cc}
 Df_{0}(z^c) &  0\\
  C  &  K
\end{array}
\right),
$$
where
$$K = \frac{\partial{f_u}}{\partial x^{u}}(\beta(0,x^u)+(0,z^c)) +  \frac{\partial{f_u}}{\partial x^{s+c}}(x^u,\beta(0,x^u)+(0,z^c)) \circ \frac{\partial{\beta}}{\partial x^{u}}(0, x^u)
$$
and
$$B=  \frac{\partial{f_u}}{\partial x^{s+c}}(x^u,\beta(0,x^u)+(0,z^c))\circ \frac{\partial{\beta}}{\partial x^{c}}(0,x^u).
$$
This map invertible if  $\frac{\partial{\beta}}{\partial x^{u}}(0, x^u)$ is sufficiently small.
Let $\cT(z^c, \beta)$ be defined by (\ref{e=betadef}) on this subset. 

Next, for $0\leq \ell\leq k-1$,
consider the map $$\cT_{f}^{\ell,1}\colon I^c\times J^{\ell,1}_{\{0\}\times I^u}(I^c\times I^u,\RR^{s+c}) \to \RR^c\times J^{\ell,1}_{\{0\}\times I^u}(I^c\times I^u,\RR^{s+c}),$$
defined (in a neighborhood of the $0$-section) by
$$ \cT_{f}^{\ell,1}\left (z^c, j_{0}^{\ell} \left(j^1_{x^u}\beta \right)\right) = \left(f_0(z^c), \,j^{\ell}_0 \left(j^1_{g_{z^c}(x^c,x^u)}\cT(z^c, \beta)\right)\right).
$$

Recall that we have been working in a single coordinate neighborhood $B(p_{-n},1)$.
We combine these definitions of $ \cT_{f^{-1}}^{\ell,1}$ over all neighborhoods to define a global map 
$$\cT_{f}^{\ell,1}\colon \bigsqcup_{n\geq 1}\left(I^c\times J^{\ell,1}_{\{0\}\times I^u}(I^c\times I^u,\RR^{s+c})\right)_{-n}\qquad\qquad\qquad\qquad\qquad\qquad\qquad$$
$$\qquad\qquad\qquad\qquad\qquad\qquad\qquad\longrightarrow \bigsqcup_{n\geq 0} \left(I^c\times J^{\ell,1}_{\{0\}\times I^u}(I^c\times I^u,\RR^{s+c})\right)_{-n}$$
(where the $-n$ subscript denotes the neighborhood $B(p_{-n},\rho)$ in the disjoint union).
This map is fiberwise $C^{k-\ell-1}$ (in particular, it is $C^1$ if $\ell  < k-1$)  and 
has the property that $\cT_{f}^{\ell,1}(z, j^{\ell,1}_{(0,x^u)}\hat\beta^u_z) = (f(z), j^{\ell,1}_{g_{z^c}(0,x^u)}\hat\beta^u_{f(z)})$.

A calculation very similar to the one in the proof of Lemma~\ref{l=graphjets} shows that
there is a norm $|\cdot|_L$ on $J^{\ell,1}_{\{0\}\times I^u}(I^c\times I^u,\RR^{s+c})$ such that,
for all  $n\geq 0$,  $z^c\in I^c_{-n-1}$, $x^s\in I^s_{-n-1}$, 
and all $j^{\ell,1}_{(0,x^u)}\beta, \,j^{\ell,1}_{(0,x^u)}\beta' \in J^{\ell,1}_{\{0\}\times I^u}(I^c\times I^u,\RR^{s+c})_{-n-1}$ 
sufficiently close to the $0$-section, we have:
\begin{eqnarray}\label{e=scontract}
\left|\cT_{f}^{\ell,1}(z^c, j^{\ell,1}_{(0,x^u)}\beta) - \cT_{f}^{\ell,1}(z^c, j^{\ell,1}_{(0,x^u)}\beta')\right|_L\qquad\qquad\qquad\qquad\qquad\\ 
\qquad\qquad\qquad\qquad\qquad\qquad\qquad\leq
{\kappa(p_{-n})}\left|j^{\ell,1}_{(0,x^u)}\beta - j^{\ell,1}_{(0,x^u)}\beta' \right|_L,
\end{eqnarray}
where $\kappa = \max\{\nu/(\gamma\hat\gamma^{\ell}), \nu/(\gamma\hat\gamma)\}$.
The $r$-bunching hypothesis implies that $\kappa < 1$.

Having made these preliminary estimates, we finish the proof of Lemma \ref{l=spinch}.
Fix $0\leq \ell\leq r$ and a continuous function $\delta < \min\{1,\gamma\}$ such that:
$$\kappa < \delta^{r-\ell} \quad\hbox{and}\quad \hat\nu\hat\gamma^{-1} < \delta^r;$$
this is possible since $f$ is partially hyperbolic and $r$-bunched.
Fix a point $z^c\in I^c$ and an integer $n\geq 0$ such that $|z^c| = \Theta(\delta_{-n}(p)^{-1})$. 
Let $z=(0,0,z^c) \in I^m_0$.  By our choice of $n$, we have that for $0\leq i\leq n$, 
$|f_0^{-i}(z^c)| \leq \gamma_{-i}(p)|z^c| \leq  \gamma_{-i}(p)\Theta(\delta_{-n}(p)^{-1}) \ll 1$,
if $|z^c|$ sufficiently small (uniformly in $p$).  Thus we may assume that $z_{-i} = f^{-i}(z)\in \hat M_p(1)$, for 
$0\leq i\leq n$.

Next, fix a point $x^u_0\in I^u$, and consider the point
$w= \hat\omega^{cu}_{z}(0,x^u_0) = (x^u_0,\hat\beta^{u,s}_z(0,x^u_0), z^c+ \hat\beta^{u,c}_z(0,x^u_0))$, which is the point of intersection of the unstable manifold $\cW^u(z)$ with $x^u=x^u_0$. For $0\leq i \leq n$, write
$w_{-i} = (w^s_{-i}, w^u_{-i}, w^c_{-i})$.  
Since $w$ lies on the unstable manifold of $z$, which is uniformly contracted by $f^{-1}$,
and since $z_{-i}\in \hat M_p(1)$ for $0\leq i \leq n$, we have that $w_{-i} \in I^m_{-i}$ for $0\leq i \leq n$.

We also will use a sequence of ``twin points''
in our calculations.  The twin $w'$ is defined $w'  = (x^u_0, 0, z^c)$;
notice that $w' \in \hW^u_p(z)$.
We then set $w_{-i}' = f^{-i}(w')$, and write $w_{-i}' = ({w^u_{-i}}', 0, {w^c_{-i}}')$, for $0\leq i\leq n-1$.
Since $w\in \cW^u(z)$, and  $w' \in \hW^u_p(z)$, it follows that
$$|w_{-n} - w_{-n}'| \leq |w_{-n} - f^{-n}(z)| + |w_{-n}' - f^{-n}(z)| \leq 2 \hat\nu_{-n}(p)^{-1} |x^u_0|.
$$
The vector $w-w'$ lies in a cone about the center-stable distribution for $f$ at $w'$.
Since this cone is mapped into itself by $Tf^{-1}$, acting as a strict contraction, it follows that
$w_{-i} - w'_{-i}$ lies in this cone as well, for $0\leq i \leq n$.  Recall that vectors
in this cone are contracted/expanded under $f$ by at most $\hat\gamma^{-1}$.
Since $|w_{-n} - w_{-n}'| = O(\hat\nu_{-n}(p)^{-1})$, it follows from a simple inductive argument
that $|w_{-i} - w_{-i}'| = O(\hat\nu_{-n}(p)^{-1} \hat\gamma_{i}(p_n)^{-1} |x^u_0|)$, for $i=0\ldots,n$.
In particular, $|w - w'| = O(\hat\nu_{-n}(p)^{-1} \hat\gamma_{n}(p_n)^{-1} |x^u_0|) = O(\hat\nu_{-n}(p)^{-1} \hat\gamma_{-n}(p) |x^u|_0)$.  Since $\hat\nu\hat\gamma^{-1} < \delta^r$, and $|z^c| = \Theta(\delta_{-n}(p)^{-1})$,
we obtain that $|w - w'| \leq |x^u_0| o(|z^c|^r)$.  But $w-w' = (\hat\beta^{u,s}_z(0,x^u),0 , \hat\beta^{u,c}_z(0,x^u))$,
and so we have shown that $|\hat\beta^u_z(0,x^0)| \leq  |x^u_0| o(|z^c|^r)$, proving the lemma
for the case $\ell=0$.

We next turn to the case $\ell>1$.
Consider the points $(z^c_{-n}, j^{\ell,1}_{(0,w^u_{-n})}\hat\beta^u_{z_{-n}} )$ and
$(z^c_{-n} j^{\ell,1}_{(0,{w^u_{-n}}')} 0)$ in 
$(I^c\times J^{\ell,1}_{\{0\}\times I^u}(I^c\times I^u,\RR^{s+c}))_{-n}$.

To simply notation, we write ``$\cT$'' for $\cT^{\ell,1}_f$
and  $j^{\ell,1}_{-i} \hat\beta^u$ for
 $j^{\ell,1}_{(0,w^u_{-i})}\hat\beta^u_{z_{-i}}$. The notation $|\cdot|_L$ is the fiberwise norm
 on  $I^c\times J^{\ell,1}_{\{0\}\times I^u}(I^c\times I^u,\RR^{s+c})$ defined above (hence
 $|(x,j^{\ell,1}_y\beta)|_L = |j^{\ell,1}_y\beta|_L$).
Having fixed this notation, we next estimate, for $0\leq i\leq n$:
\begin{eqnarray*}
|j^{\ell,1}_{-i+1}(\hat\beta^u)|_L& = &
|\cT(z^c_{-i}, j^{\ell,1}_{-i}\hat\beta^u)|_L\\
& \leq&
|\cT(z^c_{-i}, j^{\ell,1}_{-i}\hat\beta^u) -\cT(z^c_{-i}, j^{\ell,1}_{(0,w^u_{-i})} 0) |_L\\
&&\qquad\qquad\qquad + |\cT(z^c_{-i}, j^{\ell,1}_{(0,w^u_{-i})}0) |_L.
\end{eqnarray*}
We estimate the first term in this latter sum using (\ref{e=scontract}):
\begin{eqnarray*}
|(z^c_{-i}, j^{\ell,1}_{-i}\hat\beta^u) -\cT(z^c_{-i}, j^{\ell,1}_{(0,w^u_{-i})} 0)|_L
&\leq & \kappa(p_{-i}) |j^{\ell,1}_{-i}\hat\beta^u|_L.
\end{eqnarray*}
The second term is estimated using two facts. First, we have that the map $\cT$ is fiberwise
$C^1$ (since $\ell\leq r<k-1$), and so
$$|\cT(z^c_{-i}, j^{\ell,1}_{(0,w^u_{-i})}0)  -  \cT(z^c_{-i}, j^{\ell,1}_{(0,{w^u_{-i}}')}0)|_L
= O(|w_{-i} - w_{-i}'|) = O(\hat\nu_{-n}(p)^{-1} \hat\gamma_{i}(p_{-n})^{-1}).$$
Second, we note that
$\cT\left(z^c_{-i}, j^{\ell,1}_{(0, {w^u_{-i}}'))}0 \right)=
(z^c_{-i+1}, j^{\ell,1}_{(0, {w^u_{-i+1}}')}0 )$.  Hence:
\begin{eqnarray*}
| \cT(z^c_{-i}, j^{\ell,1}_{(0,w^u_{-i})} 0)|_L &\leq& |\cT(z^c_{-i}, j^{\ell,1}_{(0,w^u_{-i})} 0)  - \cT(z^c_{-i}, j^{\ell,1}_{(0, {w^u_{-i}}'))}0)|_L\\
&=& O(\hat\nu_{-n}(p)^{-1} \hat\gamma_{i}(p_{-n})^{-1}),
\end{eqnarray*}
for $i=0,\ldots,n$.
Combining these calculations, we have, for $0\leq i\leq n$:
\begin{eqnarray*}
|\,j^{\ell,1}_{-i+1}(\hat\beta^u)\,|_L =O( \kappa(p_{-i}))\,|\,j^{\ell,1}_{-i}\hat\beta^u\,|_L  + O(\hat\nu_{-n}(p)^{-1} \hat\gamma_{i}(p_{-n})^{-1}).
\end{eqnarray*}
By an inductive argument, we obtain:
\begin{eqnarray*}
|\,j^{\ell,1}_{0}(\hat\beta^u)\,| &=& O( \sum_{i=0}^n  \kappa_{i-n}(p)^{-1} \hat\nu_{-n}(p)^{-1} \hat\gamma_{i}(p_{-n})^{-1} )\\
&=& o( \sum_{i=0}^n  \delta_{i-n}(p)^{\ell-r} \hat\nu_{i-n}(p)^{-1} \hat\nu_{i}(p_{-n})\hat\gamma_{i}(p_{-n})^{-1} )\\
&=& o( \sum_{i=0}^n  \delta_{i-n}(p)^{\ell-r} \hat\nu_{i-n}(p)^{-1} \delta_{i}(p_{-n})^r )\\
&=& o(\delta_{-n}(p)^{\ell-r} ),
\end{eqnarray*}
where we have used the facts that $\kappa  < \delta^{r-\ell}$, and $\hat\nu/\hat\gamma < \delta^r$.
Since  $|z^c| = \Theta(\delta_{-n}(p)^{-1})$,
and recalling our notation for $j^{\ell,1}_{0,x_0^u}\hat\beta^s_z$,
we obtain that
\begin{eqnarray}\label{e=sbound1}
|\,j^{\ell,1}_{0}(\hat\beta^u)\,| = |\, j^{\ell,1}_{0,x_0^u}\hat\beta^u_z \,| = o(|z^c|^{r-\ell}),
\end{eqnarray}
for all $x^u_0\in I^u$.

We are not quite done yet, as (\ref{e=sbound1}) is not exactly what is claimed in the statement of Lemma~\ref{l=spinch}.
To finish the proof, we note that if $\beta$ is $C^{\ell+1}$, then by the equality of mixed partials,
we have that $j^{1}_{x^u=x^u_0}(j^\ell_{x^c=0}\beta) = j^{\ell}_0(j^1_{x^u_0}\beta) = j^{\ell,1}_{0,x_0^u}\beta$.
The quantity we want to estimate is 
$$\left|j^\ell_{0}\left(x^c\mapsto \hat\beta^u_{(0,0,z^c)}(x^c,x^u)\right)\right|$$
Consider the function $\zeta\colon I^u \to J^\ell_0(\RR^c,\RR^{c+s})$ given by
$$\zeta(x^u) = j^\ell_{0}(x^c\mapsto \hat\beta^u_{(0,0,z^c)}(x^c,x^u)).$$ 
The value $\zeta(x^u_0)$ can be obtained by integrating
its derivative along a smooth curve $\gamma(x^u)$, tangent to $\W^u_z(z)$, from
$0$ to $x^u_0$ .
But note that, since $\hat\beta^u_z$ is a $C^{\ell+1}$ function,
we must have $j_{x^u}^1\zeta =  j^{\ell,1}_{0,x^u}\beta$; (\ref{e=sbound1}) implies
that $\zeta(x^u_0) \leq |x^u_0|\cdot o(|z^c|^{r-\ell})$, for all $x^u_0\in I^u$. This completes the 
proof of Lemma~\ref{l=spinch}.
\end{proof}

We remark that the same estimates hold for the lifted fake foliations $\hW^\ast_F$ if $F$ is $C^k$ and
$r$-bunched, for $k\geq 2$ and $r=1$ or $r<k-1$.

\subsection{Fake holonomy}

In the discussion that follows, we define holonomy maps for various fake
foliations between fake center manifolds.  Because we
are interested in local properties, we will be deliberately
careless in referring to the sizes of the domains of definition.  For
example, if $x$ and $x'$ lie within distance $1$ on the same stable manifold, 
and $\tau$ and $\tau'$ are any smooth transversals to $\hW^{s}_x$ inside
$\hW^{cs}(x)$, then there is a well-defined
$\hW^s_x$ holonomy map between a $\rho'$-ball $B_\tau(x,\rho')$ in $\tau$ and 
$\tau'$, if $\rho'$ is sufficiently small.  We will suppress this restriction
of domain and just speak of the $\hW^{cs}_x$-holonomy map between $\tau$ and $\tau'$.
This abuse of notation is justified because all of the holonomy maps we consider
will be taken over paths of bounded length, and all foliations and fake
foliations are continuous.  Hence the restriction of domain can always be performed
uniformly over the manifold.  This will simplify greatly the notation in the sections
that follow.

Let $x\in M$ and $x'\in \cW^s(x,1)$. We define a $C^r$ diffeomorphism
$$\hh_{(x,x')}\colon \hW^c(x)\to \hW^c(x')$$
as the composition of two holonomy maps: first, $\hW^s_x$ holonomy
between the $C^r$ manifolds $\hW^c(x)$ and $\hW^{cs}(x)\cap \hW^{cu}(x')$,
and second, the $\hW^u_{x'}$ holonomy between  $\hW^{cs}(x)\cap \hW^{cu}(x')$
and $\hW^c(x')$.

We also define for $x'\in \cW^s(x,1)$ the {\em lifted} fake holonomy map
$$\hH_{(x,x')}\colon \hW^c_F(x)\to \hW^c_F(x')$$
by composing $\hW^s_{F,x}$ holonomy
between $\hW^c_F(x) = \pi^{-1}(\hW^c(x))$ and 
$\hW^{cs}_F(x)\cap \hW^{cu}_F(x') = \pi^{-1}(\hW^{cs}(x)\cap \hW^{cu}(x'))$,
and $\hW^u_{F,x'}$ holonomy between  $\hW^{cs}_F(x)\cap \hW^{cu}_F(x')$
and $\hW^c_F(x') = \pi^{-1}(\hW^c(x'))$.  Lemma~\ref{l=liftedfake} implies that 
$\pi\circ \hH_{(x,x')} =  \hh_{(x,x')}\circ \pi $.

We similarly define, for  $x\in M$ and $x'\in \cW^u(x,1)$ a map
$$\hh_{(x,x')}\colon \hW^c(x)\to \hW^c(x')$$
as the composition of $\hW^u_x$ holonomy
between $\hW^c(x)$ and $\hW^{cu}(x)\cap \hW^{cs}(x')$,
and $\hW^s_{x'}$ holonomy between  $\hW^{cu}(x)\cap \hW^{cs}(x')$
and $\hW^c(x')$. Finally, we define, for
 $x\in M$ and $x'\in \cW^y(x,1)$,
$$\hH_{(x,x')}\colon \hW^c_F(x)\to \hW^c_F(x')$$
to be the natural lift of $\hh_{x,x')}$, as above.

Proposition~\ref{p=localfol}, parts (6) and (7) and Lemma~\ref{l=liftedfake} immediately imply:

\begin{lemma} Suppose $f$ is $C^{k}$ and $r$-bunched, for some $k\geq 2$ and $r<k-1$ or $r=1$. 
Then for every $x\in M$ and $x'\in \cW^\ast(x,1)$, for $\ast\in\{s,u\}$, the map
$\hh_{(x,x')}$ is a $C^r$ diffeomorphism and depends continuously in the $C^r$
topology on $(x,x')$ .

If $F$ is a $C^{k}$, $r$-bunched extension of $f$, then  $\hH_{(x,x')}$  is a  $C^r$ diffeomorphism
for every $x\in M$, $x'\in \cW^\ast(x,1)$, and $\ast\in\{s,u\}$and depends continuously in the $C^r$
topology on $(x,x')$.  
Moreover, $\hH_{(x,x')}$ projects to $\hh_{(x,x')}$ under $\pi$.
\end{lemma}

The definitions of $\hh$ and $\hH$ readily extend to $(k,1)$-accessible sequences  by composition
(cf. Section~\ref{s=asv} for the definition of accessible sequence).  
Note that any $su$-path corresponds to an $(k,1)$-accessible sequence
if one uses sufficiently many successive points lying in the same stable or unstable leaf.
Lemma~\ref{l=unifaccess} implies that if $f$ is accessible, then there exists a $K_1\in \ZZ_+$ such
that any two points in $M$ can be connected by a $(K_1,1)$-accessible sequence.
For $\cS = (y_0,\ldots,y_k)$ a $(k,1)$-accessible sequence, we define
$\hh_\cS\colon \hW^c(y_0)\to \hW^c(y_k)$ by  
$\hh_\cS = \hh_{(y_{k-1},y_k)}\circ\cdots\circ\hh_{(y_0,y_1)}$ and
$\hH_\cS\colon \hW^c_F(y_0)\to \hW^c_F(y_k)$ by
$\hH_\cS = \hH_{(y_{k-1},y_k)}\circ\cdots\circ\hh_{(y_0,y_1)}$.

\begin{lemma}\label{l=smoothdepend} If $F$ and $f$ are $C^k$ and $r$-bunched for $k\geq 2$ and $r=1$ or $r<k-1$, 
then $\hh_{\cS}$ and $\hH_{\cS}$ are $C^r$ diffeomorphisms that
depend continuously in the $C^r$ topology on $\cS$.
\end{lemma}

We next define the notion of a {\em shadowing accessible sequence}.  This concept
will be crucial for proving that the $C^r$ diffeomorphisms $\hH_\cS$ can
be well-approximated by homeomorphisms that preserve the image of any 
saturated section $\sigma$.

\begin{figure}[h]\label{f=shadowpath}
\psfrag{x}{$x$}
\psfrag{x'}{$x'$}
\psfrag{y}{$y$}
\psfrag{y'}{$y'$}
\psfrag{y''}{$y''$}
\psfrag{w''}{$w''$}
\psfrag{wcx}{$\hW^c(x)$}
\psfrag{wcx'}{$\hW^c(x')$}
\psfrag{wsx}{$\cW^u(x)$}
\psfrag{wsy}{$\cW^u(y)$}
\psfrag{hwsy}{$\hW^u_{x}(y)$}
\psfrag{wuy''}{$\cW^s(y'')$}
\psfrag{hwuy'}{$\hW^s_{x'}(y')$}
\psfrag{(x,x')y}{$(x,x')_y$}
\begin{center}
\includegraphics[scale=1.0]{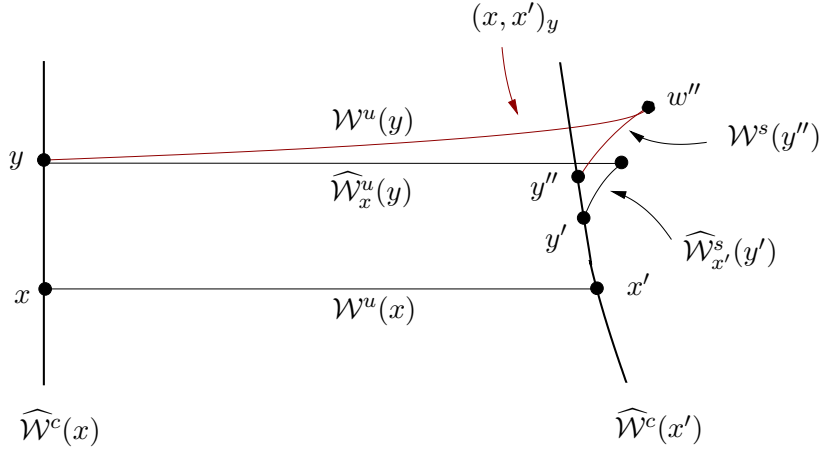}
\caption{The shadowing accessible sequence $(x,x')_y$. The distance between $y'=\hh_{(x,x')}(y)$ and $y''=h_{(x,x')}(y)$ is $O(d(x,y)^r)$; the distance between $x'$ and $y'$ is $O(d(x,y))$ (see Lemma~\ref{l=keyapprox}).}
\end{center}
\end{figure}

Let $x$ be an arbitrary point in $M$, let $x'\in \W^u(x,1)$, and let $y\in \hW^c(x)$.
The {\em shadowing accessible sequence $(x,x')_y$} is defined as follows.
Let  $w''$ be the unique point of intersection of $\W^u(y)$ with
$\bigcup_{z\in \hW^c(x')}\W^s_{loc}(z)$, and let $y''$ be
the unique point of intersection of $\W^u_{loc}(w'')$ and $\hW^c(x')$.
We set $(x,x')_y = (y,w'',y'')$; it is an accessible sequence from $y$ to a point
$y''\in \hW^c(x')$.
See Figure 6.

We have defined $(x,x')_y$ for $x'\in \W^u(x,1)$ and $y\in \hW^c(x)$.  Similarly, 
for $x'\in \W^s(x,1)$, and  $y\in \hW^c(x)$, define
the shadowing accessible sequence $(x,x')_y = (x, w'', y'')$, where
 $w''$ is the unique point of intersection of $\W^s(y)$ with
$\bigcup_{z\in \hW^c(x')}\W^u_{loc}(z)$, and $y''$ is the unique point of intersection of $\W^u_{loc}(w'')$ and $\hW^c(x')$.
It is an accessible sequence from $y$ to a point $y''\in \hW^c(x')$.  
Notice that $(x,x')_y$ is a $(2,1)$ accessible sequence, whereas $(x,x')$
is a $(1,1)$-accessible sequence. We may regard $(x,x')$ as a $(2,1)$
accessible sequence by expressing it as $(x,x',x')$.  Then it is
natural to say that $(x,x')_y\to (x,x')$ as $y\to x$.

We extend the definition of shadowing accessible sequences to all $(k,1)$-accessible
sequences by concatenation.  This defines, for each $(k,1)$-accessible sequence
$\cS$ connecting $x$ and $x'$, and for each $y\in \hW^c(x)$, a $(2k,1)$-accessible
sequence $\cS_y$ connecting $y$ to a point $y'\in \cS^c(x')$. The $(k,1)$ accessible
sequence may be regarded as a $(2k,1)$ accessible sequence by repeating the
appropriate terms in the sequence. With this convention, we have that
$\cS_y\to \cS$ as $y\to x$.
Let $K=2K_1$; henceforth we will restrict our attention to $(K,1)$-accessible sequences.

Now, for $x'\in \W^u(x,1)$ or $x'\in \W^s(x,1)$, we define the map:
$$h_{(x,x')}\colon \hW^c(x) \to \hW^c(x')$$ 
by $h_{(x,x')}(y) = \hh_{(x,x')_y}(y)$; in other words,
$h_{(x,x')}$ sends $y$ to the endpoint of $(x,x')_y$.
Notice that $h_{(x,x)'}$ is a local homeomorphism, but not a diffeomorphism.
However, we will show that $h_{(x,x')}$ has ``an $r$-jet at $x$'' (Lemma~\ref{l=keyapprox});
we will make this notion precise in the following subsections.

Similarly define $\cH_{x,x'}\colon \hW^c_F(x) \to \hW^c_F(x')$ for $x'\in \W^u(x,1)$ or $x'\in \W^s(x,1)$
by $\cH_{(x,x')}(z) = \hcH_{(x,x')_{\pi(z)}}(z)$.
The definitions of $h$ and $\cH$ extend naturally to $(K,1)$-accessible sequences
by composition; for $\cS$ a $(K,1)$-accessible sequence from $x$ to $x'$, we denote by 
$h_{\cS}\colon \hW^c(x)\to \hW^c(x')$ and $\cH_{\cS}\colon \hW^c_F(x) \to \hW^c_F(x')$ 
the corresponding maps.  

Note the simple observation that if $\cS$ is a $(K,1)$-accessible sequence from $x$ to $x'$,
then $\hh_{\cS}(x) = x' = h_\cS(x)$, and for every $z\in \pi^{-1}(x)$,
$\hcH_{\cS}(z) = \cH_\cS(z)$. 

The next lemma is an important consequence of Lemmas~\ref{l=cspinch} and \ref{l=spinch}.  
It tells us that the endpoint of the accessible sequence $(x,x')_y$ is a very
good approximation to $\hh_{(x,x')}(y)$, and this is true even on an infinitesimal level.

\begin{lemma}\label{l=keyapprox} If $f$ is $C^{k}$ and $r$-bunched, for $k\geq 2$ and $r=1$ or $r<k-1$, 
then for every  $(K,1)$ accessible sequence connecting $x$ to $x'$, every  $y\in \hW^c(x)$, 
and every integer $0\leq \ell\leq r$:
$$\| j^\ell_y \hh_{\cS} - j^\ell_y \hh_{\cS_y} \| = o(d(x,y)^{k-\ell}).$$

Moreover, if $F$ is also $C^k$ and $r$-bunched, then for any $z\in \pi^{-1}(x)$ and any $w\in B_\cB(z,1)\cap \pi^{-1}(y)$:
$$\|j^\ell_w \hcH_{\cS} - j^\ell_w \hcH_{\cS_y} \| = o(d(z,w)^{k-\ell}),$$
where the distance is measured in a uniform coordinate system containing the $su$-path
$\gamma_{\cS}$.
\end{lemma}

\begin{proof}This is almost a direct consequence of Lemmas~\ref{l=cspinch} and \ref{l=spinch} in the previous
subsection. We prove it for accessible sequences of the form $\cS = (x,x')$ with $x'\in \cW^u(x,1)$;
the general case follows easily.

Fix $x$, $x'\in \cW^u(x,1)$ and $y\in\hW^c(x)$.  
Write $(x,x')_y = (y,w'',y'')$, as in the definition. Let $v'$ be the unique point
of intersection of $\hW^u_x(y)$ and $\hW^{cs}(x')$, and let $v''$  be the unique point
of intersection of $\W^u(y)$ and $\hW^{cs}(x')$.  See Figure 7.

\begin{figure}[h]
\psfrag{x}{$x$}
\psfrag{x'}{$x'$}
\psfrag{y}{$y$}
\psfrag{y'}{$y'$}
\psfrag{y''}{$y''$}
\psfrag{y'''}{}
\psfrag{z'}{$v'$}
\psfrag{z''}{$v''$}
\psfrag{w''}{$w''$}
\psfrag{wcx}{$\hW^c(x)$}
\psfrag{wcx'}{$\hW^c(x')$}
\psfrag{wsx}{$\cW^u(x)$}
\psfrag{wsy}{$\cW^u(y)$}
\psfrag{hwsy}{$\hW^u_{x}(y)$}
\psfrag{wuy''}{$\cW^s(y'')$}
\psfrag{hwuy'}{$\hW^s_{x'}(y')$}
\psfrag{hwcux'}{$\hW^{cs}(x')$}
\psfrag{hwux'}{$\W^{s}(x')$}
\begin{center}
\includegraphics[scale=1.0]{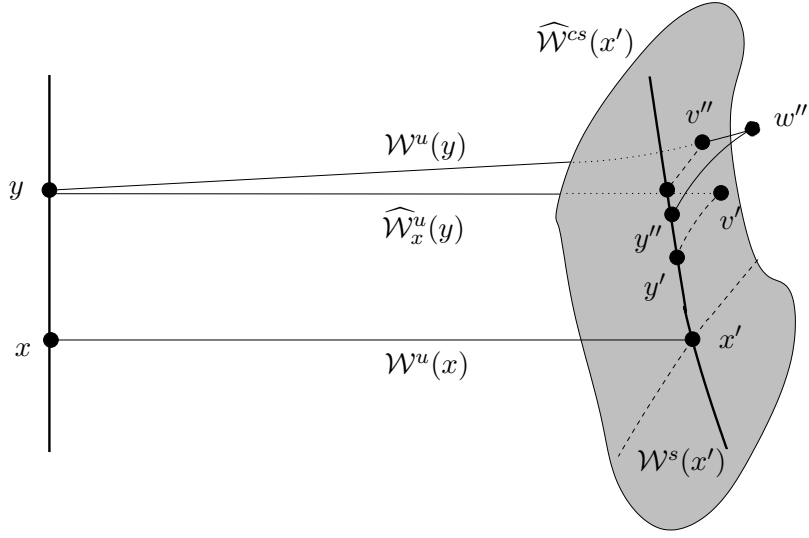}
\caption{Points in the proof of Lemma~\ref{l=keyapprox}}
\end{center}
\end{figure}

Fix a coordinate system adapted at $x$ as in Subsection~\ref{ss=further},
sending $x$ to the origin in $I^m$, $\hW^{cu}(x)$ to $\{x^s=0\}$,  $\hW^{cs}(x)$ to $\{x^u=0\}$,
$\hW^c(x)$ to $\{x^s=0\}$, $\{x^u=0\}$, and sending the fake foliations $\hW^s_x\vert_{\hW^{cs}(x)}$ and $\hW^u_x\vert_{\hW^{cu}(x)}$ to the affine foliations $\{x^u=0, x^s=\hbox{const}\}$ and  $\{x^s=0, x^u=\hbox{const}\}$,
respectively. 
Suppose that $y$ corresponds to the point $z=(0,0,z^c)$ 
and $y''$  corresponds to the point $z''$ in the adapted coordinates at $x$.

In the coordinate system at $x$, we parametrize $\hW^c(x)$ by $x^c\mapsto \hat\omega^c_0(x^c) = (0,0,x^c)$
and $\hW^c(y)$ by $x^c\mapsto \hat\omega_z(x^c)$.
Similarly  we parametrize $\hW^c(x')$ by ${x^c}\mapsto (0,0,{x^c})$
and $\hW^c(y'')$ by ${x^c}\mapsto \hat\omega_{z''}({x^c})$.
We want to compare the $\ell$-jets of $x^c\mapsto \hh_{(x,x')}(0,0,x^c)$
with $x^c \mapsto \hh_{(x,x')_y} \circ \hat\omega_z(x^c)$ at
the point $x^c=z^c$.
We first observe that, by Lemma~\ref{l=cspinch}, we
have that $j^\ell_{z^c}\hat\omega_z(x^c) = o(|z^c|^{r-\ell}) =  o(d(x,y)^{r-\ell})$;
hence we are left to compare the $\ell$-jets of the holonomies 
$\hh_{(x,x')}$ and  $\hh_{(x,x')_y}$ in the coordinates adapted at $x$,
at the point $z$.

We write the maps $\hh_{(x,x')}$ and  $\hh_{(x,x')_y}$ as compositions of several
holonomy maps, and we compare the  distance between the $\ell$-jets of the corresponding terms in 
the compositions. First, we write
$$ \hh_{(x,x')} = h_{x'}^s \circ h_x^u,$$
where $h^u_x\colon \hW^{c}(x)\to \hW^{cu}(x)\cap\hW^{cs}(x')$ is the $\hW^u_x$-holonomy
and  $h^s_{x'}$ is the $\hW^s_{x'}$ holonomy between $\hW^{cu}(x)\cap\hW^{cs}(x')$ 
and $\hW^{c}(x') $.
Next, we write:
$$\hh_{(x,x')_y} =  h^s_{y''} \circ   h^u_{y,\flat}   \circ  h^u_y \circ  h^u_{y,\sharp}
$$ 
where $h^u_{y,\sharp}\colon \hW^{c}(y) \to 
\hW^{cs}(x)\cap\hW^{cu}(y)$, $h^u_y\colon \hW^{cs}(x)\cap\hW^{cu}(y) \to 
\hW^{cu}(y)\cap \hW^{cs}(x')$ and $ h^u_{y,\flat} \colon \hW^{cu}(y)\cap \hW^{cs}(x') \to \hW^{cu}(y)\cap\hW^{cs}(y'') $
are $\hW^u_y$ holonomies, and $h^s_{y''} \colon \hW^{cu}(y)\cap\hW^{cs}(y'') \to \hW^c(y'')$ is $\hW^s_{y''}$-holonomy.

The term $h^u_{y,\sharp}$ in the second composition is expressed in 
the charts at $x$ by the map $(\hat\omega^c_z(x^c),x^c)\mapsto (\overline x^s, 0 ,\overline x^c)$,
where $(\overline x^c, \overline x^s)$ are defined implicitly by the equation 
$\hat\beta^{c,s}_z(\overline x^c, \overline x^s ) = 0 $.
Lemma~\ref{l=cspinch} implies that $|j_{z^c}\hat\omega^{cu}_z - j_{z^c} \hat\omega^{cu}_0|$ 
and $|j_{z^c}\hat\omega^{c}_z - j_{z^c} \hat\omega^{c}_0|$ 
are both $o(|z^c|)^{r-\ell}$, and so in these charts,
$|j^\ell_z h^u_{y,\sharp} -  j^\ell_z id| = o(|z^c|)^{r-\ell}$.

We may choose the coordinate system adapted at $x$ so that  $x'$ is sent to the point $(x^u_0,0,0)$ and
$\hW^{cs}(x')$ is sent to $x^u=x^u_0$, and we may do this in a way that the $C^r$ size of the
chart is bounded independently of $x,x'$; this uses the fact that $p\mapsto \hW^{cs}(p)$
is continuous in the $C^r$ topology.  
Consider the $\hW^u_x$ and $\hW^u_y$ holonomies between $x^u=0$ and $x^c=x^u_0$,
corresponding to the holonomies
$$h^u_x\colon \hW^{cs}(x)\to \hW^{cs}(x'),\quad\hbox{and}\quad h^u_y\colon \hW^{cs}(x)\to \hW^{cs}(x')$$
In the coordinates at $x$, these maps are expressed by the functions
$$(0,x^u,x^c) \mapsto  \omega^{cs}_0(x^c,x^u) ,\quad\hbox{and} \quad (0,x^u,x^c) \mapsto  \omega^{cs}_z(x^c,x^u)$$
Lemma~\ref{l=spinch} implies that  $|j_{z^c}\hat\omega^{cs}_z(\cdot,x^u) - j_{z^c}\hat\omega^{cs}_0(\cdot,x^u) = o(|z^c|)^{r-\ell}$; in the charts at $x$ we
therefore have:
$$|j_z(h^u_x) - j_z(h^u_y)| = o(|z^c|)^{r-\ell} = o(d(x,r)^{r-\ell}).
$$

Consider the image points $v' =  h^u_x(y)$ and $v'' = h^u_y(y)$ of these two holonomy maps in $M$.
Since the distances $d(v',v'')$ and $d(v',y')$ are both $o(|z^c|^r) = o(d(x,y)^{r})$, the transversality
of the bundles in the partially hyperbolic splitting implies that  $d(v'', w'')$ and $d(w'',y'')$
are also $o(d(x,y)^{r})$ (see Figure 7).   Hence the distance from $y''$ to $x$ is $O(d(x',y') + d(y',y'')) = O(d(x,y) + d(x,y)^r) = O(d(x,y))$, and similarly $d(x,v'')$ and $d(x,w'')$ are $O(d(x,y))$. 

We are left to deal with the final terms in the compositions above:
$ h^s_{y''} \circ   h^u_{y,\flat}  $ and $h^s_{x'}$.
All of these are $C^r$ holonomy maps over very short distances, on the order of 
$o(d(x,y)^r)$.  It follows that their $\ell$-jets are close to the identity -- within $o(d(x,y)^{r-\ell})$ --
once we have shown that the transversals on which they are defined have $\ell$-jets
within $o(d(x,y)^{r-\ell})$ of the vertical foliation $\{(x^s,x^u)=\hbox{const}\}$.

Lemma~\ref{l=spinch} implies that the $\ell$-jets of $\hW^{cu}(x')$ and $\hW^{cu}(x)$ coincide along
$\cW^u(x)$.  In particular, in these coordinates, $\hW^c(x')$ and the plane
$\{x^s = 0, x^u = x^u_0\}$ are tangent to order $\ell$ at $x'$. Furthermore,
since $d(x',v'')$, $d(x',w'')$, $d(x',v'')$, $d(x',y')$ and $d(x',y'')$ are all  $O(d(x,y))$,
Lemma~\ref{l=spinch} implies that the manifolds $\hW^{cu}(y)\cap \hW^{cs}(x')$ ,
$\hW^{cu}(x)\cap \hW^{cs}(x')$, $\hW^{cu}(y)\cap\hW^{cs}(y'')$, $\hW^c(y')$ and  $\hW^c(y'')$
can all be expressed in the coordinates adapted at $x$ as graphs of functions from $\{x^u=x^u_0, x^s=0\}$ to $I^{s+u}$  whose $\ell$-jets at $v''$, $v'$, $w''$, $y'$  and $y''$  respectively, are $o(d(x,y)^{r-\ell})$. 
Hence all of the  the transversals for $h^u_{y,\flat}$, $h^s_{x'}$, and  $h^s_{y''}$ have $\ell$-jets
within $o(d(x,y)^\ell)$ of the vertical foliation $\{(x^s,x^u)=\hbox{const}\}$
at their basepoints in the compositions.  It follows that
$|j^\ell_{v''}  (h^s_{y''} \circ   h^u_{y,\flat})  -  j^\ell_{v''} id| = o(d(x,y))^{r-\ell}$
and
$|j^\ell_{v'} h^s_{x'}  -  j^\ell_{v'} id| = o(d(x,y))^{r-\ell}$,
and so
$|j^\ell_{y} \hh_{(x,x')}  -  j^\ell_y\hh_{(x,x')_y}| = o(d(x,y))^{r-\ell}$, as desired.

The proof for the maps $\hcH_{(x,x')}$ and $\hcH_{(x,x')_y}$ are completely analogous.
\end{proof}

\subsection{Central jets}

Let  $(N,\cB,\pi, F)$ be a $C^{k}$, $r$-bunched partially hyperbolic extension of $f$, for some $k\geq 2$,
where $\cB = M\times N$.  We fix Riemannian metrics on $M$ and $N$.
Let $\exp\colon TM\to M$ be
the exponential map for this metric (which we may assume to be $C^\infty$),
and fix $\rho_0>0$ such that $\exp_p$ is a diffeomorphism
from $B_{T_pM}(0,\rho_0)$ to $B_M(p,\rho_0)$, for every $p\in M$.
As in the proof of Lemma~\ref{l=liftedfake}, the bundle $\cB$ pulls back via  $\exp\colon B_{TM}(0,\rho_0) \to M$
to a  $C^r$ bundle $\tilde\pi_0\colon \widetilde{\cB}_0 \to B_{TM}(0,\rho_0)$ with fiber $N$,
where $B_{TM}(0,\rho_0)$ denotes the $\rho_0$-neighborhood of the $0$-section of $TM$.
As in the proof of Lemma~\ref{l=liftedfake}, we fix, for each $p\in M$ a trivialization
of $\widetilde{\cB}_0\vert_{B_{T_pM}(0,\rho_0)}$, depending smoothly on $p\in M$.
Any section $\sigma\colon M\to \cB$  of $\cB$ pulls back
to a section $\tilde\sigma\colon B_{TM}(0,\rho_0)\to \widetilde \cB_0$ via
$\tilde\sigma(v) = (v,\sigma(\exp(v)))$.

Let $TM = \widetilde E^u\oplus \widetilde E^c \oplus \widetilde E^s$ be a $C^\infty$ approximation
to the partially hyperbolic splitting for $f$.  
Observe that $TM$ is a $C^\infty$ bundle over $\widetilde E^c$ under the map $\pi^c\colon
TM \to \widetilde E^c$   that sends 
$v^u+v^c+v^s\in \widetilde E^u(p)\oplus \widetilde E^c(p)\oplus \widetilde E^s(p)$ to 
$v^c\in \widetilde E^c(p)$.   This splitting will give us a global way to parametrize
the fake center manifolds $\hW^c(p)$.


If $f$ is $r$-bunched, for $r=1$ or $r < k-1$,  and the approximation
$TM = \widetilde E^u\oplus \widetilde E^c \oplus \widetilde E^s$ to the hyperbolic splitting is
sufficiently good, then Proposition~\ref{p=localfol} implies 
there exists a map $g^c\colon \widetilde \cB_{\widetilde E^c}(0,\rho) \to B_{TM}(0,\rho_0)$ with the following properties:
\begin{enumerate}
\item $g^c$ is a section of $\pi^c\colon B_{T_pM}(0,\rho)\to \widetilde \cB_{\widetilde E^c}(0,\rho)$,
\item the restriction of $g^c$ to $ B_{\widetilde E^c(p)}(0,\rho)$ is a $C^r$ embedding into $T_pM$,
depending continuously in the $C^r$ topology on $p\in M$;
\item for $p\in M$, the image $g^c( B_{\widetilde E^c(p)}(0,\rho))$ coincides with $\exp_{p}^{-1}(\hW^c(p))$.

\end{enumerate}

Let $\tilde{\pi}^c =\pi^c\circ \tilde\pi\colon \widetilde{\cB}_0\to B_{\widetilde E^c}(0,\rho)$. The bundles
and the relevant maps are summarized in the following commutative diagram.
\[
\begin{xy}
\xymatrixcolsep{4pc}
\xymatrixrowsep{4pc}
\xymatrix{
\widetilde{\cB}_0\ar@/_/[d]_{\tilde{\pi}}\ar@/_3pc/[dd]_{\tilde\pi^c}\ar[r]^{{\tiny\proj}_\cB} & \cB\ar@/_/[d]_{\pi}\\
 B_{T_pM}(0,\rho)\ar@/_/[u]_{\tilde\sigma}\ar@/_/[d]_{\pi^c} \ar[r]^\exp & M\ar@/_/[u]_\sigma\\
 B_{\widetilde E^c}(0,\rho) \ar@/_/[u]_{g^c} & \\
}
\end{xy}
\]
Note that $\tilde{\pi}^c \colon \widetilde{\cB}_0\to B_{\widetilde E^c}(0,\rho)$ 
is a $C^{k}$ bundle.  A different choice of exponential map or approximation
to the partially hyperbolic splitting gives an isomorphic bundle and
a different section ${g^c}'$ related to the first by a uniform graph transform on fibers.

Consider the restriction $\widetilde{\cB}_{0,p}$ of $\widetilde{\cB}_0$ to any
fiber $B_{\widetilde E^c(p)}(0,\rho)$ of $B_{\widetilde E^c}(0,\rho)$ over $p\in M$. 
For every positive integer $\ell \leq r$, we define a 
$C^{k-\ell}$ jet bundle $\cJ^\ell \to M$ whose fiber over $p\in M$ is the space $J^\ell_0(\tilde\pi^c\colon \widetilde{\cB}_{0,p}\to B_{\widetilde E^c(p)}(0,\rho))$.  

Suppose now that $\sigma\colon M\to \cB$ is a section of $\cB$, and that $\ell\leq r$.  We say that
$\sigma$ {\em has a central $\ell$-jet at $p$} if there exists a $C^\ell$ local section
$s = s_{\sigma,p} \in \Gamma^\ell_{0_p}(\tilde\pi^c\colon{\widetilde{\cB}_{0,p}}\to B_{\widetilde E^c(p)}(0,\rho))$
such that, for all $v\in B_{\widetilde E^c(p)}(0,\rho))$:
\begin{eqnarray}\label{e=centerjet}
d_{N}(\hbox{proj}_N\circ \tilde \sigma\, \circ \,g^c(v), \, \hbox{proj}_N\circ s(v)) = o(|v|^\ell).
\end{eqnarray}
It is not hard to see that $\sigma\colon M\to\cB$
has a central $\ell$-jet at $p$ if and only if
the restriction of $\sigma$ to $\hW^c(p)$ is tangent to order $\ell$
at $p$ to a $C^\ell$ local section $\sigma' \colon \hW^c(p)\to \cB$.
If $\sigma$ has a central $\ell$-jet at $p$, for every $p\in M$ then $\sigma$ induces
a well-defined section $j^{\ell}\sigma^c : M \to  \cJ^\ell$ 
that sends $p$ to $j^\ell_0 s_{\sigma,p}$.  We call $j^{\ell}\sigma^c$ the {\em central $\ell$-jet}
of $\sigma$, and we write $j^\ell_p\sigma^c$ for the image of $p$ under $j^\ell\sigma^c$.
It is easy to see that the existence of a central $\ell$-jet for $\sigma$ is independent
of the choice of smooth approximation to the partially hyperbolic splitting and independent
of choice of exponential map.  In general there is no reason to expect the central $\ell$-jet
$j^\ell\sigma^c$ to be a smooth section, even when $\sigma$ itself is smooth, because
$g^c$ is not smooth. 

\begin{remark} If $\sigma$ has a central $\ell$-jet at $p$, then (in a fixed coordinate system
about $p$), $\sigma$ has an $(\ell-l,1,C)$ expansion on $\hW^c(p)$ at $p$.  If $j^\ell\sigma^c$
is continuous, and the error term in (\ref{e=centerjet}) is uniform in $p$,
then $C$ can be chosen uniformly in a neighborhood of $p$.  
\end{remark}

\medskip

In the proof of Theorem~\ref{t=C^rsec}, we will focus attention on the pullbacks
$\cJ^\ell\vert_{\hW^c(x)}$ of $\cJ^\ell$ to various fake center manifolds over $M$.
The central observation we will make use of is that, for each $x\in M$,
there is an
isomorphism $I_x$ between the  bundles $\cJ^\ell\vert_{\hW^c(x)}$ 
and  $J^\ell(\pi\colon \cB_{{\hW^c(x)}} \to \hW^c(x))$.  
To compress notation, we will write $J^\ell({\hW^c(x)}, N)$ for
$J^\ell(\pi\colon \cB_{{\hW^c(x)}} \to \hW^c(x))$. 
For $x\in M$, the isomorphism $I_x\colon \cJ^\ell\vert_{\hW^c(x)}\to J^\ell({\hW^c(x)}, N)$
is defined: 
$$I_x(y,j^\ell_0\psi) = j_y^\ell(id_{\hW^c(x)}, \,\hbox{proj}_{N}\circ\psi\circ \pi^c \circ \exp_y^{-1}).$$

\subsection{Coordinates on the central jet bundle}

Fix $\ell \leq r$.  
We describe here a natural system of $C^{r-\ell}$ coordinate charts on $\cJ^\ell$ based on adapted coordinates
on $M$.  

Let $\widetilde{E}^s\oplus \widetilde{E}^c \oplus \widetilde{E}^u$ be a $C^\infty$  approximation
to the hyperbolic splitting to $M$. Fix a point $p\in M$ and let $(x^u,x^s,x^c)$ be a $C^r$ adapted
coordinate system on $B_M(p,\rho)$ based at $p$.  Next fix $C^r$ local trivializing coordinates $(x^m,v^c)\in \RR^m\times \RR^c$ for $\widetilde{E^c}$ over $B_M(p,\rho)$, covering the adapted charts at $p$ and  sending $B_{\widetilde{E^c}}(0,\rho_1)\vert_{B_M(p,\rho)}$
to $I^m\times I^c$.  Let $(x,v)\in I^m\times I^m$ be the corresponding charts on $B_{TM}(0,\rho_1)\vert_{B_M(p,\rho)}$. In these charts, the projection $\pi^c$ sends
$(x^m,v^u,v^s,v^c)$ to $(x^m, v^c)$.

We choose these charts such that the exponential map on $B_{TM}(0,\rho_1)$ over ${B_M(p,\rho)}$ 
in these coordinates sends $(x^m,v)$ to $x^m+v \in I^m$ (these charts are not
isometric, nor do they preserve the structure of $TM$ as the tangent bundle to $M$,
but they can be chosen to be uniformly $C^r$). 
Also fix $C^r$ coordinates  $(x^m, q)\in \RR^m\times N$ for 
$\cB$ over $B_M(p,\rho)$ sending $\pi^{-1}(B_M(p,\rho))$ to $I^m\times N$,
with $\pi(x^m,q) = x^m$.

The induced coordinates on $\widetilde{\cB}_0$ over $B_{\widetilde{E^c}}(0,\rho_0)\vert_{B_M(p,\rho)}$
take the form $(x^u, x^s,x^c+v^c, v^c, q) \in I^m\times N$. 
We may further choose these coordinates so that, $\tilde\pi$ and $\tilde\pi^c$ 
are the projections onto the $I^m\times I^c$
and $I^m$ coordinates, respectively.
These coordinates give a natural identification of $\cJ^\ell\vert_{B(p,\rho)}$
with  $I^m\times J_0^\ell(I^c,N)$.  

Finally, for each point $q\in N$, we fix $C^r$ coordinates $z^n\in \RR^n$, sending $q$ to 
$0$ and $B_N(q,\rho)$ to  $I^n$. In this way, we define, for each $z\in \widetilde{\cB}_0$, an
adapted system of coordinates  $(x^u, x^s,x^c+v^c, v^c, z^n)\in  \RR^m\times \RR^c\times \RR^n$ sending
$z$ to $0$ and  $B_{\widetilde{\cB}_0}(z,\rho)$ to $I^m\times I^c\times I^n$.

In local coordinates, each element of $\cJ^\ell$ can thus be uniquely represented as a tuple
$(x^m, \wp)$, where $x^m\in I^m$ and $\wp\in P^\ell(c,n)$.
If $\sigma$ has an $\ell$-jet at $p$ for every $p$, we can thus represent
locally the section $j^\ell\sigma^c$ as a function from $I^m$ to $P^\ell(c,n)$, using
the adapted charts in a neighborhood of $\sigma(p)$.

Consider the set $I^c\times J_0^\ell(I^c,N)$.  We may regard this as a natural object associated
to $p\in M$ in either of two ways.  First,  
$I^c\times J_0^\ell(I^c,N)$ embeds as the subset $\{x^u=0, x^s=0\}\times J_0(I^c,N)$
in an adapted coordinate system for $\cJ^\ell\vert_{B(p,\rho)}$,
which gives an identification of  $I^c\times J_0^\ell(I^c,N)$ with $\cJ^\ell\vert_{\hW^c(p)}$.
Second, in the same adapted coordinate system, we have the identification of
$I^c\times J_0^\ell(I^c,N)$ with $J^\ell(\hW^c(p),N)$.
We will use both identifications in what follows. 
We can further put local coordinates on $I^c\times J_0^\ell(I^c,N)$, as follows.
Given a point $z\in \pi^{-1}(x)$, we fix an adapted coordinate system 
$(x^c,z^n)\in I^c\times I^n$ for  
$\hW^c_F(z)$, sending $z$ to $0$.
This gives local coordinates $(x^c,\wp) \in I^c\times P^\ell(c,n)$ on 
$I^c\times J_0^\ell(I^c,N)$ sending $z$ (regarded as an element of $J_0^0(I^c,N)\hookrightarrow J_0^\ell(I^c,N)$)
to $(0,0)$.

Let us give a name to these adapted coordinates and define them more precisely. For $z\in \cB$, 
fix an adapted chart $\hat\varphi_z \colon I^m\times I^c \to B_\cB(z,\rho)$ at $z$,
sending $(0,0)$ to $z$,
sending $\{x^u=0, x^s=0\}$ to $\hW^c_F(z)$, and so on.  We may further assume
that the projection $I^m\times I^c\to  I^m$ is conjugate to $\pi$ under $\hat\varphi$.
The maps $\hat\varphi_z$ induce adapted coordinates $\varphi_z = \pi\circ\hat\varphi_z \circ \iota \colon I^m \to B_M(\pi(z),\rho)$ at $\pi(z)$,
where $\iota$ is the inclusion $x^m \to (x^m,0)$.  We will denote by $\hat\omega^c$  the parametrization
of $\hW^c$ manifolds in the $\varphi_z$ coordinates.
Let $\theta_z\colon I^c \to B_{\widetilde E^c(\pi(z))}(0,\rho)$ be defined by:
$$\theta_z(x^c) = \pi^c \circ \exp_{\pi(z)}^{-1}(\varphi_z(0,0, x^c)).$$

We now define the parametrizations $\eta_z$  and  $\nu_z$ of the bundles  $\cJ^\ell\vert_{\hW^c(\pi(z))}$
and $J^\ell(\hW^c(\pi(z)),N)$ discussed
above.
Let $\eta_z\colon I^c\times P^\ell(c,n) \to \cJ^\ell_{\hW^c(\pi(z))}$
be defined by
\begin{eqnarray*}
\eta_z(x^c, \wp) =
(\varphi_z(0,0,x^c), j^\ell_0\, (id_{\widetilde E^c(\varphi_z(0,0,x^c))}, \hat\varphi_z(0,0,\theta_z^{-1}, \wp(\theta_z^{-1} - x^c)),
\end{eqnarray*}
(recall here that elements of $\widetilde{\cB}_{0,p}$ are of the form $(v,z)\in B_{\widetilde E^c(p)}(0,\rho) \times \cB$ with
$\exp_p(v) = \pi(z)$).
Finally, let $\nu_z\colon I^c\times P^\ell(c,n) \to J^\ell(\hW^c(\pi(z)),N)$
be the map:
$$\nu_z(x^c,\wp) = j^\ell_{\varphi_z(0,0,x^c)} (\hat\varphi_z\circ ( {\varphi_{z}}^{-1} ,\, \wp\left( \hbox{proj}_{I^c} \circ \varphi_{z}^{-1} - x^c )\right) ).  $$

We make all of these choices uniformly in $z$.  Strictly speaking, all of these parametrizations
are defined only on a neighborhood of the zero-section in $P^\ell(c,n)$, but as with the holonomy
maps, we will ignore restriction of domain issues to simplify notation.

Recall the isomorphism $I_x\colon \cJ^\ell\vert_{\hW^c(x)}\to J^\ell(\hW^c(x),N)$
constructed in the previous subsection.
For $w\in \hW^c_F(z,\rho)$, consider the map $I_{z,w}\colon  I^c\times P^\ell(c,n)\to  I^c\times P^\ell(c,n)$
given by $I_{z,w} = \nu_w^{-1} \circ I_{\pi(z)} \circ \eta_z$.  We have constructed these coordinates so that
$I_{z,z} = id_{I^c\times P^\ell(c,n)}$.  The following lemma is a direct consequence
of Lemmas~\ref{l=cspinch} and \ref{l=cspinch}.
\begin{lemma}\label{l=Izw} For every $z\in\cB$ and $w\in \hW^c_F(z,\rho)$, and $\ell\leq r$, we have:
$$|j_0^\ell I_{z,w} - j_0^\ell id_{I^c\times P^\ell(c,n)}| = o(d(z,w)^{r-\ell}).$$
\end{lemma}

\subsection{Holonomy on central jets}

Let $\cS$ be a $(K,1)$-accessible sequence from $x$ to $x'$.
In this subsection, we will define, for each $0 \leq \ell\leq r$,
and each $(K,1)$ accessible sequence from $x$ to $x'$, 
two bundle maps 
$$\hcH^\ell_\cS\colon J^\ell(\hW^c(x),N) \to J^\ell(\hW^c(x'),N) $$
and
$$\cH^\ell_\cS\colon \cJ^\ell\vert_{\hW^c(x)} \to \cJ^\ell\vert_{\hW^c(x')};$$
we will make use of the identification $I_x$ between $J^\ell(\hW^c(x),N)$
and $\cJ^\ell\vert_{\hW^c(x)}$ to compare these maps. (Recall that ``$J^\ell({\hW^c(x)}, N)$''
is shorthand notation for the jet bundle
$J^\ell(\pi\colon \cB_{{\hW^c(x)}} \to \hW^c(x))$).

The map $\hcH^\ell_\cS$ is just the action on $\ell$-jets induced by the diffeomorphism $\hcH_\cS$, defined by:
$$ \hcH^\ell_\cS(j^\ell_y\psi) = j_{\hh_\cS(y)}^\ell \hcH_\cS\circ\psi \circ \hh_\cS^{-1};
$$
Then $\hcH^\ell_\cS$ is a $C^{r-\ell}$ bundle map, covering $\hh_\cS$ (see Section~\ref{s=jets}).
Lemma~\ref{l=smoothdepend} implies:
\begin{lemma}\label{l=smoothdepend2} If $F$ and $f$ are $C^k$ and $r$-bunched for $k\geq 2$ and $r=1$ or $r<k-1$, 
then  $\hH_{\cS}^\ell$ is a  $C^{r-\ell}$ diffeomorphism that
depends continuously in the $C^{r-\ell}$ topology on the $(K,1)$-accessible sequence $\cS$.
\end{lemma}

Fix a point  $z\in \pi^{-1}(x)$ and let $z' = \cH_{\cS}(z)$.  In 
coordinates on  $ J^\ell(\hW^c(x),N)$ and   $ J^\ell(\hW^c(x'),N)$ induced
by the adapted coordinates at $z$ and $z'$, we have a map
$$ \hcH_{\cS,z}^\ell = \nu_{z'}^{-1} \circ \hcH_{\cS}^\ell \circ \nu_z \colon I^c\times P^\ell(c,n) \to I^c\times P^\ell(c,n).$$
Similarly, if $\cS$ connects $x$ and $x'$, we set
$\hh_{\cS,x}(x^c)  =  \varphi_{x'}^{-1} \hh_{\cS}\circ \varphi_x \colon I^c\to I^c$. 

Writing $P^\ell(c,n) = \Pi_{i=0}^\ell L^i_{sym}(\RR^c,\RR^n)$, we have coordinates
$$(x^c,\wp) \mapsto (x^c,\wp_0,\ldots, \wp_\ell)$$ on $I^c\times P^\ell(c,n)$, where
$\wp_i = D^i_{x^c}\wp$.  
Denote by $\hcH_{\cS,z}^\ell(x^c,\wp)_i$
the $L^i_{sym}(\RR^c,\RR^n)$-coordinate of $\hcH_{\cS,z}^\ell(x^c,\wp)$, so that
$$\hcH_{\cS,z}^\ell(x^c,\wp) = (\hh_{\cS,z}(x^c), \hcH_{\cS,z}^\ell(x^c,\wp)_0,\ldots, \hcH_{\cS,z}^\ell(x^c,\wp)_\ell),$$
where $\hcH_{\cS,z}^\ell(x^c,\wp)_0 = \hcH_{\cS,z}(x^c,\wp_0)$.

The following is an immediate consequence of the discussion in Section~\ref{s=jets}.

\begin{lemma}\label{l=uppertriangle} For every $\ell \leq r$, there exists a $C^{r-\ell}$ map $$R^\ell\colon \RR^c \times P^{\ell-1}(c,n) \to L^\ell_{sym}(\RR^c,\RR^n)$$ such that, for every $(x^c,\wp)\in \RR^c\times P^\ell(c,n)$, we have:
$$\hcH_{\cS,z}^\ell(x^c,\wp)_\ell = R^\ell(x^c,\wp_0,\ldots,\wp_{\ell-1}) + \frac{\partial \cH_{\cS,z}}{\partial \wp_0}(x^c,\wp_0)\cdot\wp_\ell\circ (D_{x^c} \hh_{\cS,z})^{-1}.$$
\end{lemma}

We have now defined, for each $(K,1)$-accessible sequence $\cS$ connecting $x$ and $x'$,
a natural lift of the $C^r$ diffeomorphism $\hcH_\cS\colon \hW^c_F(x)\to \hW^c_F(x')$
to a $C^{r-\ell}$ diffeomorphism $\hcH_\cS^\ell \colon J^\ell(\hW^c(x),N) \to   J^\ell(\hW^c(x'),N)$
on the corresponding central $\ell$-jet bundles.  We have also derived in Lemma~\ref{l=uppertriangle} the
important fact that $\hcH_\cS^\ell$ has an upper triangular form with respect to 
the natural local adapted coordinate systems on  $J^\ell(\hW^c(x),N)$ and  $J^\ell(\hW^c(x'),N)$.  

Our next task is to define, for each $(K,1)$-accessible sequence $\cS$ from $x$ to $x'$,
a lift of the homeomorphism $\cH_{\cS}\colon \hW^c_F(x)\to \hW^c_F(x')$
to a map $\cH_{\cS}^\ell \colon \cJ^\ell\vert_{\hW^c(x)}\to  \cJ^\ell\vert_{\hW^c(x')}$
with two essential properties:
\begin{itemize}
\item $\cH_\cS^\ell$ and $\hcH_\cS^\ell$ are tangent to order $r-\ell$ at $x$, under the natural
identification of $J^\ell(\hW^c(x),N)$
and $\cJ^\ell\vert_{\hW^c(x)}$;
\item $\cH_\cS^\ell$ preserves central $\ell$-jets of bisaturated sections of $\cB$.
\end{itemize}

Recall that for $x'\in \cW^s(x,1)$ or $x'\in \cW^u(x,1)$,
we defined $h_{(x,x')}(y) = \hh_{(x,x')_y}(y)$ and
$\cH_{(x,x')}(z) = \hcH_{(x,x')_{\pi(z)}}(z)$; we then extended
this definition to $(K,1)$-accessible sequences via composition.
We further extend this definition to central $\ell$-jets.
If $\cS$ is a $(K,1)$-accessible sequence from $x$ to $x'$, we set:
$$\cH_{\cS}^\ell(y, j_0^\ell\psi)  =  I^{-1}_{h_{\cS}(y)}\circ \hcH_{\cS_{y}}^\ell (I_x \circ (y, j_0^\ell\psi)),
$$
where $I_x\colon \cJ^\ell\vert_{\hW^c(x)}\to J^\ell(\hW^c(x),N)$
is the previously constructed isomorphism. Clearly we have
that $\cH^\ell_{\cS}\colon \cJ^\ell\vert_{\hW^c(x)}\to  \cJ^\ell\vert_{\hW^c(x')}$ is a map 
covering $\cH_{\cS}$, under the projection $\cJ^\ell\vert_{\hW^c(x)} \to \pi^{-1}(\hW^c(x)) = \hW^c_F(x)$.

We now address the first important property of $\cH^\ell_{\cS}$: order $r-\ell$ tangency
to $\hcH^\ell_{\cS}$.  For $\cS$ connecting $x$ and $x'$, we set
$h_{\cS,x}(x^c)  =  \varphi_{x'}^{-1}\circ h_{\cS}\circ \varphi_x\colon I^c\to I^c$,
and for $z\in \pi^{-1}(x)$, we define
$$\cH_{\cS,z}^\ell =  \eta_{z'}^{-1} \circ \cH_{\cS}^\ell \circ \eta_z\colon I^c\times P^\ell(c,n)\to  I^c\times P^\ell(c,n),$$
where $z' = \hcH_\cS(z) = \cH_\cS(z)$. 
Chasing down the definitions, we see that in
$I^c\times P^\ell(c,n)$-coordinates, the map $\cH_{\cS,z}^\ell$ takes the form
$$\cH_{\cS,z}^\ell(x^c,\wp) = I_{\cH_\cS(z(x^c,\wp_0)),z' }^{-1}\circ \hcH^\ell_{\cS_{y(x^c,\wp_0)}} \circ 
I_{z(x^c,\wp_0),z}(x^c,\wp)$$
where $y(x^c) = \varphi_z(0,0,x^c)$, $z(x^c,\wp_0) = \hat\varphi_z(0,0,x^c,\wp_0)$,
and the maps $I_{z,w}$ are defined in the previous subsection.

Hence, by the definition of $\hcH^\ell$,
the difference $|\hcH_{\cS,z}^\ell(x^c,\wp) - \cH_{\cS,z}^\ell(x^c,\wp)|$
can by estimated by bounding:
\begin{itemize}
\item $|j^\ell_{z} \hcH_{\cS} - j^\ell_{y(x^c,\wp_0)} \hcH_{\cS_{z(x^c,\wp_0)}}|$
and  $|j^\ell_{z} \hh_{\cS}^{-1} - j^\ell_{y(x^c,\wp_0)} \hh_{\cS_{y(x^c,\wp_0)}}^{-1}|$
which are both $o(|(x^c,\wp_0)|^{r-\ell})$, by Lemmas~\ref{l=cspinch} and \ref{l=keyapprox}; and
\item $|j_0^\ell I_{\cH_\cS(z(x^c,\wp_0)),z' }^{-1} - j_0^\ell id_{I^c\times P^\ell(c,n)}|$ and  $|j_0^\ell( I_{z(x^c,\wp_0),z}(x^c,\wp)) -  j_0^\ell id_{I^c\times P^\ell(c,n)}|$,
which are both $o|(x^c,\wp_0)|$, by Lemma~\ref{l=Izw}.
\end{itemize}
We thereby obtain:
\begin{lemma}\label{l=tangency} Let $\cS$ be a $(K,1)$-accessible sequence from $x$
to $x'$, and let $z\in \pi^{-1}(x)$.

For each $x^c\in I^c$, $\wp\in P^\ell(c,n)$ with $|\wp|$ bounded,  and for every $0\leq \ell \leq r$ we have:
$$|\hcH_{\cS,z}^\ell(x^c,\wp) - \cH_{\cS,z}^\ell(x^c,\wp)| = o(|(x^c,\wp_0)|^{r-\ell}).$$
\end{lemma}
In this sense, the maps $\cH_\cS^\ell$ and $\hcH_\cS^\ell$ are tangent to order $r-\ell$ at $x$.

As mentioned above, another important property of $\cH^\ell$ is that it
preserves central $\ell$-jets of saturated
sections. 
\begin{lemma}\label{l=preservesjets} Let $\sigma\colon M\to \cB$ be
a bisaturated section.  Then for every $(K,1)$-accessible sequence from
$x$ to $x'$, and any $y\in \hW^c(x)$,
we have $\cH_\cS(\sigma(y)) = \sigma(h_\cS(y))$.

If, in addition $\sigma:M\to \cB$ is Lipschitz and  
has a central $\ell$-jet $j_y^\ell\sigma^c$ at $y$ for some $1\leq \ell < r$, 
then $\sigma$ has a central $\ell$-jet $j_{h_\cS(y)}^\ell\sigma^c$ at $h_\cS(y)$,
and: 
$$j^\ell_{h_\cS(y)}\sigma^c = \cH^\ell_{\cS} (j_y^\ell\sigma^c).$$

\end{lemma}
\begin{proof} Fix $x\in M$ and $\cS$ connecting $x$ to $x'$.  Let $\sigma\colon M\to \cB$
be a bisaturated section.  It suffices to prove the lemma in the case where $x'\in \cW^u(x,1)$
and $\cS=(x,x')$.

Let $y\in \hW^c(x)$.
By definition of $\hcH_\cS$, the value $\hcH_\cS(\sigma(y))$ is the endpoint
of an $su$-lift path for the foliations $\cW^s_F$ and $\cW^u_F$, covering the path $(x,x')_y$.
The endpoint of $(x,x')_y$ is $h_\cS(y)$.  It follows immediately from saturation of $\sigma$
that $\cH_\cS(\sigma(y)) = \sigma(h_\cS(y))$.

Next assume that $\sigma$ is Lipschitz and has a central $\ell$-jet $j^\ell_y\sigma^c$
at $y$, for some $1\leq \ell < r$.  This means that the restriction of $\sigma$ to $\hW^c(y)$ is tangent to order $\ell$
at $y$ to a $C^\ell$ local section $\sigma' \colon \hW^c(y)\to \cB$.  
Let $y' = \hh_{(x,x')_y}(y) = h_{(x,x')}(y)$.
Consider the images
of $\sigma$ and $\sigma'$ under $\hcH_{(x,x')_y}$.  Since $\hcH_{(x,x')_y}$ is a $C^\ell$ diffeomorphism
and covers the $C^\ell$ diffeomorphism $\hh_{(x,x')_y}$, the local sections
 $\hcH_{(x,x')_y}\circ\sigma\circ \hh_{(x,x')_y}^{-1}$
and $\hcH_{(x,x')_y}\circ\sigma'\circ \hh_{(x,x')_y}^{-1}$ 
over $\hW^c(y')$ are tangent to order $\ell$ at $y'$.

Since $\cH^\ell_{(x,y)}$ is defined by the induced action of $\cH^\ell_{(x,y)_y}$ on $\hW^c(y)$,
it suffices to show that  the local sections
 $\hcH_{(x,x')_y}\circ\sigma\circ \hh_{(x,x')_y}^{-1}$ and $\sigma\vert_{\hW^c(y')}$
are tangent to order $\ell$ at $y'$. If this is the case, then $\sigma\vert_{\hW^c(h_{(x,x')}(y))}$
and $\hcH_{(x,x')_y}\circ\sigma'\circ \hh_{(x,x')_y}^{-1}$ are also tangent to order $\ell$ at $y'$;
since the latter section is $C^\ell$, this implies that $\sigma$ has a central $\ell$-jet
at $y'$, and moreover that $j^\ell_{y'}\sigma^c = \cH^\ell_{(x,x')} (j_y^\ell\sigma^c)$.

Lemma~\ref{l=keyapprox} implies that for all $z\in\hW^c(x)$, 
$$d_\cB(\cH_{(x,x')}(\sigma(z)), \hcH_{(x,x')_y}(\sigma(z))) = o(d(\sigma(y),\sigma(z))^r);$$
since $\sigma$ is Lipschitz,  we obtain that
$$d_\cB(\cH_{(x,x')}(\sigma(z)), \hcH_{(x,x')_y}(\sigma(z))) = o(d(y,z)^r).$$
We have already shown that for all $z\in\hW^c(x)$,  $\cH_{(x,x')}(\sigma(z)) = \sigma(h_{(x,x')}(z))$.
Hence $d_\cB(\sigma(h_{(x,x')}(z)), \hcH_{(x,x')_y}(\sigma(z))) = o(d(y,z)^r)$,
and so  $\hcH_{(x,x')_y}\circ\sigma\circ \hh_{(x,x')_y}^{-1}$ and $\sigma\vert_{\hW^c(y')}$
are tangent to order $r$ at $y'$.  Since $\ell < r$, this completes the proof.
\end{proof}

\subsection{$E^c$ curves}

The final tool that we will need in our proof of Theorem~\ref{t=C^rsec} is the concept
of an $E^c$-curve.  As in the proof of Theorem~\ref{t.Crsubmanif}, we will use an inductive
argument to prove that a bisaturated section has central $\ell$-jets.  In the inductive step of the  proof of 
 Theorem~\ref{t.Crsubmanif}, we prove that the $\ell$-jets are Lipschitz continuous, and using
Rademacher's theorem, we obtain $\ell+1$ jets. The analogue of that argument in this context
would be to show that $j^\ell\sigma^c$ is Lipschitz  and then
apply Rademacher's theorem. As mentioned before, this is not possible, since the function
$g^c$ is not Lipschitz, even along $\hW^c$-manifolds.  What we have shown in Lemma~\ref{l=cspinch} is that
$g^c$ and its jets are Lipschitz along $\hW^c(x)$ {\em at $x$}, and what we will show in our
inductive step here is that $j^\ell\sigma^c$ is Lipschitz along $\hW^c(x)$ at $x$, for every $x\in M$.  
This leaves the question of how to apply Rademacher's theorem to obtain anything at all,
let alone $\ell+1$ central jets.  The answer is $E^c$ curves.

An $E^c$ curve is simply a  curve in $M$ that is everywhere tangent to $E^c$.  Such
$C^1$ curves always exist by Peano's existence theorem, but we ask a little more: 
that they be $C^r$.  Rather gratifyingly, there is a simple way to construct such
curves, and when $f$ is $r$-bunched,  Campanato's theorem (Theorem~\ref{t=campanato}) implies that they $C^r$.  If
a function $s$ is Lipschitz along $\hW^c(x)$ at $x$, for every $x\in M$,
then for any $E^c$ curve $\zeta$, it is not hard to see that $s$ must be Lipschitz along $\zeta$, and so
differentiable almost everywhere.  What is more, if a section $\sigma$ has
a central $\ell$-jet $j^\ell\sigma^c$, then restricting  $j^\ell\sigma^c$ to an $E^c$
curve $\zeta$  gives the {\em actual $\ell$-jet for $\sigma$ restricted to $\zeta$}
if $\sigma\vert_\zeta$ is $C^\ell$.
We will use both of these properties of $E^c$ curves in our proof of Theorem~\ref{t=C^rsec}.

\begin{lemma}\label{l=Eccurves} Let $f$ be $C^{k}$ and $r$-bunched, where $k\geq 2$ and $r=1$ or $r < k-1$.
Let $V$ be a coordinate neighborhood of $p$, and let $p^{su}_p:V\to \hW^c(p)$
be a $C^{r}$ submersion. For any $C^{r}$  curve 
$\hat\zeta\colon(-1,1) \to \hW^c(p)$ with $\hat\zeta(0) = p$,
there exists a $C^{r}$ (or $C^{r-1,1}$ if $r>1$ is an integer)
curve $\zeta\colon(-1,1)\to M$ such that, 
for all $t\in (-1,1)$:
\begin{enumerate}
\item $\hat\zeta(t) = p^{su}(\zeta(t))$,
\item $\zeta'(0) = \hat\zeta'(0)$,
\item $\zeta'(t) \in E^c(\zeta(t))$, 
\item $d(\zeta(t), \hat\zeta(t)) \leq O(|t|^r)$, and
\item $|\zeta^{(\ell)}(t) - \hat\zeta^{(\ell)}(t)) | \leq o(|t|^{r-\ell})$, for all $1\leq \ell\leq r$;
what is more, the distance between the $\ell$-jets of $\hW^c(\hat\zeta(t))$ at $\hat\zeta(t)$
and the $\ell$-jets of  $\hW^c(\zeta(t))$ at $\zeta(t)$ is $o(|t|^{r-\ell})$, for all  $1\leq \ell\leq r$.
\end{enumerate}

Moreover, for each $y\in V$
there is a $C^r$ submersion $p^{su}_y:V\to \hW^c(y)$ with the following property.
For each  $s, t\in (-1,1)$,  there exists a point $x_s\in \hW^c(\zeta(t))$ 
such that $x_s$ is connected to  $p^{su}_{\zeta(t)}(\zeta(t+s))$ by an $su$-path whose length
is $o(|s|^r)$,  and such that:
\begin{enumerate}
\item[(6)]  properties (1)-(5) hold for the curves $\zeta_t(s) = \zeta(t+s)$ 
and $\hat\zeta_t(s) = p^{su}_{\zeta(t)}(\zeta(t+s))$, and
\item[(7)] $d(x_s,\zeta_t(s))  = o(|s|^{r})$.
\end{enumerate}
All of these statements hold uniformly in $x\in M$.
\end{lemma}

\begin{proof} 
Let $\hat\zeta$ be given and assume without loss of generality that $\hat\zeta$ is unit speed.
We may also assume that we are working in $C^r$ local coordinates and
that $p^{su}_p$ is projection along an affine plane field $E^{su}$ transverse to $E^c$. 
This planefield then defines for each $y\in M$ a smooth projection $p^{su}_y\colon V\to \hW^c(y)$.

The curve $\hat\zeta$ induces a  vector field on $(p^{su})^{-1}(\hat\zeta)$ by intersecting
$E^c$ with $(Dp^{su})^{-1}(\dot\hat\zeta)$, (note that the two distributions meet transversely
in a linefield).  Integrating this vector field, we get the $E^c$-curve $\zeta$.  
Clearly $\zeta$ satisfies properties (1)-(3).  

To prove (4), we show first that for every $s$ and $t$,
the distance between $\zeta(t+s)$ and the $p^{su}_{\zeta(t)}$-projection of 
$\zeta(t+s)$ onto $\hW^{c}(\zeta(t))$ is $o(|s|^{r})$.  
The proof of this fact is very similar to the proof of 
Lemma~\ref{l=cspinch}.
\begin{figure}[h]
\psfrag{x}{$x$}
\psfrag{x'}{$x'$}
\psfrag{y}{$y$}
\psfrag{y'}{$y'$}
\psfrag{z}{$z$}
\psfrag{x0}{$w$}
\psfrag{z'}{$v'$}
\psfrag{z'}{$z'$}
\psfrag{w''}{$w''$}
\psfrag{zeta}{$\zeta$}
\psfrag{hatzeta}{$\hat\zeta$}
\psfrag{O(delta)}{$\Theta(|s|)$}
\begin{center}
\includegraphics[scale=1.0]{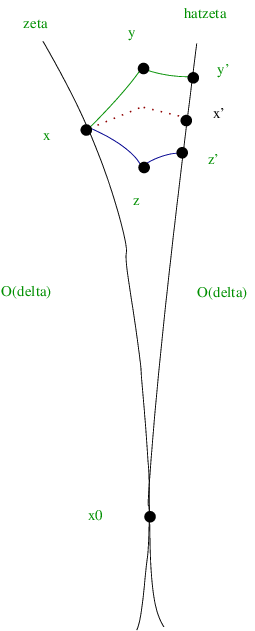}
\caption{An $E^c$-curve $\zeta$ and its shadow $\hat\zeta$}
\end{center}
\end{figure}

Let $w = \zeta(t)$, let $x = \zeta(s+t)$, and let
$x' = p^{su}_{w}(x)$.  Let $y$ be the unique point of intersection of $\W^u(x)$ with
$\bigcup_{z\in \hW^c(x)}\W^s_{loc}(z)$, and let $y'\in \hW^c(x)$ be
the unique point of intersection of $\W^u_{loc}(y)$ and $\hW^c(x)$
Similarly, let $z$ be the unique point of intersection of $\W^s(x)$ with
$\bigcup_{z\in \hW^c(x)}\W^u_{loc}(z)$, and let $z' \in \hW^c(x)$ be
the unique point of intersection of $\W^s_{loc}(z)$ and $\hW^c(x)$ 
(note that $y'$ and $z'$ do not necessarily lie on $\hat\zeta$, but this is not important).  
Note that, because $p^{su}_{w}$ is smooth,
the distance between $x'$ and $x_0$ is $O(|s|)$.
Continuity of the partially hyperbolic splitting and transversality of $E^{su}$ to $E^c$
then imply that $d(y',w)$ and $d(z',w)$ are also $O(|s|)$.
We are going to show that $d(x,y)$ and $d(x,z)$ are both $o(|s|^{r})$;
continuity of the partially hyperbolic splitting and transversality of $E^{su}$ to $E^c$ then
imply that $d(x,x') = o(|s|^{r+\eps})$.

Assume that we have fixed a continuous function $\delta < \{\hat\gamma, 1\}$ satisfying $\delta\hat\nu\hat\gamma^{-1} < \gamma^{r}$; this is possible because $f$ is $r$-bunched.
Choose $n\geq 1$ such that $|s| = \Theta(\delta_n(w))$.  Apply $f^i$ to the picture, for
$i=1,\ldots,n$.  Since $x$ is connected to $x_0$ by a curve everywhere tangent
to $E^c$,  the distance between $x_i$  and $w_i$ is $O(\delta_n(w) \hat\gamma_i(w)^{-1})$.
Since $y'$ lies on $\hW^c(w)$, the distance between  
$x_i$  and $y_i'$ is also $O(\delta_n(w) \hat\gamma_i(w)^{-1})$;
these numbers are less than $1$ for all $i=1,\ldots,n$.
So the distance between $x_n$ and $y_n'$ is less than 
$d(x_n,w) + d(y_n',w) = O(\delta_n(w) \hat\gamma_n(w)^{-1})$.

Since $y\in \cW^s(y')$, the distance between $y_n$ and $y_n'$ is $O(\nu_n(w))$.
But $1$-bunching implies that $\nu_n(w) = o(\delta_n(w) \hat\gamma_n(w)^{-1})$,
and so the distance between $y_n$ and $x_n$ is $O(\delta_n(w) \hat\gamma_n(w)^{-1})$.
Now apply $f^{-n}$ to this picture.  Since $x_n$ and $y_n$ lie on the same unstable
manifold, the distance between their inverse iterates is contracted by $\hat\nu$
at each step. Thus $d(x,y) = O(\hat\nu_n(w) \delta_n(w) \hat\gamma_n(w)^{-1})$.
But we chose $\delta$ so that $\delta\hat\nu\hat\gamma^{-1} < \gamma^{r}$.
Hence $d(x,y) = o(\hat\gamma_n(w)^{-r}) = o(|s|^{r})$.  A similar argument
replacing $f$ by $f^{-1}$ shows that $d(x,z) = o(|s|^{r})$.
Setting $t=0$ we obtain conclusion (4).

To show that $\zeta$ is $C^r$ we use Theorem~\ref{t=campanato}.  Note that
for each $t\in (-1,1)$, the projection  $p^{su}_{\zeta(t)}\zeta$ onto
$\hW^c(\zeta(t))$ is the same as  $p^{su}_{\zeta(t)}\hat\zeta$;
in particular,  $p^{su}_{\zeta(t)}\zeta$ is uniformly $C^r$,
since $\hat\zeta$ and $p^{su}$ are $C^r$, and $\hW^c(\zeta(t))$
is uniformly $C^r$, by $r$-bunching of $f$.
But the previous calculation now implies that there exists a constant
$C>0$, and for every $t\in (-1,1)$, 
a $C^r$ function $p^{su}_{\zeta(t)}\zeta\colon (-1,1) \to M$
such that: 
$$d(p^{su}_{\zeta(t)}\zeta(t+s), \zeta(t+s)) \leq C|s|^r,$$
for every $s\in (-1,1)$.  Theorem ~\ref{t=campanato} implies that
$\zeta$ is $C^r$ (or $C^{r-1,1}$, if $r>1$ and $r$ is an integer).

The proof of item (5) is very similar to the proof of Lemma~\ref{l=cspinch}
and is left as an exercise.

Conclusion (6)  of the lemma is immediate from the previous calculations.
The proof of conclusion (7) is very similar to the calculation above, and is
also left to the reader. \end{proof}

\begin{remark}
In fact $E^{cs}$, $E^{cu}$ and $E^{c}$ are all  $C^r$ along $E^c$-curves.
The proof uses Campanato's theorem again.  This time the
smooth approximating functions are parametrizations of the manifolds $\hW^{cs}$
and  $\hW^{cu}$.  
\end{remark}

\section{Proof of Theorem~\ref{t=C^rsec}}

Suppose $F$ is a $C^{k}$ and $r$-bunched extension of $f$ where  $k\geq 2$ and $r=1$ or $r < k-1$,
and let $\sigma\colon M\to \cB$ be a bisaturated section.  
The first step of the proof is to show:

\begin{lemma}\label{l=hasajet} 
$\sigma$ has a central $\lfloor r \rfloor$-jet at every point in $M$, and $j^{\lfloor r \rfloor}\sigma^c$
is continuous.
\end{lemma}

\begin{proof}

We prove the following inductive statements, for $\ell\in [0,\lfloor r \rfloor]$:
\begin{itemize} 
\item[I$_\ell$.] $\sigma$ has a central $\ell$-jet at every point.
\item[II$_\ell$.] The central $\ell-1$-jets of $\sigma$ along $\hW^c(x)$ are Lipschitz at $x$, uniformly in $x\in M$,
for $\ell \geq 1$.
\item[III$_\ell$.] The restriction of $\sigma$ to $E^c$ curves is uniformly $C^\ell$.
\end{itemize}

We first verify I$_0$--III$_0$.   Statement II$_0$ is empty. 
Since $\sigma$ is bisaturated, Theorem~\ref{t=asv} implies that $\sigma$ is continuous.  This implies
I$_0$--III$_0$. Now assume that statements I$_\ell$--III$_\ell$ hold, for some $\ell\in \{0,\ldots, \lfloor r \rfloor-1\}$.

\bigskip

\noindent{\bf The central $\ell$-jets are continuous.}
We note that $\cJ^\ell$ is an admissible bundle; the holonomy
map for the accessible sequence $\cS$ for $x$ to $x'$
is just the restriction of the map $\cH^\ell_\cS$ to the fibers $\cJ^\ell\vert_{\{x\}}$
and  $\cJ^\ell\vert_{\{x'\}}$. Lemma~\ref{l=preservesjets} implies that
if $\sigma$ has a  central $\ell$-jet $j^\ell\sigma^c$, then $j^\ell\sigma^c$ is
a bisaturated section of $\cJ^\ell$. Continuity follows from Theorem~\ref{t=asv}.

\bigskip 

\noindent{\bf The central $\ell$-jets of $\sigma$ along $\hW^c(x)$ are Lipschitz at $x$.}
We first show that for every $x$, the restriction of $j^\ell\sigma^c$ to $\hW^c(x)$ is Lipschitz at $x$
(where the Lipschitz constant is uniform in $x$).

By Lemma~\ref{l=pathfamily} each point $x\in M$ has a uniformly large
neighborhood $U_x$ and a family of $(K,1)$-accessible sequences
$\{\cS_{x,y}\}_{y\in U_x}$ such that
$\cS_{x,y}$ connects $x$ to $y$, $\cS_{x,x}$ is a palindromic accessible cycle and
$\lim_{y\to x}\cS_{x,y} = \cS_{x,x}$, uniformly in $x$.
We may assume that $\hW^c(x)$ is contained in the neighborhood $U_x$.

We fix $x=x_0$ and
$x_1\in \hW^c(x_0)$ and choose a sequence of points $x_i\in U_{x_0}$ as follows.
Let $U_{x_0}$ and $\{\cS_{x,y}\}_{y\in U_{x_0}}$ be given by Lemma~\ref{l=pathfamily}.
For each $i\geq 1$, given $x_i\in U_{x_0}$, the accessible sequence 
$\cS_i = \cS_{x_0,x_i}$ determines a map  $h_i:=h_{\cS_i}\colon\hW^c(x_0)\to \hW^c(x_i)$, satisfying
$x_{i} = h_i(x_0)$.  We set $x_{i+1} = h_i(x_1) \in \hW^c(x_i)$.
\medskip

We now write things in adapted coordinates.  Let $\wp_{\sigma}^\ell\colon U_{x_0} \to P^\ell(c,n)$
be the function satisfying $j_y^\ell \sigma^c = \nu_{\sigma(y)}(\wp^\ell_\sigma(y))$.
Then $\wp_\sigma^\ell$ assigns in adapted  coordinates the appropriate central $\ell$-jet of $\sigma$ to each
point in $U_{x_0}$.
We are going to show that the restriction
$\wp_\sigma^\ell\colon \hW^c(x) \to P^\ell(c,n)$ is Lipschitz at $x$.

Let $\cH^\ell_{\cS_i}\colon  \cJ^\ell_{\hW^c_{x_0}} \to \cJ^\ell_{\hW^c_{x_i}}$
be the lifted ``true holonomy on jets,'' which covers $h_{\cS_i}$ and let 
$\hcH^\ell_{\cS_i} \colon J^\ell(\hW^c(x_0),N) \to J^\ell(\hW^c(x_i),N)$ be the lifted ``fake holonomy on jets,''
which covers $\hh_{\cS_i}$.
This defines maps $\cH^\ell_i = \cH^\ell_{\cS_i,\sigma(x_0)}$
and  $\hcH^\ell_i =\hcH^\ell_{\cS_i,\sigma(x_0)}$ on
$I^c\times P^\ell(c,n)$.
Write $\cH^\ell_{i}(v,\wp) = (h_{i}(v), H^\ell_{i}(v,\wp))$ and 
$\hcH^\ell_{i}(v,\wp) = (\hh_{i}(v), \hH^\ell_{i}(v,\wp))$.
Observe that $\varphi_{\sigma(x_i)}(0,0,0) = 0$ for all $i\geq 0$; let
$v_{i+1}\in I^c$ be the point satisfying $ \varphi_{\sigma(x_i)}(0,0,v_{i+1}) = x_{i+1}$.
Note that $|v_1| = O(|x_1-x_0|)$, $|v_{i+1}| = O(|x_{i+1} - x_i|)$, and
$v_{i+1} = \hh_{i}(v_1)$, for all $i\geq 0$.

Then, since $j^\ell\sigma^c$ is bisaturated and continuous (and hence bounded) Lemma~\ref{l=preservesjets}
implies:
$$\cH_i^\ell (0,\wp_\sigma^\ell(x_0)) =(0, \wp_\sigma^\ell(x_i)),\quad\hbox{and}\quad \cH_{i}^\ell (v_1,\wp_\sigma^\ell(x_1)) = (v_{i+1},\wp_\sigma^\ell(x_{i+1})).$$
By definition of $\cH_i^\ell$ and $\cH_i^\ell $, we have
$\hcH_i^\ell (0,\wp_\sigma^\ell(x_0))=  \cH_i^\ell (0,\wp_\sigma^\ell(x_0))$; furthermore, Lemma~\ref{l=tangency} implies
\begin{eqnarray}\label{e=hatvsnohat}
|\cH_i^\ell(v_1,\wp_\sigma^\ell(x_1)) - \hcH_i^\ell(v_1,\wp_\sigma^\ell(x_1))|\qquad\qquad\qquad\qquad\qquad \\
\qquad\qquad \leq o(|x_1-x_0|^{r-\ell} + |\wp_\sigma^{\ell-1}(x_1)-\wp_\sigma^{\ell-1}(x_0)|^{r-\ell} ).
\end{eqnarray}
Now Lemma~\ref{l=smoothdepend2} implies that $\hcH^\ell_i$ is $C^{r-\ell}$, and uniformly close to the identity map,
since $\cS_{x_0,x_0}$ is palindromic and $\cS_{x_0,y}\to \cS_{x_0,x_0}$ as $y\to x_0$, uniformly
in $x_0$.

Lemma~\ref{l=uppertriangle} then implies that for every $i$ with $|x_i - x_0| = O(1)$,
there exist linear maps ,
$A_i = D\hh_i (0)\colon \RR^c \to \RR^c$, $B_i = D_{v}\hH_i^\ell(0,\wp_\sigma^\ell(x_0))\colon \RR^c \to P^\ell(c,n)$
and $C_i = D_{\wp_\ell}\hH_i^\ell(0,\wp_\sigma^\ell(x_0))\colon P^\ell(c,n) \to P^\ell(c,n),$ such that
\begin{eqnarray}\label{e=linearize1}
v_{i+1} = \hh_i(v_1) = A_i(v_1) + o(|v_1|),  
\end{eqnarray}
and
\begin{eqnarray*}
\hH_i(v_1,\wp_\sigma^\ell(x_1)) - \hH_i(0,\wp_\sigma^\ell(x_0)) 
&= &B_i(v_1) +  C_i(\wp_\sigma^{\ell}(x_1)-\wp_\sigma^{\ell}(x_0))\\
&& + o(|v_1| + |\wp_\sigma^{\ell-1}(x_1)-\wp_\sigma^{\ell-1}(x_0)|)
\end{eqnarray*}
Moreover, we may assume that, for all $i$ with $|x_i-x_0| = O(1)$:
\begin{eqnarray}\label{e=linearbounds}
\|A_i - Id_{\RR^c}\| < \frac{1}{4}, \quad\|C_i-Id_{P^\ell(c,n)}\| < \frac{1}{4},\quad\hbox{and}\quad
 \|B_i\| < \frac{1}{4}.
\end{eqnarray}
By the inductive hypothesis $II_{\ell}$,  the central $(\ell-1)$-jets of $\sigma$ along $\hW^c(x)$ are Lipschitz at $x$.
Hence $ |\wp_\sigma^{\ell-1}(x_1)-\wp_\sigma^{\ell-1}(x_0)|=O (|x_1-x_0|)$, and so combining
(\ref{e=hatvsnohat}) and (\ref{e=linearize1}) we obtain
\begin{eqnarray}\label{e=linearize2}
\hH_i(v_1,\wp_\sigma^\ell(x_1)) - \hH_i(0,\wp_\sigma^\ell(x_0))\qquad\qquad\qquad\qquad\qquad\qquad\qquad \\
\qquad\qquad\qquad = B_i(v_1) +  C_i(\wp_\sigma^{\ell}(x_1)-\wp_\sigma^{\ell}(x_0))+ o(|x_1-x_0|).
\end{eqnarray}
(Notice that when $\ell=0$ the $ |\wp_\sigma^{\ell-1}(x_1)-\wp_\sigma^{\ell-1}(x_0)|$ terms do not appear
in these expressions, and so Lipschitz
regularity of $\sigma$ is not an issue.  This is due to upper triangularity of $\hcH$.)

The proof now proceeds as the proof of Theorem~\ref{t.Crsubmanif}. Notice
here that we do not need to assume a priori that $\sigma$ is $C^1$;
the reason is that the derivatives of $\hcH^\ell_i$ are upper triangular,
(unlike the maps $H_u^\ell$ in the Proof of Theorem~\ref{t.Crsubmanif}) which
allows for more precise estimates.
We choose $N = \Theta(|x_1 - x_0|^{-1})$.  By (\ref{e=linearize1}) and (\ref{e=linearbounds}),
this choice of $N$ ensures that
$|x_N - x_0| = O(1)$.
Summing (\ref{e=linearize2}) from $i=0$ to $N-1$, we obtain:
\begin{eqnarray*}
\sum_{i=0}^{N-1} \hH_i(v_1,\wp_\sigma^\ell(x_1)) - \hH_i(0,\wp_\sigma^\ell(x_0)) &= & (\sum_{i=0}^{N-1} B_i)(v_1)\\
&& +(\sum_{i=1}^N C_i)(\wp_\sigma^\ell(x_1)-\wp_\sigma^\ell(x_0)) \\
&& + N o(|x_1-x_0|).\\
\end{eqnarray*}
Equation (\ref{e=hatvsnohat}) implies that $\sum_{i=0}^{N-1} \hH_i(v_1,\wp_\sigma^\ell(x_1)) - \hH_i(0,\wp_\sigma^\ell(x_0))=$
\begin{eqnarray*}
\qquad\qquad\qquad\qquad\qquad\qquad&=& \sum_{i=0}^{N-1} \left(H_i(v_1,\wp_\sigma^\ell(x_1)) - H_i(0,\wp_\sigma^\ell(x_0))\right) \\
&& \qquad\qquad\qquad\qquad + N o(|x_1-x_0|^{r-\ell})\\
&=& \sum_{i=0}^{N-1} \wp_\sigma^\ell(x_{i+1}) - \wp_\sigma^\ell(x_i)+ N o(|x_1-x_0|^{r-\ell})\\
&=&   \wp_\sigma^\ell(x_{N}) - \wp_\sigma^\ell(x_1) + N o(|x_1-x_0|^{r-\ell}).
\end{eqnarray*}
Hence, since $r-\ell \geq 1$:
\begin{eqnarray*}
\frac{1}{N}(\wp_\sigma^\ell(x_{N}) - \wp_\sigma^\ell(x_1)) & = &    \left(\frac{1}{N}\sum_{i=0}^{N-1} B_i\right)(v_1) \\
&&+  \left(\frac{1}{N}\sum_{i=1}^N C_i\right)(\wp_\sigma^\ell(x_1)-\wp_\sigma^\ell(x_0)) + o(|x_1-x_0|).\\
\end{eqnarray*}
Rearranging terms and taking norms, we get
\begin{eqnarray*}
 |\frac{1}{N}(\sum_{i=1}^N C_i)(\wp_\sigma^\ell(x_1)-\wp_\sigma^\ell(x_0))| &\leq &
|\frac{1}{N}(\wp_\sigma^\ell(x_{N}) - \wp_\sigma^\ell(x_1))|\\
&& + |\frac{1}{N}(\sum_{i=0}^{N-1} B_i)(v_1)|+  o(|x_1-x_0|)\\
&\leq& O(\frac{1}{N}) + \frac{1}{4}|(x_1-x_0)| + o(|x_1-x_0|),
\end{eqnarray*}
using (\ref{e=linearbounds}) and the fact that $\wp^\ell_\sigma$ is continuous, and
hence bounded.  Again using (\ref{e=linearbounds}) we have that
\begin{eqnarray*}
 \left|\left(\frac{1}{N}\sum_{i=1}^N C_i\right)(\wp_\sigma^\ell(x_1)-\wp_\sigma^\ell(x_0))\right| &\geq& \frac{3}{4} |\wp_\sigma^\ell(x_1)-\wp_\sigma^\ell(x_0)|.
\end{eqnarray*}
Combining the previous two estimates, we get:
\begin{eqnarray*}
|\wp_\sigma^\ell(x_1)-\wp_\sigma^\ell(x_0)| &\leq& \frac{4}{3} \left(O(\frac{1}{N}) + \frac{1}{4}|(x_1-x_0)| + o(|x_1-x_0|) \right).
\end{eqnarray*}
\bigskip
Finally, since $\frac{1}{N} = \Theta(|x_1 - x_0|)$,
we obtain that
\begin{eqnarray*}
|\wp_\sigma^\ell(x_1)-\wp_\sigma^\ell(x_0)| &=& O(|x_1-x_0|),
\end{eqnarray*}
which is the desired estimate.  This verifies II$_{\ell+1}$.

\bigskip

\noindent{\bf $\sigma$ is Lipschitz.}
If $\ell=0$, we know that $\sigma$ is Lipschitz at $x$ along $\hW^c(x)$ leaves, for every $x$,
and differentiable along $\W^u$ leaves, and $\W^s$ leaves, with the partial derivatives continuous.  This
readily implies that $\sigma$ is Lipschitz.  

\bigskip

\noindent{\bf $\sigma$ has a central $(\ell+1)$-jet at every point.}
We fix a uniform system of $C^r$ submersions $p^{su}_x\colon V_x\to \hW^c(x)$
defined in coordinate neighborhoods in $M$.  We define $E^c$ curves using these
submersions.

\begin{lemma} $j^\ell\sigma^c$ is uniformly Lipschitz along $E^c$ curves.
\end{lemma}
\begin{proof} This is a straighforward consequence of Lemma~\ref{l=Eccurves} and the fact that
$j^\ell\sigma^c$ is Lipschitz along $\hW^c(x)$ at $x$,  for every $x\in M$.

\end{proof}

Fix an $E^c$ curve $\zeta^1$ inside of a coordinate neighborhood $V$.
Since $j^\ell\sigma^c$ is Lipschitz along
$\zeta^1$, it is differentiable almost everywhere.  Fix a point $x_1 = \zeta^1(t)$ of differentiability.
Then $j^\ell\sigma^c$ has a partial derivative along $\zeta^1$ {\em at $x_1$}.
Let $\{p^{su}_y\colon V\to \hW^c(y)\}_{y\in V}$ be the system
of submersions in the neighborhood $V$ given by Lemma~\ref{l=Eccurves}.
Consider the $C^r$ curve $\hat\zeta_{x_1}^1(s) := p_{x_1}^{su}\circ \zeta^1(t+s)$ in $\hW^c(x_1)$.
Lemma~\ref{l=Eccurves} implies that for each $s$,
there is a point $x_s\in \hW^c(\zeta(t+s))$ that is
connected to $\hat\zeta^1_{x_1}(s)$ by a $su$-path $\cS$ whose length is $o(|s|^r)$.
Since $j^\ell\sigma^c$ is bisaturated, we have that $j^\ell_{x_s}\sigma^c = \cH^\ell_\cS(j^\ell_{\hat\zeta^1_{x_1}(s)}\sigma^c)$.  Lemma~\ref{l=spinch}
implies that
$$d(j^\ell_{x_s}\sigma^c, j^\ell_{\hat\zeta^1_{x_1}(s)}\sigma^c)
= O(\hbox{length}(\cS)) + O(d(j_{x_s}^\ell\hW^c(x_s), j_{\hat\zeta^1_{x_1}(s) }^\ell\hW^c(\hat\zeta^1_{x_1}(s)))).
$$
Lemmas~\ref{l=Eccurves} (5), implies that $d(j_{x_s}^\ell\hW^c(x_s), j_{\hat\zeta^1_{x_1}(s) }^\ell\hW^c(\hat\zeta^1_{x_1}(s))) = o(|s|^{r-\ell})$.
Hence:
$$d(j^\ell_{\hat\zeta_{x_1}^1(s)}\sigma^c, j^\ell_{x_s}\sigma^c) = o(|s|^r) + o(|s|^{r-\ell}) = o(|s|^{r-\ell}).$$
Since $j^\ell\sigma^c$ is Lipschitz along $\hW^c(\zeta(t+s))$ at $\zeta(t+s)$,
we also obtain that $d(j^\ell_{x_s}\sigma^c, j^\ell_{\zeta(t+s)}\sigma^c) = O(d(x_s, \zeta(t+s))) = o(|s|^r)$.
Thus, in local coordinates, we have:
$$j^\ell_{\hat\zeta_{x_1}^1(s)}\sigma^c - j^\ell_{x_1}\sigma^c = j^\ell_{\zeta(t+s)}\sigma^c - j^\ell_{x_1}\sigma^c + o(|s|^{r-\ell});$$
since $\ell\leq r-1$ and $j^\ell\sigma^c\circ\zeta$ is differentiable at $x_1=\zeta(t)$,
this implies that $j^\ell\sigma^c$ is differentiable at $x_1$ along
the $C^r$ curve $\hat\zeta_{x_1}^1$ in $\hW^c(x_1)$. 

Let $U_{x_1}$ and $\{\cS_y^1\}_{y\in U_{x_1}}$ be the family of
accessible sequences given by Lemma~\ref{l=pathfamily}.
Since $j^\ell\sigma^c$ is bisaturated, Lemmas~\ref{l=tangency} and \ref{l=preservesjets}
imply that the image of $\hat\zeta_{x_1}^1$ under $\hcH_{\cS_y^1}$ is a $C^r$ path $\hat\zeta_y^1$ 
in $\hW^c(y)$ along which $j^\ell\sigma^c$
is differentiable at $y$.  Furthermore, $y\mapsto \hat\zeta_y^1$ is
 continuous at $x_1$ in the $C^r$ topology, and and the derivative
of $j^\ell\sigma^c$ along $\zeta_y^1$ at $y$ is continuous at $x_1$.

Now choose another $E^c$ curve $\zeta^2$ through $x_1$, quasi-transverse to $\zeta^1$ (that is,
such that the tangent spaces to $\zeta^1$ and $\zeta^2$ at $x_1$ are linearly independent). 
Again $j^\ell\sigma^c$ is Lipschitz along $\zeta^2$, and we choose a point of differentiability $x_2$. Since $x_1$ is 
a point of continuity of the curves $\{\hat\zeta^1_y \}_{y\in U_{x_1}}$, we may assume
(by choosing $x_2$ close to $x_1$)
that $\zeta^2$ and $\hat\zeta^1_{x_2}$ are quasi-transverse at $x_2$; hence
$\hat\zeta^1_{x_2}$ and $\hat\zeta^2_{x_2} = p^{su}_{x_2}\zeta^2$ are quasi-transverse
curves in $\hW^c(x_2)$ along which $j^\ell\sigma^c$ has partial derivatives at $x_2$.  

Let $U_{x_2}$ and $\{\cS_{y}^2\}_{y\in U_{x_2}}$ be given by Lemma~\ref{l=pathfamily}
for the point $x_2$.
Applying the fake holonomy $\hcH_{\gamma_y^2}$ to the transverse pair of curves $\hat\zeta^1_{x_2}$ and
$\hat\zeta^2_{x_2}$, and reusing the label $\hat\zeta^1_y$ now to denote the  curve $\hcH_{\gamma_y^2}\circ \hat\zeta^1_{x_2}$, 
we obtain a family of pairs $\{(\hat\zeta^1_y, \hat\zeta^2_y)\}_{y\in U_{x_2}}$ of  quasi-transverse curves
along which $j^\ell\sigma^c$ is differentiable at their intersection and such that 
$y\mapsto (\hat\zeta_y^1,\hat\zeta_y^2)$ is continuous at $x_2$ in the $C^r$ topology.

Repeating this procedure $c=\hbox{dim}(E^c)$ times, we obtain a point $x_c$, a neighborhood $U_{x_c}$ of $x_c$,
and a family of $c$-tuples of curves $\{(\hat\zeta^1_y,\ldots, \hat\zeta^c_y)\}_{y\in U_{x_c}}$
such that, for each $y\in U_{x_c}$:
\begin{enumerate}
\item the curves $(\hat\zeta^1_y,\ldots, \hat\zeta^c_y)$ contain
$y$ and lie in $\hW^c(y)$;
\item the tangent lines to  $(\hat\zeta^1_y,\ldots, \hat\zeta^c_y)$
at $y$ span $E^c_y$;
\item $j^\ell\sigma^c$ is differentiable at $y$ along $\hat\zeta^c_y$, 
\item the map $z\mapsto (\hat\zeta^1_z,\ldots, \hat\zeta^c_z)$ is continuous at $x_c$ in the
$C^r$ topology; and
\item for each $i$, the partial derivative of $j^\ell\sigma^c$ along $\zeta^i_z$ at $z$ is continuous at $z=x_c$.
\end{enumerate}

We claim that this implies that $j^\ell\sigma^c$ is differentiable along $\hW^c(x_c)$ at $x_c$.

\begin{lemma}\label{l=goodpath}
Let $x_c$ be given as above. Then for every $z\in \hW^c(x_c)$, there exists a path $\eta$
from $x^c$ to a point $w$ in $M$ with the following properties.  The path $\eta$
is a concatenation of $\hat\zeta^i$ paths $\eta = \hat\zeta^1_{1}\hat\zeta^2_{2}\cdots\hat\zeta^c_{c}$,
with $d(w, p^{su}_{x_c}(w)) = o(d(z,x_c)^r)$ and $d(p^{su}_{x_c}(w),z) = o(d(z,x_c))$. \end{lemma}

\begin{proof} Denote by $\zeta^i_y$ the $\zeta^i$ curve anchored at $y$ (so that $\zeta^i_y(0)=y$).
Starting with $x_c$, we take the union $\cP_1: = \bigcup_{q\in  \hat\zeta^1_{x_c}} \hat\zeta^2_q$.
Similarly, for $i\geq 1$, we define $\cP_{i+1}: = \bigcup_{q\in  \cP_i} \hat\zeta^{i+1}_q$.
The quasi transversality of the curves $\zeta^1,\ldots,\zeta^c$ at every point 
and continuity of $\zeta^i_y$ at $y=x^c$ implies that there exists a point
$w'\in p^{su}_{x_c}(\cP_c)$ with $d(w',z) = o(d(x_c,z))$.
Fix a point $w\in   (p^{su}_{x_c})^{-1}(w')\cap \cP_c$.
Tracing the $\hat\zeta^i$-curves in $\cP_c$ back from $w$ to $x_c$ 
produces the desired path $\eta$ from $x_c$ to $w$.  An inductive argument using
Lemma~\ref{l=Eccurves} shows that $d(w',w) = o(d(x_c,z)^r)$.
\end{proof}

Let us see how this implies that $j^\ell\sigma^c$ is differentiable along $\hW^c(x_c)$
at $x_c$.  This is essentially the same as the proof that a function with continuous partial
derivatives is $C^1$.  We will use:
\begin{lemma}\label{l=kindalip} For every $y\in V$ and every pair of points $z_1,z_2\in \hW^c(y)$:
 $$d(j^\ell_{z_1}\sigma^c ,j^\ell_{z_2}\sigma^c) = O(d(z_1,z_2) + d(z_1,y)^{r-\ell} + d(z_2,y)^{r-\ell}).$$ 
\end{lemma}
\begin{proof} 
This follows from the facts that $j^\ell\sigma^c$ is saturated and Lipschitz along $E^c$ curves,
and that $p^{su}_{y}$  has the properties given in Lemma~\ref{l=Eccurves}.
\end{proof}

Working in local charts on $\hW^c(x^c)$ sending $x^c$ to $0$, we may assume that
the curves $\hat\zeta^i_{x^c}$ are unit speed and correspond to the axes $\cap_{i\ne j}\{x^j=0\}$.  
Define constants
$a^i = a^i(x_c)\in P_0^\ell(c,n)$, for $i=1\ldots, c$  by
$$a_i = \lim_{y\to x^c}  (j^\ell \sigma^c\circ\hat\zeta^i_y)'(0).$$
We now define a linear map 
$A\colon \RR^c\to P_0^\ell(c,n)$ by
$$A(t_1,\ldots,t_c) = \sum_{i=1}^c a_i t_i.
$$
We claim that this map is the derivative of $j^\ell\sigma^c$ along $\hW^c(x_c)$ at $x_c$. Let
$z\in \hW^c(x_c)$ be given, and consider the path $\eta$
from $x_c$ to $w$ given by Lemma~\ref{l=goodpath}.  Let $v_1 = 0$, and write
$\eta = \hat\zeta^1_{v_1}\cdot \hat\zeta^2_{v_2} \cdots \hat\zeta^c_{v_c}$;
for $i=1,\ldots c-1$, let $t_i$ satisfy $\hat\zeta^i_{v_i}(t_i) = v_{i+1} = \hat\zeta^{i+1}_{v_{i+1}}(0)$,
and let $t_c$ satisfy $\hat\zeta^c_{v^c}(t_c)=w$.
The length of the curve $\eta$ is $\Theta(\sum_{i=1}^c|t_i|) = \Theta(d(x_c,z))$.  
Lemma~\ref{l=cspinch} readily implies that the distance between the $\ell$-jets of $\hW^c(w)$ at $w$
and $\hW^c(p^{su}_{x_c}(w))$ at $p^{su}_{x_c}(w)$ is $o(\hbox{length}(\eta)^{r-\ell})
= o(d(x_c,z)^{r-\ell})$.
Since $j^\ell\sigma^c$ is bisaturated and Lipschitz, we obtain from Lemma~\ref{l=spinch} that
\begin{eqnarray*}
d(j^\ell_{w}\sigma^c, j^\ell_{p^{su}_{x_c}(w)}\sigma^c) &=&  O(d(w,p^{su}_{x_c}(w))) + o(d(x_c,z))^{r-\ell})\\
&=&  O(d(x_c,z)^r) + o(d(x_c,z))^{r-\ell})\\
&=& o(d(x_c,z)),
\end{eqnarray*}
where we have used the facts that $d(w, p^{su}_{x_c}(w)) = o(d(z,x_c)^{r})$
and $\ell\leq r-1$. Also,  since $d(z,p^{su}(w)) = o(d(z,x_c))$, Lemma~\ref{l=kindalip} implies that
$$d(j^\ell_{z}\sigma^c ,j^\ell_{p^{su}_{x_c}(w)}\sigma^c) = o(d(z,x_c)),$$
and so 
$$d(j^\ell_{z}\sigma^c ,j^\ell_{w}\sigma^c) = o(d(z,x_c)).$$

Using the fact that
$j^\ell\sigma^c$ has a directional derivative along each $\hat\zeta^i$ subpath of $\eta$ at its anchor point $v_i = \hat\zeta^i_{v_i}(0)$,
and writing things in local coordinates sending $x^c$ to $0$,
we obtain that:
\begin{eqnarray*}j^\ell_z\sigma^c - j^\ell_0\sigma^c &=& \sum_{i=1}^c (j^\ell_{\hat\zeta^i_i(t_i)}\sigma^c - j^\ell_{\hat\zeta^i_i(0)}\sigma^c) + (j^\ell_{z}\sigma^c  - j^\ell_{w}\sigma^c)\\
&=& \sum_{i=1}^c (j^\ell \sigma^c\circ\hat\zeta^i_i)'(0)\cdot t_i  + o(|z|) \\
&=& A(z) + o(|z|).
\end{eqnarray*}
Hence $j^\ell\sigma^c$ is differentiable along $\hW^c(x_c)$ at $x_c$, with derivative $A$.

Now we have that $j^\ell\sigma^c$ is differentiable at $x_c$ along $\hW^c(x_c)$,
we can spread this derivative around using $\hcH^\ell$, and we get that 
the derivative of $j^\ell\sigma^c$ along $\hW^c(x)$ at $x$ exists for every
$x$ and is a continuous function on $M$.
We still need to show that $\sigma$ has central $\ell+1$ jets, with uniform error term.

The derivative of $j^\ell\sigma^c$ at $x$ gives a candidate $j^{\ell+1}_x\sigma^c$
for a central $\ell+1$ jet at $x$;
the $\ell+1$st coordinate in $j^{\ell+1}_x\sigma^c$ is just the derivative at $x$ along
$\hW^c_x$ of the $\ell$th coordinate of $j^\ell\sigma^c$. To show that
$\sigma$ has a central $\ell+1$-jet at $x$, we must show that for every $v\in B_{\widetilde E^c(x)}(0,\rho)$:
\begin{eqnarray}\label{e=jetform0} d_{N}(\hbox{proj}_N \circ \tilde\sigma\circ g^c(v), \hbox{proj}_N\circ j^{\ell+1}_x\sigma^c(v)) = o(|v|^{\ell+1}).
\end{eqnarray}

We first note that $j^\ell\sigma^c$ is differentiable along $E^c$ curves.
To see this, let $\zeta$ be an $E^c$ curve in $M$. For each $t\in I$, Lemma~\ref{l=Eccurves}
implies there exists a $C^r$ curve $\hat\zeta_t$ in $\hW^c(\zeta(t))$ with $\hat\zeta_t(0) = \zeta(t)$
and such that $\hat\zeta_t$ and $\zeta(s+t)$ are tangent to order $r$ at $0$. Furthermore,
the previous arguments using saturation of $j^\ell\sigma$ 
show that the distance between $j^\ell_{\zeta(s+t)}\sigma^c$ and $j^\ell_{\hat\zeta_t(s)}\sigma^c$
is $o(|s|^{r-\ell})$.  Since $j^\ell\sigma^c$ is differentiable along $\hat\zeta_t$ at $s=0$,
this implies that  $j^\ell\sigma^c$ is differentiable along $\zeta(s+t)$ at $s=0$. Since
$t$ was arbitrary, we see that $j^\ell\sigma^c$ is differentiable, and in fact $C^1$,
along $\zeta$.

Our induction hypothesis implies that $\sigma$ is $C^\ell$ along $E^c$ curves.
We next observe that, for any $E^c$ curve $\zeta$, the $\ell$-jet of $\sigma\circ \zeta$
at $t\in I$ satisfies: 
\begin{eqnarray}\label{e=jetform}\hbox{proj}_N\circ j^\ell_t(\sigma\circ\zeta) =  \hbox{\proj}_N \circ j^\ell_{\zeta(t)}\sigma^c\circ j^\ell_{\zeta(t)}(\pi^c\circ\exp_{\zeta(t)}^{-1}) \circ j^\ell_t\zeta.
\end{eqnarray} 
To see this, let $\hat\zeta_t$ be given by Lemma~\ref{l=Eccurves}.  Since $\zeta(t+s)$ and
$\hat\zeta_t(s)$ have the same $\lfloor r\rfloor$ jets at $s=0$, and $\sigma$ is Lipschitz,
the functions $\sigma\circ\zeta_t(s)$ and $\sigma\circ\zeta(s+t)$ have the same $\ell$-jets
at $s=0$.  But the definition of central $\ell$-jets implies that:
$$d_{N}(\hbox{proj}_N\circ \sigma\circ \hat\zeta_t(s), \,\hbox{proj}_N\circ j^\ell_{\hat\zeta_t(0)}\sigma^c\,\circ\, \pi^c\,\circ\,\exp_{\hat\zeta_t(0)}^{-1}\circ\,\hat\zeta_t(s)) = o(|s|^\ell);
$$
from the naturality of jets under composition, (\ref{e=jetform}) follows immediately.

Now, since both $j^\ell\sigma^c$  and $j^\ell(\pi^c\,\circ\,\exp^{-1})$ are differentiable
along $E^c$ curves, it follows that $\sigma$ is $C^{\ell+1}$ along every $E^c$ curve $\zeta$,
and by Taylor's theorem, the $\ell+1$ jets of $\sigma\circ\zeta$ are given by the formula
\begin{eqnarray}\label{e=jetform2}j^{\ell+1}_t(\sigma\circ\zeta) = j^{\ell+1}_{\zeta(t)}\sigma^c\,\circ\, j^{\ell+1}_{\zeta(t)}(\pi^c\,\circ\,\exp_{\zeta(t)}^{-1})\, \circ \,j^{\ell+1}_t\zeta.
\end{eqnarray}

Finally, let $v\in B_{\widetilde E^c(x)}(0,\rho)$ be given, and let $y=\exp_x g^c(v) \in \hW^c(x)$.
Fix a geodesic arc $\hat\zeta$ in  $\hW^c(x)$ from $x$ to $y$, with $\hat\zeta(0)=x$
and $\hat\zeta(1)=y$.  Let $\zeta$ be the $E^c$ curve
given by Lemma~\ref{l=Eccurves}, tangent to order $r$ to $\hat\zeta$ at $\hat\zeta(0)=x$.
Equation (\ref{e=jetform2}) now implies that 
$$d_{N}(\hbox{proj}_N\circ \sigma\circ \hat\zeta(t),\,\hbox{proj}_N \circ j^{\ell+1}_x\sigma^c(tv)) = o(|tv|^{\ell+1}).
$$
Since $d(\hat\zeta(t),\zeta(t)) = o(|tv|^r)$, and $\sigma$ is Lipschitz, we obtain (\ref{e=jetform0}).
Hence $\sigma$ has a central $\ell+1$ jet at $x$, and it is given by $j^{\ell+1}_x\sigma^c$.
We have verified both I$_{\ell+1}$ and III$_{\ell+1}$. 
\bigskip

\begin{proposition} $\sigma$ is $C^r$.
\end{proposition}

\begin{proof}
If $r=1$, then we have already shown that the $0$-jet of $\sigma$ is differentiable along
$\cW^c(x)$ at $x$, for every $x$, and this derivative varies continuously at $M$.
Since $\sigma$ is $C^1$ along the leaves of $\W^s$ and $\W^u$, this readily
implies that $\sigma$ is $C^1$.

Assume, then that $1<r<k-1$. 
Let $\overline\ell = \lfloor r \rfloor$, and let $\overline\alpha = r-\overline\ell$.
We first show:

\bigskip

\noindent{\bf $j^{\overline\ell}\sigma^c$ is $C^{\overline\alpha}$ at $x$ along $\hW^c(x)$, for every $x\in M$.}
The proof is a slight adaptation of the proof that $j^{\overline\ell}\sigma^c$
is Lipschitz at $x$ along $\hW^c(x)$, for every $x\in M$, for $\ell<r$;
the central observation that allows one to modify this proof is that
$H_\cS^{\overline\ell}(x,\wp)$  still covers the diffeomorphism $H_\cS(x,\wp)$, 
and for $i\geq 1$,
$H_\cS^{\overline\ell}(x,\wp)_i$ is $\overline\alpha$-H{\"o}lder continuous in the $(x,\wp_0)$-variable,
and $C^\infty$ in the $(\wp_1,\cdots,\wp_{\overline\ell})$-variables.  (See the proof
of part II of Theorem~\ref{t=main} as well).  We omit the details.

\bigskip

\noindent{\bf $\sigma$ has an $(\overline\ell, \overline\alpha, C)$ expansion at $x$ along $\hW^c(x)$, uniformly in $x\in M$.}  This is essentially the same as the proof that
$\sigma$ has a central $\ell$-jet at every point for $\ell<r$, except one sharpens the
estimates on the remainder of the Taylor expansions along $E^c$ curves,
using the $\overline\alpha$-H{\"o}lder continuity of the central $\overline\ell$-jets.

\bigskip

\noindent{\bf The section $\sigma$ is $C^r$.}
Since $r$-bunching is an open condition, as is the condition $r<k-1$, by increasing $r$ slightly,
we may assume that $r$ is not an integer.

We have shown that $\sigma$ has central $\overline\ell$-jets, and that $j^{\overline{\ell}}\sigma^c$ is 
$\overline\alpha$-H{\"o}lder continuous.
Fix a point $p\in M$. 
The fake center-stable manifolds $\hW^{cs}(x)$, for $x$ in a neighborhood $U$ of $p$, form a 
continuous family of $C^{r} = C^{\overline\ell,\overline\alpha}$ embedded disks.

Fix $x$ in this neighborhood $U$, and 
consider the foliation $\{\hW^s_x(y)\}_{y\in \hW^{cs}(x)}$ of the plaque $\hW^{cs}(x)$
by fake stable manifolds.  Since $\sigma$ is $\cW^s$ saturated, it is $C^{k}$ along $\cW^s(y)$, for any
$y\in M$.  In particular, it has a $(\overline\ell,\overline\alpha,C)$-expansion along $\cW^s(y)$, for
any $y$. For $y\in\hW^{c}(x)$ corresponding to $(0,0,x^c)$ in adapted coordinates at $x$,  Lemma~\ref{l=cspinch}
implies that the distance between $\hat\omega^{cs}_{(0,0,x^c)}(0,x^s)$ and $\hat\omega^{cs}_{0}(x^c,x^s)$
is $o(d(x,y)^r)$.  Since $\sigma$ is Lipschitz, and $\sigma$  has a $(\overline\ell,\overline\alpha,C)$-expansion along  $\hat\omega^{cs}_{(0,0,x^c)}(0,x^s)$ (which corresponds to $\W^s(y)$), this implies that
$\sigma$ has a $(\overline\ell,\overline\alpha,C)$-expansion along $\hW^s(y)$ (corresponding to
$\hat\omega^{cs}_{0}(x^c,x^s)$) with an error term that is 
on the order of $d(x,y)^r$.  

Next consider the family of plaques $\{\widetilde \W^c(y)\}_{y\in\hW^{cs}(x)}$
defined by $\widetilde \W^c(y) = \hW^{cs}(x)\cap \hW^{cu}(y)$. This forms a continuous
family of $C^r$-embedded disks. 
Paired with the the $\hW^s_x$ foliation, the family of $\widetilde \W^c$ plaques gives a $C^r$ transverse  pair of
plaque families in $\hW^{cs}(x)$. 
Lemma~\ref{l=cspinch} implies that for 
each $y\in \hW^{cs}(x)$, the distance between the
$\overline\ell$-jets of $\hW^{cs}(x)$ at $x$ and $\hW^{cs}(y)$ at $y$ is $o(d(x,y)^{\overline\alpha})$.  Since 
$\hW^c(y) = \hW^{cs}(y)\cap \hW^{cu}(y)$, it follows that the the distance between the
$\overline\ell$-jets at $y$ of 
$\widetilde \W^c(y)$ and $\hW^c(y)$ is also $o(d(x,y)^{\overline\alpha})$.  But $\sigma$ is Lipschitz,
and $\sigma$ has an $(\overline\ell,\overline\alpha,C)$ expansion at $y$ along $\hW^c(y)$, for every $y$. 
This implies that
in an adapted coordinate system at $x$ , we can write the plaques $\widetilde W^c(y)$
as a parametrized family along which $\sigma$ 
has an $(\overline\ell,\overline\alpha,C)$ expansion at $y$ along $\widetilde W^c(y)$, for every $y\in \hW^{cs}(x)$, with an error term that is on the order of $d(x,y)^r$.
Hence we can apply Theorem~\ref{t=journe} to conclude that
$\sigma$ has an $(\overline\ell,\overline\alpha,C)$-expansion  along $\hW^{cs}(x)$ at $x$,
for every $x$ in $U$, where $C$ is uniform in $x$.

Now the family $\{\hW^{cs}(x)\}_{x\in U}$ is a uniformly continuous family of $C^r$ plaques in $U$. 
Paired with the local $\cW^u$ foliation, it gives a transverse  $C^{\overline\ell,\overline\alpha}$ pair of
plaque families in $U$. Since $\sigma$ is $u$-saturated, it is $C^k$ along
$\W^u$-leaves and in particular has an $(\overline\ell,\overline\alpha,C)$-expansion  
along $\cW^u(x)$ at every $x\in U$.
Applying Journ\'e's theorem again, we obtain that 
$\sigma$ has a $(\overline\ell,\overline\alpha,C')$-expansion expansion at every $x\in U$,
where $C'$ is uniform in $x\in U$.  Theorem~\ref{t=campanato}
implies that $\sigma$ is $C^{r}$ in $U$.  As $p$ was arbitrary, we obtain that  $\sigma$ is $C^{r}$.
\end{proof}

This completes the proof of Theorem~\ref{t=C^rsec}. \end{proof}

\section{Final remarks and further questions}

The proofs here could admit several improvements and generalizations. Some are not difficult:
for example, the compactness of the manifold $M$ was not essential.  The definition
of partial hyperbolicity in the noncompact cases merely needs to be modified to
ensure that the functions $\nu,\hat\nu, \nu/\gamma, \hat\nu/\hat\gamma$ are uniformly bounded away from $1$,
and the definition of $r$-bunching must be similarly adjusted.  Other improvements on
Theorem~\ref{t=main} are more challenging.  For example, there is no counterpart in Theorem~\ref{t=main}
to the analyticity conclusions in Theorem~\ref{t=livsicold}, part IV.
Another question is whether the H{\"o}lder exponent in  Theorem~\ref{t=main}, part II can be improved.
Finally, we ask whether the loss of one derivative in Theorem~\ref{t=main} part IV (and Theorem C)
is really necessary: is it true that if $\phi$ is $C^r$, $f$ is $C^r$, accessible and $r$-bunched, where $r\geq 1$,
then any continuous solution to (\ref{e=livsic2}) is $C^r$ (or perhaps $C^{r-\eps}$, for all $\eps>0$)?

\section*{Acknowledgments}

It is a pleasure to thank Danijiela Damjanovi\'c for initial conversations that led to this paper,
Artur Avila and Marcelo Viana for discussions about \cite{ASV}, Shmuel Weinberger for telling me about \cite{rss}, Danny Calegari, Benson Farb, and Charles Pugh for useful conversations, and Dusan Repov\v s  and Arkadiy Skopenkov  
for helpful communications. Thanks also to Rafael de la Llave and Andrew T{\"o}r{\"o}k for comments on an early version
of this manuscript and pointing out some important references.  The author was supported by an NSF grant.

\end{document}